\documentclass[10pt,DIV=11,twoside=semi]{scrartcl}
\usepackage{scrpage2}

\addtokomafont{title}{\normalfont}
\addtokomafont{section}{\normalfont \bf}
\addtokomafont{subsection}{\normalfont \bf}
\addtokomafont{subsubsection}{\normalfont \bf}
\addtokomafont{paragraph}{\normalfont \bf}

\usepackage{amsmath, amssymb, amsthm}
\usepackage{graphicx}
\usepackage[subrefformat=parens,labelformat=parens]{subfig}

\usepackage[T1]{fontenc}
\usepackage{lmodern}

\numberwithin{equation}{section}

\usepackage{hyperref}

\hypersetup{
  pdftitle    = {Multilevel Preconditioning of Discontinuous-Galerkin Spectral Element Methods, Part I: Geometrically Conforming Meshes},
  pdfauthor   = {Kolja Brix, Martin Campos Pinto, Claudio Canuto and Wolfgang Dahmen},
  pdfkeywords = {Discontinuous Galerkin discretization for elliptic problems, interior penalty method, variable polynomial degrees, auxiliary space method, Legendre-Gauss-Lobatto grids, associated dyadic grids},
  pdfborder={0 0 0.5}
}

\DeclareMathOperator{\vol}{vol}

\newcommand{\lsim}{\lesssim}

\newcommand{\oI}{\overline{I}}

\newcommand{\cD}{{\cal D}}
\newcommand{\cF}{{\cal F}}
\newcommand{\cG}{{\cal G}}

\newcommand{\cI}{{\cal I}}

\newcommand{\cO}{{\cal O}}

\newcommand{\cR}{{\cal R}}

\newcommand{\cT}{{\cal T}}

\newcommand{\cW}{{\cal W}}

\newcommand{\bp}{{\bf p}}
\newcommand{\bq}{{\bf q}}
\newcommand{\bell}{{\boldsymbol{\ell}}}
\newcommand{\bH}{{\bf H}}
\newcommand{\delt}{1}

\newcommand{\R}{\mathbb{R}}
\newcommand{\N}{\mathbb{N}}
\newcommand{\Z}{\mathbb{Z}}

\newcommand{\newsharp}{\downarrow}

\newcommand{\id}{{\rm id}}
\newcommand{\spn}{{\rm span}\,}

\def\endproof{\hfill\vbox{\hrule height0.6pt\hbox{%
   \vrule height1.3ex width0.6pt\hskip0.8ex
   \vrule width0.6pt}\hrule height0.6pt
  }\bigskip  }

\newcommand{\BIGOP}[1]{\mathop{\mathchoice%
{\raise-0.22em\hbox{\huge $#1$}}%
{\raise-0.05em\hbox{\Large $#1$}}{\hbox{\large $#1$}}{#1}}}
\newcommand{\BIGboxplus}{\mathop{\mathchoice%
{\raise-0.35em\hbox{\huge $\boxplus$}}%
{\raise-0.15em\hbox{\Large $\boxplus$}}{\hbox{\large $\boxplus$}}{\boxplus}}}

\newcommand{\bigtimes}{\BIGOP{\times}}

\newcommand{\abs}[1]{\lvert#1\rvert}
\newcommand{\norm}[1]{\ensuremath{\left|\!\left|#1\right|\!\right|}}

\newtheorem{definition}{Definition}[section]

\newtheorem{remark}[definition]{Remark}

\newtheorem{theorem}[definition]{Theorem}
\newtheorem{property}[definition]{Property}
\newtheorem{proposition}[definition]{Proposition}
\newtheorem{lemma}[definition]{Lemma}
\newtheorem{corollary}[definition]{Corollary}

\title{\vspace{-4mm}\LARGE Multilevel Preconditioning of\\ Discontinuous-Galerkin Spectral Element Methods \\ Part I: Geometrically Conforming Meshes}

\author{\large%
Kolja Brix\footnotemark[1], \
Martin Campos Pinto\footnotemark[4]\ \footnotemark[5], \
Claudio Canuto\footnotemark[6] \ and
Wolfgang Dahmen\footnotemark[1]}

\date{\large \today%
\vspace{-10mm}}

\clearscrheadfoot
\lehead{\pagemark \qquad \textsc{\small K. Brix. M. Campos Pinto, C. Canuto and W. Dahmen}}
\rohead{\textsc{\small Multilevel Preconditioning of DG-SEM, Part I: Geometrically Conforming Meshes} \qquad \pagemark}

\begin{document}

\maketitle

\renewcommand{\thefootnote}{\fnsymbol{footnote}}
\footnotetext[1]{Institut f\"ur Geometrie und Praktische Mathematik, RWTH Aachen, Templergraben 55, 52056 Aachen, Germany, e--mail: \texttt{\{brix,dahmen\}@igpm.rwth-aachen.de}}
\footnotetext[4]{CNRS, UMR 7598, Laboratoire Jacques-Louis Lions, 4 Place Jussieu, 75005 Paris, France}
\footnotetext[5]{UPMC Univ Paris 06, UMR 7598, Laboratoire Jacques-Louis Lions, 4 Place Jussieu, 75005 Paris, France, e--mail: \texttt{campos@ann.jussieu.fr}}
\footnotetext[6]{Dipartimento di Scienze Matematiche, Politecnico di Torino, Corso Duca degli Abruzzi 24, 10129 Torino, Italy, e--mail: \texttt{claudio.canuto@polito.it}.}
\renewcommand{\thefootnote}{\arabic{footnote}}


\thanks{\small
\paragraph{Abstract}
This paper is concerned with the design, analysis and implementation of preconditioning concepts
for spectral Discontinuous Galerkin discretizations of elliptic boundary value problems.
While presently known techniques realize a growth of the condition numbers that is logarithmic in the polynomial degrees when
all degrees are equal and quadratic otherwise,
our main objective is to realize full robustness with respect to arbitrarily large locally varying polynomial degrees degrees,
i.e., under mild grading constraints condition numbers stay uniformly bounded with respect to the mesh size and variable degrees.
The conceptual foundation of the envisaged preconditioners is the auxiliary space method.
The main {conceptual ingredients} that will be shown in this framework to yield ``optimal'' preconditioners in the above sense
are Legendre-Gauss-Lobatto grids in connection with certain associated anisotropic nested dyadic grids as well as specially adapted wavelet preconditioners for the resulting low order auxiliary problems.
Moreover, the preconditioners have a modular form that facilitates somewhat simplified partial realizations. One of the
components can, for instance, be conveniently
combined with domain decomposition,
at the expense though of a {logarithmic growth of condition numbers}.
Our analysis is complemented by quantitative experimental studies of the main components.


}


\renewcommand{\thefootnote}{}

\footnotetext{This work was supported in part by the European Commission Improving Human Potential Programme within project ``Breaking Complexity'' (HPRN-CT-2002-00286), by the RWTH Aachen Seed Funds and Distinguished Professorship projects, as well as by the Graduate School AICES funded by the Excellence Initiative of the German federal and state governments, and by DFG project ``Optimal preconditioners of spectral Discontinuous Galerkin methods for elliptic boundary value problems'' (DA 117/23-1).}

\renewcommand{\thefootnote}{\arabic{footnote}}


\section{Introduction} \label{sec:intro}

Attractive features of {Discontinuous} Galerkin (DG) discretizations are on the one hand their versatility regarding
a variety of different problem types,
and on the other hand their flexibility regarding local mesh refinement
and even {locally variable polynomial order} of the discretization.
While initially the main focus has been on transport problems like hyperbolic conservation laws,
an increased attention has recently been paid to diffusion problems,
which naturally enter the picture in more complex applications like incompressible Navier-Stokes equations.
Due to the possible occurrence of both singularities
and regions of high regularity on the other hand,
the use of variable and possibly arbitrarily high degrees is particularly attractive, see, e.g., \cite{SSH,WFS};
see also \cite{CaHuQuZa07}.
The central theme of this paper is to develop efficient solvers for the systems of equations resulting from
{\em Discontinuous Galerkin Spectral Element} (DG-SE) discretizations.
By this we mean that
nodal-based polynomials of {\em arbitrarily high locally variable} polynomial degree
on meshes with arbitrarily small mesh size are permitted.
In this work we confine the discussion to {\em geometrically conforming} meshes, i.e., the intersection of any
two mesh elements is empty or a common facet so that hanging nodes are not permitted.

To formulate our objectives in more specific terms and also to indicate some intrinsic obstructions we briefly recall the state of the art
regarding preconditioners for the DG method applied to {\em second order elliptic boundary value problems}.

The first group of results refers to {\em uniformly bounded} polynomial degrees.
The multigrid scheme proposed in \cite{kanschat}
gives rise to uniformly bounded
condition numbers
provided that
(i) the underlying hierarchy of meshes is \emph{quasi-uniform} and
(ii) the solution exhibits a certain (weak) regularity.
This scheme has been extended in \cite{Kanschat2004} to locally refined meshes showing a similar performance
without a theoretical underpinning though. Domain decomposition preconditioners investigated in \cite{AA1,AA2}
give rise to only logarithmically growing condition numbers when the mesh size decreases. A two-level scheme in the sense
of the {\emph{auxiliary space method}} (see, e.g., \cite{Brenner1996a,Oswald1996,Xu1996}) is proposed
in \cite{DLVZ2006} and shown to exhibit mesh-independent convergence again
on quasi-uniform {geometrically} conforming meshes with a \emph{fixed} uniformly equal polynomial degree.
In the framework of the auxiliary space method,
preconditioners providing uniformly bounded condition numbers for locally refined meshes have been developed in \cite{BCD,BCDM}
under weak grading constraints and for variable but uniformly bounded polynomial degrees.
In all those results the condition numbers {\em depend on the bound} for the polynomial degrees.

The following second group of results {\em concerns the quantitative dependence} of condition numbers on the polynomial degree
aiming at the use of polynomial elements of {\em arbitrarily high order}.
These results draw primarily on domain decomposition
concepts, see, e.g., \cite{TW2005}. More precisely, two essentially distinct cases arise, namely

(a) all polynomial degrees are equal,

(b) the polynomial degrees may vary from element to element.\\[1mm]
In the case (a) the condition numbers can be shown to exhibit only a logarithmic growth in the {polynomial} degree $p$ (see \cite{TW2005}).
This may be perceived as quite satisfactory for practical purposes if one accepts unnecessarily large {polynomial} degrees even near singularities. In the case (b), however,
when arbitrarily high polynomial degrees are to be used only in part of the domain, the best known bounds to us grow like $p^2$ (see \cite{SSH}) which does call for improvements.

In summary, (1) none of the currently known preconditioners gives rise to uniformly bounded condition numbers {\em independent of the polynomial degrees}, (2)
a {\em strong} growth of condition numbers may occur
when non-uniform degree distributions are used, i.e., when
the quotient of the largest and lowest degree is allowed to be unbounded.

To see why there is an essential difference between (1) and (2) it is instructive to consider first the extreme case of a
{\em spectral} trial space spanned by a high-order polynomial on a single quadrilateral element.
All strategies for this case known to us can
be interpreted as employing
a low-order \emph{auxiliary} space to precondition the system for the high-order discretization. To distinguish
the two types of spaces we sometimes refer to the original elements in the high-order finite element mesh, comprised in the extreme
case under consideration of a single element, as
{\em macro-elements} while the grid inside each macro-element can be viewed as a {\em subgrid}.
In \cite{BeuSchneiSchwab}, at least for
the Laplace operator, piecewise linear finite elements on a Cartesian equidistant subgrid in conjunction with wavelet bases for
weighted spaces give rise to
precondition numbers with logarithmic growth in the polynomial degree. An alternative, perhaps more versatile (with respect to problem specification) approach put forward in \cite{DM1985} and theoretically supported by \cite{Ca94,PaRo95}, is to our knowledge the only way to
obtain uniformly condition numbers.
It uses a very special low-order space based on so called \emph{Legendre-Gauss-Lobatto} (LGL) grids associated with the high-order space.
It is known that using the inverse of the low-order discretization on the LGL {\em subgrid} as a preconditioner for the high-order trial space
on the single element
gives rise to \emph{uniformly bounded} condition numbers \cite{CaHuQuZa06}.

When the finite element mesh is comprised of more than a single element it is shown in \cite{TW2005} how to still use this concept in conjunction
with {\em domain decomposition} in a near optimal way, provided that the LGL subgrids on adjacent elements
match at element interfaces, which means that all polynomial degrees are {\em equal}. The growth of the condition numbers can then be kept logarithmic while, however, such bounds no longer exist when the polynomial degrees vary locally in the above strong sense.
In fact, apparently
the {\em non-nestedness} of LGL grids is then the essential obstruction to contriving {\em optimal} preconditioners in the sense that:

(P1) the condition numbers remain {\em uniformly bounded} independent of the mesh size and locally variable arbitrarily high degrees;

(P2) the preconditioner can be applied at a computational cost that stays {\em proportional} to the problem size.\\[1mm]
Here are a few indications why non-nestedness causes serious problems when the degrees vary locally.
First, the jump terms
at element interfaces corresponding to the high-order trial functions
are {\em not} equivalent to those for the corresponding auxiliary low-order spaces which are conforming only inside
each macro-element. In fact, in the high-order case a jump between
two adjacent polynomial elements of formally different polynomial degrees may still vanish due to matching traces, while the jumps for corresponding
low-order finite element traces would not vanish because the nodes of the LGL subgrids interlace at the marco-element interfaces.
Second, the non-nestedness of LGL grids implies that, when the polynomial degrees on adjacent macro-elements disagree,
the corresponding subgrids have only a trivial intersection at such an interface.
Therefore, one fails to find sufficiently rich globally conforming auxiliary low-order spaces.
As a consequence, one faces serious difficulties in verifying the relevant auxiliary space conditions.
Third, the non-nestedness of LGL grids not only affects (P1) but also the application complexity (P2).
In fact, when employing locally large polynomial degrees,
the iterative solution of an auxiliary problem, even on only a single macro-element,
becomes problematic because one cannot resort to efficient multilevel techniques and hence (P2) is not clear.

In summary, it seems that the currently know concepts are not sufficient to provide optimal preconditioners for
DG-SE discretizations. The central objective of this paper is to construct optimal preconditioners for DG-SE discretizations in the sense of (P1), (P2).
Specifically, to overcome the obstructions outlined above, we introduce the following new conceptual ingredients.
The first one is the construction of certain
\emph{dyadic} grids that are associated in a strong sense with the LGL grids but are in addition \emph{nested}. This association
manifests itself through a number of stability estimates based on suitable comparison criteria for different grids.
These findings allow us eventually
to ensure (P1). Concerning (P2), the auxiliary grid hierarchies allow one, in principle, to employ multilevel techniques to
solve the resulting low order problems.
However, the strong anisotropies of the dyadic meshes, which are inherited from the LGL grids, appear to
prevent standard techniques like BPX-preconditioners from working well. Therefore, as a second ingredient, we propose
a specially tailored
\emph{wavelet preconditioner}
using suitable piecewise polynomial $L_2$-orthogonal multi-wavelets.

As mentioned above, we choose the \emph{auxiliary space method} as a conceptual platform for putting these tools to work, see \cite{Oswald1996,Xu1992,Xu1996}.
The key issue is to construct a
conforming auxiliary space comprised of globally continuous piecewise
multi-linear functions on dyadic subgrids
of the original macro mesh. It turns out though that the identification of suitable ingredients and their analysis is facilitated best by
realizing the ``final'' auxiliary space through several {\em stages}. That is the auxiliary space method is {\em iterated}.
This offers in our opinion at least two major benefits. First, it turns out that the identification of a ``proper smoother'' is not obvious
in the ``one-step-mode'' while it is naturally obtained as a result of concatenating the intermediate stages. Second,
the intermediate stages can be used as ``stand-alone'' results in several contexts. For instance, in the first stage
we use a high-order conforming subspace as an auxiliary space (to get rid of the above mentioned ``jump-problem'' arising from non-matching LGL subgrids).
The corresponding result can be used directly as an essential tool for a domain decomposition
preconditioner for the DG-SE discretization, see \cite{CPP}, accepting a logarithmic growth of condition numbers, but now for
arbitrary locally variable polynomial degrees. Also, associated dyadic grid concepts as well as the associated wavelet preconditioner
can be used for high-order {\em conforming} Galerkin discretizations or for a pure spectral discretization on a single element
offering a performance that does not seem to be available yet for these scenarios either.

The paper is organized as follows. In Section~\ref{sect2.1} we formulate a simple model problem
and describe the main ingredients of the DG-SE discretization. In Section~\ref{sec:asm}, following essentially \cite{Oswald1996},
we briefly recall the \emph{auxiliary space} concept in a way that is most conveniently applied
in the present context. In Section~\ref{sec:sdgtoscg} we explain first, as indicated above, why we essentially split the construction of a suitable auxiliary space into several stages, and formulate a result for the first ``intermediate'' stage, see Theorem~\ref{th:stageI}. Section~\ref{sec:conf} is devoted to the second and main stage
of constructing globally conforming low order finite element spaces on certain {\em strongly associated} dyadic grids.
As stated in Theorem~\ref{th:2}, even uniformly bounded condition numbers
(avoiding the logarithmic growth of the domain decomposition approach) can be obtained, once a corresponding
preconditioner for the conforming low-order problem is available. Due to the fact that the partitions for the low-order auxiliary spaces
involve highly anisotropic cells, this is not completely obvious and standard BPX-type techniques do not work well enough.
In Section~\ref{sec:multiwavelet} we develop, so to speak as a third stage, a \emph{change-of-bases} preconditioner based on specially tailored
multi-wavelets that does give rise to uniformly bounded condition numbers, see Theorem~\ref{thm:tildeA2-prec}.
In Section~\ref{sec:comp} we present the resulting ``composite preconditioner'', see Theorem~\ref{thm:combined}.
Each stage is concluded by some numerical experiments
quantifying its performance.

Sections~\ref{sec:proof-I} and~\ref{sec:proof-II} are then devoted to the proofs of Theorems~\ref{th:stageI} and~\ref{th:2}. These are necessarily
rather technical but we have tried to organize them in a way that brings out the main mechanisms.
From a bird's view one could say that the robust treatment of varying polynomial degrees hinges on two main ingredients, namely
first the fact that certain interpolation operators provide uniformly stable $L_2$- and $H^1$-isomorphisms between
high-order polynomial spaces and low-order finite element spaces on LGL partitions, and second,
a proper notion of \emph{uniformly equivalent} grids that allows one to deal with different polynomial degrees and to switch
to the associated dyadic grids which, in turn, allows one to take advantage of nestedness.

Throughout the paper we shall employ the following notational convention. By $a \lsim b$
we mean that the quantity $a$ can be bounded by a constant multiple of $b$ uniformly in the parameters
$a$ and $b$ may depend on. Likewise $a\simeq b$ means $a\lsim b$ and $b\lsim a$.
For two vectors $\bp, \bq \in \R^n$ an inequality $\bq \leq \bp$ is to be understood {\em componentwise},
i.e., $q_k \leq p_k$ for $1 \le k \le n$.

\section{Model problem and discretizations}\label{sect2.1}

The methods developed in the sequel apply to second order {symmetric} elliptic boundary value problems with
variable coefficients.
To keep the technical level of the exposition as low as possible we confine the discussion to the model problem
\begin{equation}
\label{eq:setting.1}
-\Delta u = f \quad \mbox{in }\Omega \;, \qquad u=0 \quad \mbox{on }\partial\Omega \;,
\end{equation}
where $\Omega \subset \mathbb{R}^d$, is a bounded Lipschitz domain with piecewise
smooth boundary and $f \in L_2(\Omega)$. {Here one should think of $d\in \{1,2,3\}$.}
Such domains can be partitioned into images of (hyper-)rectangles
through smooth mappings (such as iso/sub-parametric mappings, or Gordon-Hall transforms).
Again, for the sake of technical simplicity, it suffices to treat unions of {closed} (hyper-)rectangles, i.e.,
we assume that the closure of $\Omega$ can be
partitioned into the union of an essentially disjoint finite collection $\cR$ of {closed} (hyper-)rectangles $R$.
{These (hyper-)rectangles are often simply referred to as \emph{elements}.}
Moreover, we confine the discussion in this paper to
{\em geometrically conforming partitions $\cR$}, which means that any
nonempty intersection between two elements $R$ and $R'$ in $\cR$ is an $l$-facet $F$ for both of them,
for some $0 \leq l \leq d-1$.
It will be convenient to introduce the complex
\begin{equation}
\label{facetcomplex}
\cF_l := \bigcup_{R\in \cR} \cF_l(R)
\end{equation}
of all $l$-dimensional {closed facets} associated with the macro-mesh consisting of the elements $R\in \cR$.
{We call a $(d-1)$-dimensional facet $F \in \cF_{d-1}(R)$ a face or an interface.}
Moreover, for each $F \in \cF_l$, we define $\cR(F):=\{R'\in \cR : R' \cap F\neq \emptyset\}$.

To describe the eligible discretizations, $\bH=(H_1,\ldots,H_d)=\bH(R)\in \R_+^d$ denotes the vector of the $k$-th side lengths of
$R$, $k=1,\ldots,d$. Likewise $\bp=(p_1,\ldots,p_d)=\bp(R)\in \N^d$ denotes the vector of coordinate-wise polynomial degrees of an element
of $\mathbb{Q}_\bp(R):= \bigotimes_{k=1}^d \mathbb{P}_{p_k}(I_k)$
It will be convenient to denote by
$\bH:= \sum_{R\in\cR}\bH(R)\chi_R$, $\bp = \sum_{R\in \cR}\bp(R)\chi_R$ the corresponding (global) meshsize- and
degree-functions and we use $\delta := (\bH,\bp)$ as the corresponding discretization parameter.

The trial spaces for the Discontinuous Galerkin Spectral Element (DG-SE) method are then of the form
\begin{equation}
\label{eq:setting.2}
V_\delta = \{v \in L_2(\Omega) \ : \ v_R:= v\mid_R \in \mathbb{Q}_\bp(R),\,\,\forall\,\, R\in \cR
 \}\;,
\end{equation}
while $V_\delta^c:= V_\delta \cap H^1_0(\Omega)$ denotes the largest conforming subspace of $V_\delta$ accommodating
the prescribed boundary conditions.
$\bH, \bp$ will always be assumed to satisfy the {\em grading conditions}
\begin{equation}
\label{eq:setting.50}
\frac{\max_k H_k(R)}{\min_k H_k(R)} \lsim 1, \qquad \frac{\max_k p_k(R)}{\min_k p_k(R)} \lsim 1 , \qquad
\max_\pm \max_k \frac{p_k(R^\pm)}{p_k(R^\mp)} \lsim 1,\qquad R\in \cR,
\end{equation}
uniformly in $\cR$, where $R^+, R^-$ stand for two adjacent elements sharing an interface.

We impose one further assumption on $\bp$ which is needed only when $d\geq 3$.
For $2 \leq l \leq d-1$ and for each $F \in \cF_l$, there exists
$R \in \cR(F)$ such that
\begin{equation}
\label{ass:locdeg}
\bp(F,R) \leq \bp(F,R') \qquad \forall R' \in \cR(F) \;,
\end{equation}
where $\bp(F,R):= \bp(R)\mid_{F}$.
Note that the property is trivially true for $l=1$ since $\bp(F,R) \in \mathbb{N}$, and for $l=d$ since $\cR(F)=\{F\}$.

Denoting by
${\bf n}_{R^\pm,F}=-{\bf n}_{R^\mp,F}$ the unit normal vectors on $F=R^+\cap R^-$ pointing to the
exterior of $R^\pm$, and
for $F \in \cF_{d-1}(R)$
we define as usual the jumps $[v]_F$ and averages $\{v\}_F$
as follows:
taking the homogeneous Dirichlet boundary condition into account,
for $F\subset \partial\Omega$ we set
$[v]_F = {\bf n}_{R,F} \, v\mid_{F}$
and
$\{v\}_F = v_{|F}$, while
for $F\subset \Omega$ we define
\begin{equation*}
[v]_F = {\bf n}_{R^-,F} \, v^{-}_{|F} + {\bf n}_{R^+,F} \, v^{+}_{|F} ,\qquad \{v\}_F = \frac12 \left(v^{-}_{|F} + v^{+}_{|F}\right) .
\end{equation*}

The {\em Symmetric Interior-Penalty Discontinuous Galerkin
Spectral-Element discretization} of Problem \eqref{eq:setting.1} is defined as follows (see \cite{Ar82,ArBrCoMa02}):
find $u \in V_\delta$ such that
\begin{equation}
\label{eq:conf.1}
{a}_\delta(u,v)=(f,v)_{0,\Omega}
\quad \forall v \in V_\delta \;,
\end{equation}
where the bilinear form $ {a}_\delta(\cdot,\cdot) : V_\delta \times V_\delta \to \R$ is given by
\begin{equation}
\label{eq:conf.2}
{a}_\delta(u,v)=\sum_{R \in \cR} (\nabla u,\nabla v)_{0,R} +
\sum_{F \in \cF_{d-1}} \big(- (\{\nabla u\},[v])_{0,F} - (\{\nabla v\},[u])_{0,F}
+\gamma \omega_F ([u],[v])_{0,F} \big) \;.
\end{equation}
As usual, we denote by $(\cdot,\cdot)_{0,G}$ the standard $L_2$-inner product over the domain $G$ and set $\|v\|_{0,G}:= (v,v)_{0,G}^{1/2}$.
The weights $\omega_F$ are defined as follows. When $F$ is orthogonal
to the $k$-th coordinate direction and $\cR(F)=\{R^\pm\}$,
we set
\begin{equation}
\label{eq:conf.3}
\omega_F= \max \left(\frac{(p_k(R^-)+1)^2}{H_k(R^-)},
\frac{(p_k(R^+)+1)^2}{H_k(R^+)}\right) \; ,
\end{equation}
with the obvious modification when $F\subset \partial\Omega$.
The definition is motivated by the inverse trace inequality
$\Vert {v} \Vert_{0,F} \leq \frac{p_k(R)+1}{\sqrt{H_k(R)}}\,
\Vert {v} \Vert_{0,R}$ for all $v \in \mathbb{Q}_\bp(R)$, $R=R^\pm$, which allows one to prove
the uniform coercivity and continuity of $a_\delta(\cdot,\cdot)$ provided the constant $\gamma>0$ is properly chosen.
Indeed, the following result holds (see, e.g.,\cite{SSH}).

\begin{proposition}
\label{prop:conf.1}
There exists a constant $\gamma_0>0$ such that for all $\gamma > \gamma_0$ the bilinear form
$a_\delta$, defined in \eqref{eq:conf.2}, satisfies
\begin{equation}\label{eq:conf.6}
a_\delta(v,v) \simeq \Vert v \Vert_{DG,\delta}^2 := \sum_{R \in \cR} \Vert \nabla v \Vert_{0,R}^2
+\gamma \sum_{F \in \cF_{d-1}} \omega_F \Vert \, [v] \, \Vert_{0,F}^2
\qquad
\forall v \in V_\delta \;.
\end{equation}
The constant $\gamma_0$ and the constants implied by the symbol $\simeq$ can be chosen independently
of $\delta =(\bH,\bp)$.
\end{proposition}

The central objective of this paper is to develop and analyze preconditioners for the linear systems arising from
\eqref{eq:conf.1}. Their concrete structure depends on the bases for the spaces $V_\delta$.
Nodal basis functions for certain specific subgrids on the elements $R$ will be seen to have particularly favorable properties,
among them uniformly equivalent but computationally more efficient quadrature formulations as discussed
in \cite{brix-thesis,BCD2013}.

\section{The auxiliary space method}\label{sec:asm}

The so called \emph{auxiliary space method} {(ASM)} will serve as the conceptual platform for developing preconditioners for the linear system
\eqref{eq:conf.1},
(see, e.g.,\cite{Brenner1996a,Xu1992,Xu1996,BCDM}). Specifically, we adopt
the abstract
framework, see \cite{Oswald1996,Nepomnyaschikh1990,Nepomnyaschikh1992} because the specific conditions
formulated there are best suited for the present application. In particular, the envisaged auxiliary spaces
are {\em not} contained in the DG trial spaces $V_\delta$.
For convenience of the reader we briefly recall the relevant ingredients in appropriate generality.

For a given finite-dimensional Hilbert space $V$ and a symmetric positive-definite bilinear form
$a \, : \, V \times V \to \mathbb{R}$ we seek
an auxiliary finite-dimensional Hilbert space $\tilde{V}$, endowed with a
symmetric positive-definite bilinear form $\tilde{a} \ : \ \tilde{V} \times \tilde{V} \to \mathbb{R}$ so that the respective variational problems
are spectrally equivalent. To ensure this when $\tilde V \not\subset V$ we consider the sum $\hat{V}=V+\tilde{V}$ and
 two further symmetric positive-definite bilinear forms $\hat{a}, \, {b}\ : \ \hat{V} \times \hat{V} \to \mathbb{R}$
 which satisfy the following conditions:\\[1mm]
 {\bf ASM1:}
$\hat{a}$ is a spectrally equivalent extension of both $a$ and $\tilde{a}$, i.e.,
\begin{equation}
\label{eq:asm_a_equiv_ahat}
a(v,v) \simeq \hat{a}(v,v), \quad \forall\, v \in V, \qquad
\tilde{a}(\tilde{v},\tilde{v}) \simeq \hat a(\tilde{v},\tilde{v}) \quad \forall\, \tilde{v} \in \tilde{V} .
\end{equation}
\noindent
{\bf ASM2:}
${b}$ dominates ${a}$ on ${V}$, i.e.,
$\,
{a}({v},{v}) \lsim {b}({v},{v})$, for all $ {v} \in {V}
$;\\[2mm]
{\bf ASM3:}
there exist linear operators $Q: \tilde{V} \to V$ and $\tilde{Q}: V \to \tilde{V}$ such that
\begin{equation}
 \label{eq:asm_bound_Qtilde}
\tilde{a}(\tilde{Q}v,\tilde{Q}v) \lsim {a}(v,v) \quad \forall\, v \in V, \quad
{a}(Q \tilde{v},Q \tilde{v}) \lsim \tilde{a}(\tilde{v},\tilde{v})
\quad \forall\, \tilde{v} \in \tilde{V} ,
\end{equation}
and
\begin{equation}
 \label{eq:asm_approx_Qtilde}
{b}(v - \tilde{Q}v,v - \tilde{Q}v) \lsim a(v,v) \quad \forall\, v \in V, \quad
{b}(\tilde{v} - Q \tilde{v},\tilde{v} - Q \tilde{v}) \lsim \tilde a(\tilde{v},\tilde{v}) \quad \forall\, \tilde{v} \in \tilde{V} .
\end{equation}

These conditions imply the following stable splitting.
\begin{proposition}[{\cite{Oswald1996}, Theorem 1 (2)}]\label{prop:asm1}
The conditions {\bf ASM1-3}
imply the following norm equivalence in $V$:
\begin{equation}\label{eq:asm4}
\hat{c} \, a(v,v) \ \leq \ \inf_{{\begin{array}{cc}{w \in V, \ \tilde{v} \in \tilde{V}} \\ {v=w+Q\tilde{v}} \end{array}}}
\left\{ {b}(w,w)+\tilde{a}(\tilde{v},\tilde{v}) \right\} \ \leq \ \hat{C} \, a(v,v) \qquad \forall\, v \in V \;,
\end{equation}
where the constants $\hat{c}$ and $\hat{C}$ depend only on the constants implied by the assumptions
(see
\cite{brix-thesis} for explicit expressions).
\end{proposition}

 Proposition~\ref{prop:asm1} has the following main consequence. Let $\mathbf{A}$, $\tilde{\mathbf{A}}$
and $\mathbf{B}$ denote the Gramian matrices for the bilinear forms $a$, $\tilde{a}$ and ${b}$ (restricted
to $V \times V$) with respect to suitable bases of the spaces $V$ and $\tilde{V}$.
Let $\mathbf{S}$ be the matrix representation of $Q$ with respect to these bases.
\begin{corollary}[see \cite{Oswald1996}, Theorem 2]
\label{cor:asm1}
Let $\mathbf{C}_\mathbf{B}$ and
$\mathbf{C}_{\tilde{\mathbf{A}}}$ be symmetric preconditioners for $\mathbf{B}$ and $\tilde{\mathbf{A}}$,
respectively, satisfying the following spectral bounds:
\begin{equation*}
 0 < \Lambda_{\min} \le
\min\{\lambda_{\min}(\mathbf{C}_\mathbf{B} \mathbf{B}), \ \lambda_{\min}(\mathbf{C}_{\tilde{\mathbf{A}}} \tilde{\mathbf{A}})\} \le
\max \{ \lambda_{\max}(\mathbf{C}_\mathbf{B} \mathbf{B}), \ \lambda_{\max}(\mathbf{C}_{\tilde{\mathbf{A}}} \tilde{\mathbf{A}})\} \le
\Lambda_{\max} \;.
\end{equation*}
{Moreover, let $\hat{c},\hat{C}>0$ be the constants in the norm equivalence \eqref{eq:asm4}.}
Then, under the assumptions {\bf ASM1-3},
$\mathbf{C}_\mathbf{A} = \mathbf{C}_\mathbf{B}+ \mathbf{S} \mathbf{C}_{\tilde{\mathbf{A}}}\mathbf{S}^T$
is a symmetric preconditioner for $\mathbf{A}$, and
\begin{equation*}
\kappa(\mathbf{C}_\mathbf{A} \mathbf{A}) \leq \frac{\Lambda_{\max}}{\Lambda_{\min}} \; \frac{\hat{C}}{\hat{c}} \;.
\end{equation*}
\end{corollary}

Note that only the operator $Q$ enters the actual construction of the preconditioner while $\tilde Q$ is only needed for its analysis.
We proceed recalling the following convenient simplifications.

\begin{proposition}
If {\bf ASM2} is replaced by the stronger condition:
{\bf ASM2':} ${b}$ dominates $\hat{a}$ on $\hat{V}$, i.e.,
$\, \hat{a}(\hat{v},\hat{v}) \lsim {b}(\hat{v},\hat{v})$, for all $ \hat{v} \in \hat{V}$,
then one can skip checking the inequalities \eqref{eq:asm_bound_Qtilde}
in {\bf ASM3}.
\end{proposition}

\noindent The statement easily follows from the first part of the proof of Theorem 1 in \cite{Oswald1996}.

\begin{proposition}
\label{prop:asm2}
{\rm
When $\tilde{V} \subset V$,
 then
$\hat{V}=V$ and one can take $\tilde{a}=\hat{a}=a$, so that an obvious
choice for $Q$ is the canonical injection. In this case, the only conditions that need be verified
for Proposition~\ref{prop:asm1} are {\bf ASM2} together with
 the existence of a linear operator $\tilde{Q}: V \to \tilde{V}$ such that}
\begin{equation}\label{eq:asm6}
{b}(v - \tilde{Q}v,v - \tilde{Q}v) \lsim {a}(v,v) \qquad \forall\, v \in V \;.
\end{equation}
\end{proposition}

The central remaining objective is now to identify for $V= V_\delta$ a suitable {\em conforming} auxiliary space $\tilde V\subset H^1_0(\Omega)$
along with the operators $Q, \tilde Q$. Since one cannot expect $\tilde V\subset V_\delta$ the operator $Q$ will be non-trivial.

\section{Reduction to a conforming problem}
\label{sec:sdgtoscg}

A common strategy for treating nonconforming discretizations is to look for a suitable conforming subspace
in order to treat the corresponding auxiliary problem with the aid of efficient multilevel techniques while the remaining
high-frequency, non-conforming part is efficiently treated by smoothing realized by relaxation sweeps. Eventually, we shall be looking for
a low-order conforming finite element space as auxiliary space as well. However, as will be explained later in more detail,
the principal tool that suggests itself,
namely introducing for each high-order element a so called {\em Legendre-Gauss-Lobatto} (LGL) (sub-)grid,
on which piecewise multilinear finite elements could be defined, does not comply with global conformity
as soon as the polynomial degrees $\bp(R)$ vary from element to element.

We will therefore split the construction of a DG-preconditioner in two main steps. The first one is to view
the largest conforming subspace $V^c_\delta \subset V_\delta \cap H^1_0(\Omega)$ as an auxiliary space.
This comes with the additional benefit that such a step can be combined with domain decomposition techniques
to obtain a DG-preconditioner with only mildly growing condition numbers \cite[]{CPP}, but now retaining this
performance for varying polynomial degrees.

As indicated above, the first main conceptual tool revolves around LGL quadrature nodes
$a={\xi}_0 < \dots < {\xi}_{j-1}
< {\xi}_{j} < \dots < {\xi}_{p}=b$ of order $p$ in the interval $I=[a,b]$ of positive length $H=b-a$.
We refer to the collection of these nodes as the LGL grid $\cG_p(I)$. It is well-known that
there exist positive weights ${w}_0, \dots, {w}_{p} \in \R$ such that
\begin{equation}\label{eq:basic.1}
\sum_{j=0}^p v({\xi}_j) \, {w}_j = \int_{{I}} v(x)\, dx
\qquad \forall v \in \mathbb{P}_{2p-1}(I) \;.
\end{equation}
Obviously, any $v \in \mathbb{P}_p(I)$ is uniquely determined by its
values on ${\cal G}_p({I})$. We recall that nodes and weights are classically defined on the reference interval
$\hat{I}=[-1,1]$, as $\hat{\xi}_j$, $\hat{w}_j$, {respectively}. By affine transformation, one has
$\xi_j=a+H(\hat{\xi}_j+1)/2$ and $w_j=(H/2)\hat{w}_j$.
We also recall that weights with index $j$ close to $0$ or $p$ satisfy ${w}_{j} \simeq H p^{-2}$ (in particular,
$w_0=w_p=H (p(p+1))^{-1}$), whereas
weights with index $j$ close to $p/2$ satisfy ${w}_{j} \simeq Hp^{-1}$. Yet, the variation in the
order of magnitude is smooth, as made precise by the estimates
\begin{equation}\label{eq:basic.1bis}
{w}_{j-1} \simeq {w}_{j} \quad \text{for} \quad 1 \le j \le p \;,
\end{equation}
which hold uniformly in $p$ and $H$ (see, e.g., \cite{Ca94,PaRo95}).

Moreover, given any element $R = \bigtimes_{k=1}^d I_k$ of $\cR$ and
denoting by $\xi_{k, j_k}$ the corresponding LGL nodes in $I_k$ the product set
\begin{equation}
\label{eq:basic.10bis}
\cG_\bp(R)=\bigtimes_{k=1}^d \cG_{p_k}(I_k)=\{\xi =(\xi_{1,j_1}, \xi_{2,j_2}, \dots, \xi_{d,j_d})
\quad \mbox{for} \ \ 0 \leq j_k \leq p_k,\ \ 1 \leq k \leq d\}
\end{equation}
may be viewed as a {\em sub-grid} for the element $R$ and will be referred to as (tensorial) LGL grid on $R$, obviously forming
a unisolvent set for $\mathbb{Q}_\bp(R)$.
Such grids have been intensely used and studied in the context of spectral methods (see, e.g., \cite{BeMa97,CaHuQuZa06,HeWa08})
and more detailed properties will be given later.

Regarding the ASM framework, setting $\tilde V= V_\delta^c \subset V_\delta=V$, we can take $\hat a(\cdot,\cdot)
= a_\delta(\cdot,\cdot)$ and $\tilde a(\cdot,\cdot)
:= a_\delta(\cdot,\cdot)\mid_{V_\delta^c\times V_\delta^c}$. Moreover, since $\tilde V\subset V$, the operator $Q: \tilde V\to V$ is
the canonical injection. Thus, defining a preconditioner based on $V^c_\delta$ as an auxiliary space it remains
to specify
the auxiliary bilinear form $b(\cdot,\cdot)$. The following definition is inspired by
a refined version of the classical inverse inequality
$\Vert v' \Vert_{0,I} \lsim p^2 H^{-1} \, \Vert v \Vert_{0,I}$
for all $v \in \mathbb{P}_p(I)$, and reads in terms of LGL nodes and weights:
\begin{equation}
\label{eq:basic.9ter}
\Vert v' \Vert_{0,I}
\lsim \left(\sum_{j=0}^p v^2({\xi}_j) {w}^{-1}_j \right)^{1/2}
\qquad \forall v \in \mathbb{P}_p(I)
\end{equation}
(see Section~\ref{ssec:stab} for a proof).
Introducing the {\em product weights} $w_{\xi}=w_{1,j_1}w_{2,j_2} \cdots w_{d,j_d}$
associated to each node $\xi \in \cG_\bp(R)$, and
denoting by
$w_{\xi,k}$ the factor $w_{j_k}$ coming from the $k$-th direction, we introduce the weights
\begin{equation}
\label{eq:ascs1}
W_\xi= \left(\sum_{k=1}^d w_{\xi,k}^{-2} \right) w_\xi
= \sum_{k=1}^d w_{\xi,k}^{-1}\Big(\prod_{j\neq k}w_{\xi,j}\Big) \;.
\end{equation}
The bilinear form $b_\delt(\cdot,\cdot) \ : \ V_\delta \times V_\delta \to \mathbb{R}$ is defined as
\begin{equation}
\label{def:b_delta}
b_\delt(u,v) = \sum_{R \in \cR} b_R(u,v)\;, \qquad \text{with } \quad
b_R(u,v):= \sum_{\xi \in \cG_\bp(R)}
u(\xi) \, v(\xi) \, c_\xi W_\xi \;.
\end{equation}
Here
the strictly positive coefficients $c_\xi \simeq 1$
(meaning that they are bounded from above and from below independently of $\xi$, $\bp$ and $\bH$)
will be chosen in applications so as to
enhance the effectivity of the ASM preconditioner, see Section~\ref{sec:numres-I} for the details.

Note that $b_\delt(\cdot,\cdot)$ is defined strictly element-wise, i.e. it does not involve any coupling
between different elements $R$; in particular, for the Lagrange basis $\{\phi_{R,\xi}: \xi \in \cG_\bp(R),\, R\in \cR\}$
associated with the LGL (sub-)grids of the elements, the matrix
${\bf B}_1:= \left(b_\delta(\phi_{R,\xi},\phi_{R',\xi'})\right)_{(R,\xi),(R',\xi')}$ is {\em diagonal}.

Since $\tilde V=V^c_\delta \subset V_\delta$, the operator $Q$ is just the canonical injection whose matrix representation ${\bf S}_1$
is, however, not the identity matrix.
The main result of this section, whose proof is deferred to Section~\ref{sec:proof-I}, reads as follows.

\begin{theorem}
\label{th:stageI}
Let $\tilde{\bf A}_1$ be the stiffness matrix for the conforming problem:
find $u_\delta\in V_\delta^c$ such that
\begin{equation}
\label{eq:conf}
a(u_\delta,v)= (f,v)_{0,\Omega},\quad v\in V_\delta^c,
\end{equation}
where $a(v,w)= \sum_{R\in\cR}(\nabla v,\nabla w)_{0,R}$ agrees with $a_\delta(\cdot,\cdot)$ on $V_\delta^c \times V_\delta^c$,
and assume that ${\bf C_{\tilde A_1}}$ is a symmetric preconditioner for \eqref{eq:conf}.
Then, denoting by ${\bf A}$ the stiffness matrix for \eqref{eq:setting.1}, there exists a constant $C_1$ such that ${\bf \bar C_A}:= {\bf B}_1^{-1} +
{\bf S}_1 {\bf C_{\tilde A_1}}{\bf S}_1^T$ is
a symmetric preconditioner for \eqref{eq:setting.1} satisfying
\begin{equation}
\label{eq:perform-1}
\kappa({\bf \bar C_A}{\bf A})\leq C_1 \kappa ({\bf C_{\tilde A_1}}\tilde{\bf A}_1),
\end{equation}
uniformly in $\delta=(\bH,\bp)$ with $C_1$ depending only on the grading conditions \eqref{eq:setting.50}.
\end{theorem}
LGL quadrature is not only essential for the analysis of the proposed schemes but allows one to formulate ``equivalent
discrete DG-bilinear forms'' which enhance computational efficiency, see the discussion in \cite{brix-thesis}.

\subsection{Numerical experiments}\label{sec:numres-I}

We demonstrate next the performance of the preconditioner ${\bf \bar C_A}$ from Theorem~\ref{th:stageI} when ${\bf C_{\tilde A_1}}= \tilde {\bf A}_1^{-1}$,
that is, the auxiliary problem is solved exactly.
For more extensive tests we refer to \cite{brix-thesis,BCD2013}.
We consider two test scenarios as shown in Fig.~\ref{fig:checkerboard}, exhibiting
a checkerboard distribution of polynomial degrees and a more monotonic polynomial degree distribution which one would expect for a mesh refinement.
Concerning the values of $p,q \in \N$, we consider the following five cases in the first scenario: (i) we either use a constant polynomial degree $q=p$, (ii) we simulate a small variation of the polynomial degree by choosing $q=p+2$, or a large variation of the polynomial degree represented
by multiplicative relations (iii) $q=3/2p$ for even $p$, (iv) $q=7/4p$ for $p$ chosen as a multiple of $4$ or (v) $q=2p$.

\begin{figure}[t]
\centering
\subfloat[First test scenario: checkerboard distribution of polynomial degrees.\label{fig:checkerboard2}]{ \includegraphics[width=0.3\linewidth]{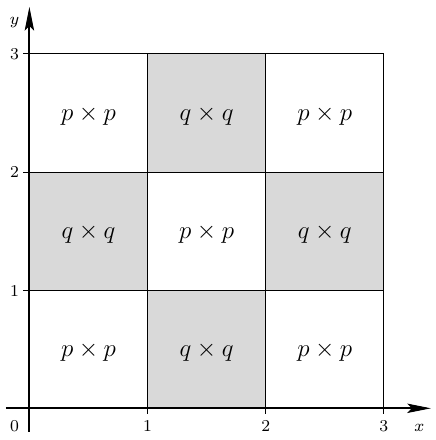}}\hspace{0.1\linewidth}
\subfloat[Second test scenario: monotonic distribution of polynomial degrees\label{fig:checkerboard5}]{\includegraphics[width=0.3\linewidth]{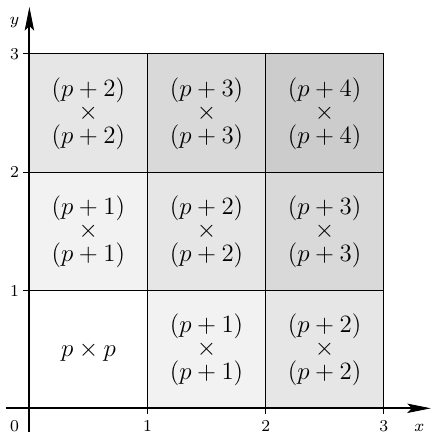}}
\caption{Grid and distribution of polynomial degrees for the test scenarios. Darker shading indicates higher polynomial degree in the patch.}
\label{fig:checkerboard}
\end{figure}

The main issue at this stage is to calibrate the tuning parameters in the bilinear form $b_\delt(\cdot,\cdot)$.
From \eqref{def:b_delta}
\begin{equation*}
b_{\delta}(u,v) =
\beta_1 \left(
c_1^2 \sum_{R \in \cR} \sum_{\xi \in \cG_\bp(R)} u(\xi) \, v(\xi) \, W_{\xi}
+\gamma \rho_1 \sum_{F \in \cF_{d-1}} \omega_F \sum_{\pm} \sum_{\xi \in \cG_\bp(F,R^\pm)} w_{F,R^{\pm}} u^{\pm}(\xi) v^{\pm}(\xi)
\right),
\end{equation*}
where $w_{F,R^{\pm}}$ is the LGL quadrature weight on $F$ seen as a face of $R^{\pm}$.
Consequently, a reasonable ansatz for the constants $c_{\xi}$ in \eqref{eq:ascs1} is
\begin{equation*}
c_{\xi} = \left\{
\begin{array}{cc}
\beta_1(c_1^2+\gamma \rho_1 \omega_F w_{F,R}/W_{\xi}),& \textnormal{for} \ \xi \in \cG_\bp(F,R), \ F\in \cF_{d-1}(R), \ R \in \cR,\\
\beta_1 c_1^2, & \textnormal{else}.
\end{array}\right.
\end{equation*}
The diagonal structure of ${\bf B}_1$ is, of course, not affected by the choice of these parameters.
Preliminary experiments concerning the constant arising in the inverse estimate \eqref{eq:basic.9ter}
reveal that $c_1^2=10$ is a good choice which we fix in our subsequent tests.

The condition numbers $\kappa(\mathbf{\bar C}_\mathbf{A} \mathbf{A})$
are depicted as contour plots for $p=8$ and for $p=16$ in Fig.~\ref{fig:contourCond1}
as functions of $\beta_1 \in [0.05,1.1]$ and $\rho_1\in [0,2]$. We observe that for both values of $p$ there is a very flat minimum of the condition number that is located near the parameter point $(\beta_1,\rho_1)=(0.15,1.25)$. From now on we fix these parameter values for the rest of the paper.
\begin{figure}[t]
\centering
\subfloat[$p=8$, $q=16$]{\includegraphics[width=0.43\linewidth]{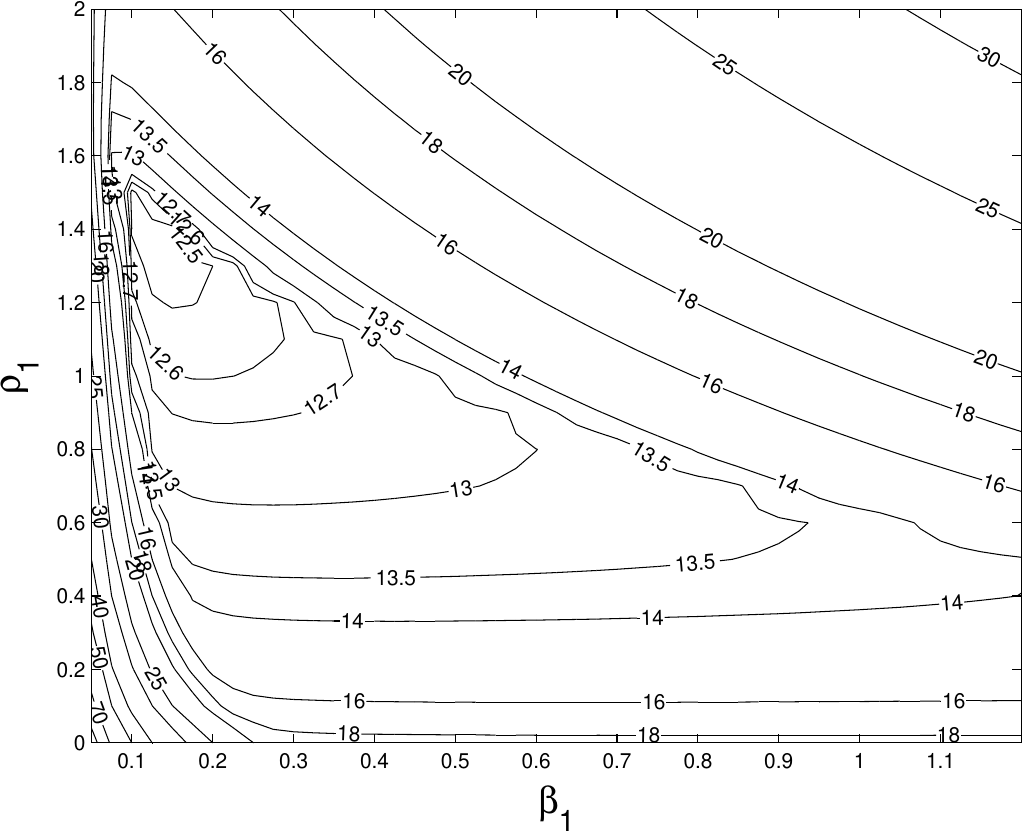}}\hspace{0.03\linewidth}
\subfloat[$p=16$, $q=32$]{\includegraphics[width=0.43\linewidth]{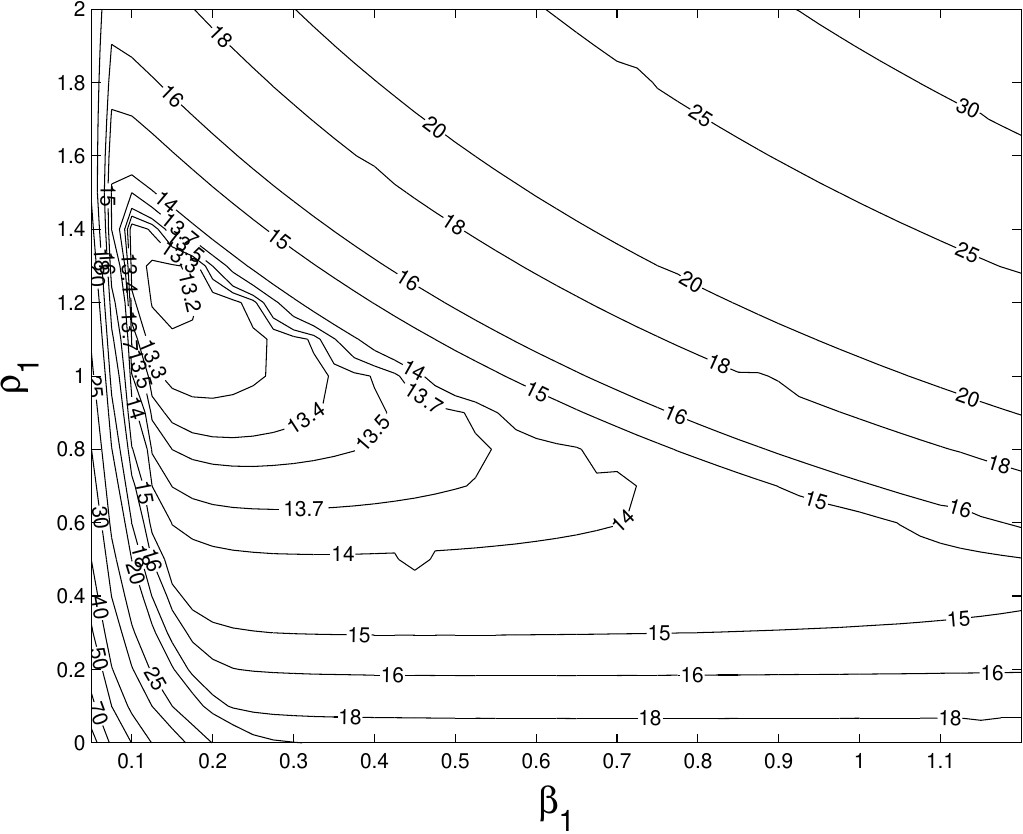}}
\caption{Contour plots illustrating the dependence of the condition number $\kappa(\mathbf{\bar C}_\mathbf{A} \mathbf{A})$ on $\beta_1 > 0$ and $\rho_1 \ge 0$.}
\label{fig:contourCond1}
\end{figure}

Fig.~\subref*{fig:resultsASM1a} shows the condition numbers obtained in the first scenario for the relations (i) - (v)
between $p$ and $q$.
We observe that the condition numbers stay uniformly bounded {as $p$ increases although} the upper bound depends on the ratio $q/p$.
In the case of additively increasing the polynomial degree $q=p+2$, the quotient $q/p$ decreases, which is also visible in Fig.~\subref*{fig:resultsASM1a}. The analogous plot for the second test scenario, representing a typical $p$-adaptation, is depicted in Fig.~\subref*{fig:resultsASM1b}.
In this case the condition numbers are almost constant and slightly smaller than $7.5$.

\begin{figure}
\centering
\subfloat[First test scenario.\label{fig:resultsASM1a}]{\includegraphics[height=0.385\linewidth]{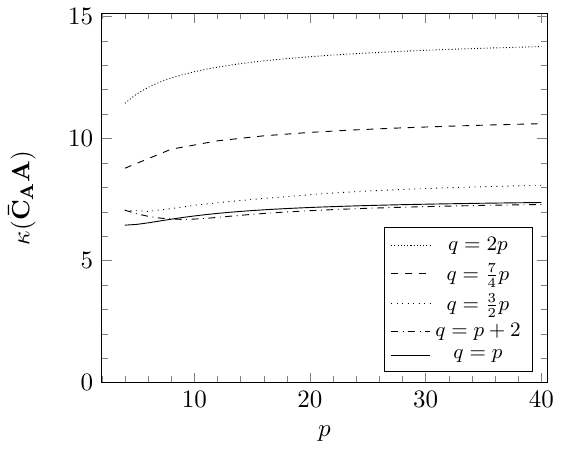}}\hspace{0.01\linewidth}
\subfloat[Second test scenario.\label{fig:resultsASM1b}]{\includegraphics[height=0.385\linewidth]{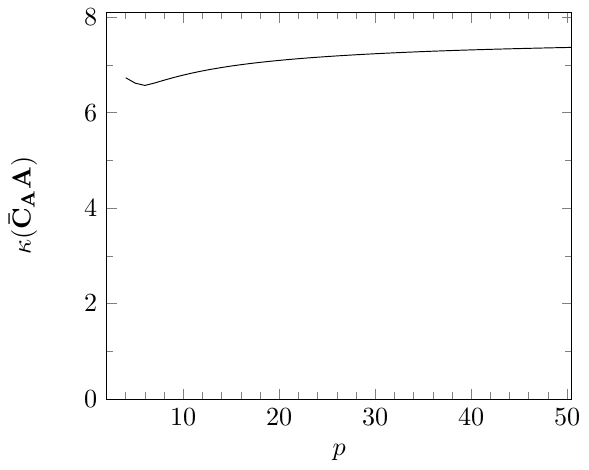}}
\caption{Condition numbers $\kappa(\mathbf{\bar C}_\mathbf{A} \mathbf{A})$.}
\label{fig:resultsASM1}
\end{figure}

\section{The conforming problem}\label{sec:conf}
In view of Theorem~\ref{th:stageI}, it remains to develop a preconditioner for the related conforming problem \eqref{eq:conf}
over the high-order conforming subspace $V^c_\delta$. We emphasize that the results of this section are of
interest in their own right because the proposed preconditioner for the conforming problem offers a solver performance that
does not seem to be available yet so far.

In what follows one should keep in mind that the degree function $\bp$ defining $V_\delta^c$, determines for each $R\in \cR$
a unique LGL (sub-)grid $\cG_\bp(R)$ along with the corresponding (micro-)partition
$\cT_\bp(R)=\cT(\cG_\bp(R))$ of $R$ into (hyper-)rectangles
$S=S_{\boldsymbol{\ell}}=\bigtimes_{k=1}^d I_{k,\ell_k}$ for $\boldsymbol{\ell} \in \bigtimes_{k=1}^d \{1, \dots, p_k\}$,
where each $I_{k,\ell_k}$ is an interval in the univariate partition $\cT_{p_k}(I_k)$ of $I_k$
(see Fig.~\subref*{fig:GNC-grid a}).
With $\cG_\bp(R)$ we associate next the local piecewise multi-linear finite element space
on $\cT_\bp(R)$ given by
\begin{equation}
\label{def-tens-interp2}
V_{h,\bp}(R)=\{v \in C^0(R) \, : \, v\mid_{S} \in \mathbb{Q}_1 \ \ \forall S \in \cT_\bp(R) \} = \bigotimes_{k=1}^d V_{h,p_k}(I_k).
\end{equation}

To state now a property that is pivotal for our purposes, we introduce the following tensor product interpolation operators
\begin{equation}
\label{def-tens-interp1}
\cI_\bp = \cI_\bp^R \ : \ C^0(R) \to \mathbb{Q}_\bp(R)\;, \qquad \cI_\bp^R = \bigotimes_{k=1}^d \cI_{p_k}^{I_k},
\end{equation}
and
\begin{equation}
\label{def-tens-interp3}
\cI_{h,\bp} = \cI_{h,\bp}^R \ : \ C^0(R) \to V_{h,\bp}(R) \;, \qquad \cI_{h,\bp}^R = \bigotimes_{k=1}^d \cI_{h,p_k}^{I_k} \;.
\end{equation}
where $\cI_{p_k}^{I_k}, \cI_{h,p_k}^{I_k}$ are the corresponding univariate operators, satisfying
$\cI_{p_k}^{I_k} v \in \mathbb{P}_{p_k}(I_k)$, $\cI_{h,p_k}^{I_k} v \in V_{h,p_k}(I_k)$ and
\begin{equation*}
(\cI_{p_k}^{I_k}v)(\xi_{k,j})= v(\xi_{k,j}), \qquad (\cI_{h,p_k}^{I_k}v)(\xi_{k,j})= v(\xi_{k,j}), \quad0\leq j\leq p_k.
\end{equation*}
The crucial fact is that the operators $\cI_{h,\bp}^R$ induce uniformly bounded topological isomorphisms between
$\mathbb{Q}_\bp(R)$ and $V_{h,\bp}(R)$
with respect to both the $L_2$ and $H^1$ norms, whose inverse is $\cI_\bp^R$.
\begin{property}
\label{prop:Nhequivalence-multi}
\! {\rm (\cite{Ca94})} \, For any $v \in \mathbb{Q}_\bp(R)$,
set $v_h:=\cI_{h,\bp}^R v$. Then,
\begin{equation}
\label{eq:basic.16}
\Vert {v} \Vert_{0,R} \simeq \Vert {v}_{h} \Vert_{0,R} \qquad \mbox{and } \qquad
\Vert \nabla {v} \Vert_{0,R} \simeq \Vert \nabla {v}_h \Vert_{0,R} \;.
\end{equation}
\noindent
The constants in both relations
are independent of $\bp$ and $\bH$. \endproof
\end{property}
The relations \eqref{eq:basic.16} follow from the analogous univariate relations (see Property~\ref{prop:Nhequivalence-uni})
combined with the following
frequently used fact.
\begin{proposition}
\label{propos:tensorprop}
For $1 \leq k \leq d$, let $V_k, \ W_k \subset H^1(I_k)$ be finite dimensional subspaces.
Let $L_k : V_k \to W_k$ be linear operators satisfying, for $m=0, 1$,
\begin{equation}\label{eq:tensorprop1}
\Vert L_k v \Vert_{m,I_k} \lsim \Vert v \Vert_{m,I_k} \quad \forall v \in V_k\;.
\end{equation}
Then, setting $V:=\bigotimes_{k=1}^d V_k$ and $ W:=\bigotimes_{k=1}^d W_k$,
the operator $L= \bigotimes_{k=1}^d L_k : V \to W$ satisfies, for $m=0, 1$,
\begin{equation}\label{eq:tensorprop2}
\Vert L v \Vert_{m,R} \lsim \Vert v \Vert_{m,R} \quad \forall v \in V \;.
\end{equation}
Moreover, the same relations hold for the corresponding seminorms.
\end{proposition}

Property~\ref{prop:Nhequivalence-multi} plays a pivotal role in the analysis of the preconditioners, see
in particular Section~\ref{ssec:stab}. Its significance is indicated by the fact that $V_{h,\bp}(R)$ is an ideal auxiliary space for $\mathbb{Q}_\bp(R)$
when $\cR$ consists of a single element. However, even ignoring the question of efficiently solving the obtained
low-order problem, when dealing with the general case of several elements the corresponding
LGL (sub-)grids do not match at the element interfaces when the degrees $\bp(R)$ vary from element to element as
shown by Fig.~\subref*{fig:GNC-grid a}.
Globally conforming piecewise multilinear finite elements on the corresponding composite partitions would therefore {\em not}
form suitable auxiliary spaces since the traces at the element interfaces would have to be at best linear on the whole interface.

\begin{figure}
\centering
\subfloat[LGL grids and partitions \label{fig:GNC-grid a}]{\includegraphics[width=0.43\textwidth]{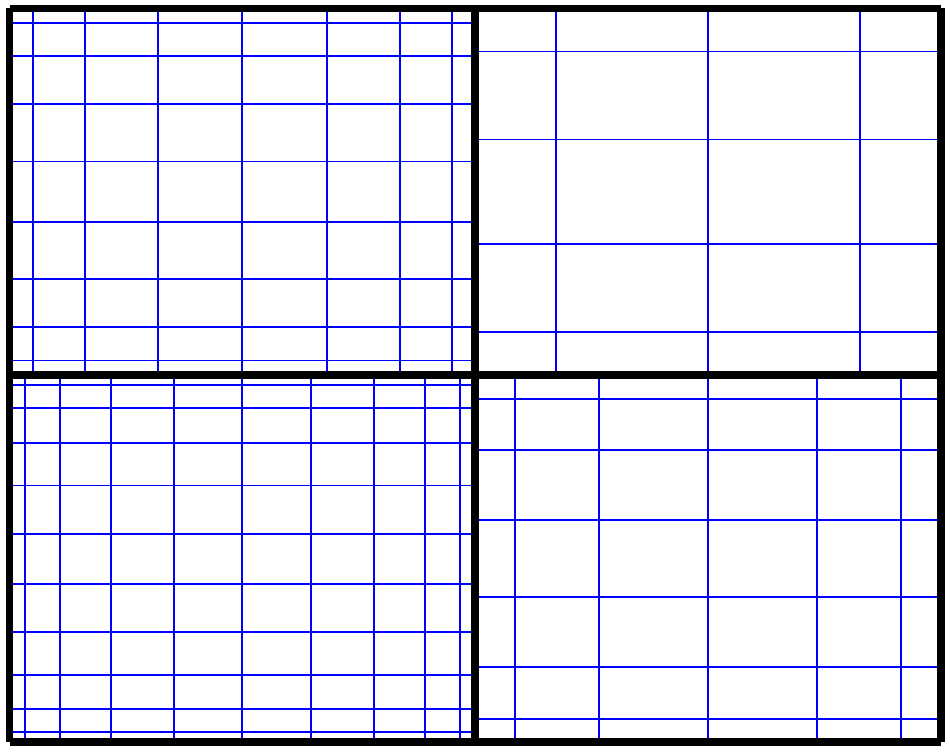}}\hspace{0.1\linewidth}
\subfloat[Associated dyadic grids and partitions \label{fig:GNC-grid b}]{\includegraphics[width=0.43\textwidth]{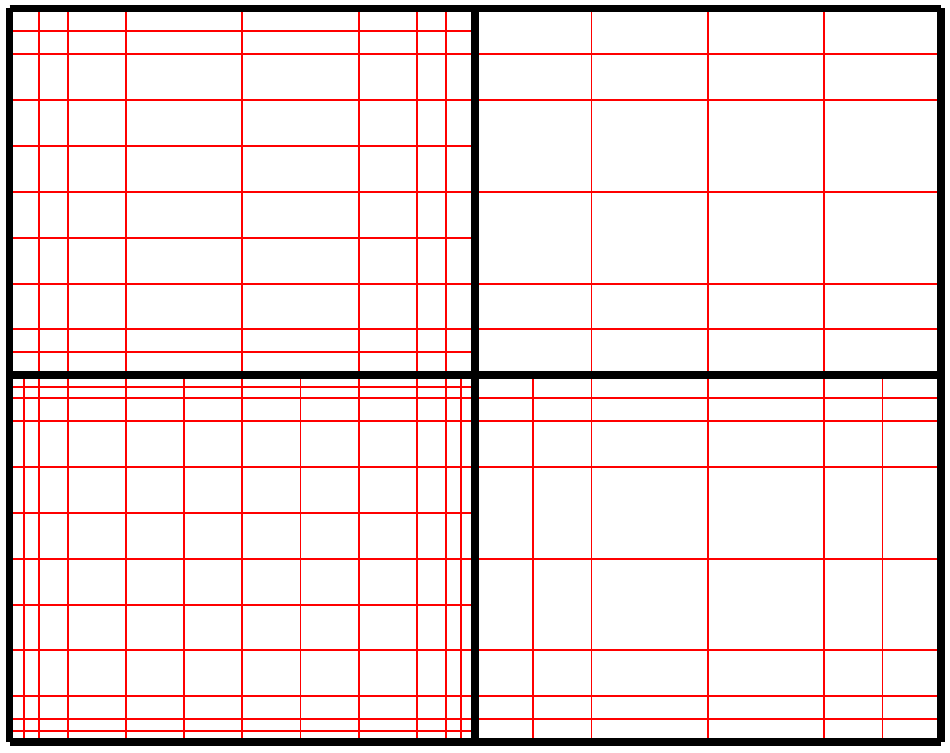}}
\caption{LGL and dyadic sub-grids in a patch of elements in $\mathbb{R}^2$}
\label{fig:GNC-grid}
\end{figure}

The next major conceptual tool is therefore the construction of suitable composite grids which, on the one hand are
still sufficiently close to the composite LGL grids but, on the other hand, give rise to sufficiently rich
globally conforming {\em nested} low-order finite element spaces, to be used as auxiliary spaces.

\subsection{Dyadic meshes}\label{sec:dyadic}

As before we refer to an ordered collection $\cG=\{\xi_j : 0 \le j \le p\} \subset I =[a,b]$ with
$
a=\xi_0 < \dots < \xi_{j-1}< \xi_j < \dots < \xi_p=b$, as a grid in $I$ and continue to denote
by
$\cT=\cT(\cG)$ the induced partition comprised of the intervals $I_j = [\xi_{j-1},\xi_j]$.
Such a grid is called {\em locally $C_g$-quasiuniform} if there exists
a constant $C_g > 1$ such that
\begin{equation}\label{grid}
C_g^{-1} \le \frac{ \abs{I_{j+1}} }{ \abs{I_j} }\le C_g \;, \qquad 1 \le j \le p-1 \,.
\end{equation}
LGL grids are locally $C_g$-quasiuniform, see \cite{BCD2013} for estimates of $C_g$.

Moreover, to be able to ``compare'' different grids we call
a grid $\cG$ in $I$ {\em locally $(A,B)$-uniformly equivalent} to another ordered grid $\tilde{\cG}$ if there exist
constants $0 < A < B$ such that
\begin{equation}
\label{gridunieq}
\forall I_j \in \cT(\cG), \ \ \forall \tilde{I}_l \in \cT(\tilde{\cG}) \;, \qquad I_j \cap \tilde{I}_l \neq \emptyset
~\implies~
A \le \frac {\abs{I_j}}{\abs{\tilde{I}_l}} \le B \;.
\end{equation}

The central issue of this section is to construct for a given grid $\cG$ in $I$ an equivalent (in the sense of \eqref{gridunieq})
dyadic grid $\cD$.
To this end, we first introduce an intermediate construction.

Given a real $\alpha >0$ and an initial dyadic partition $\cD_0$, a dyadic partition $\cD$
(identified for brevity with the associated grid)
can be constructed {iteratively} as follows: \\

\noindent
{\bf Dyadic}$\,[\cG,\cD_0,\alpha]\to \cD$:
\begin{itemize}{\it
\item[(i)]
Set $\cD:=\cD_0$.
\item[(ii)]
While there exists $D\in \cD$ such that
\begin{equation}
\label{split}
\abs{D} > \alpha |\oI(D,\cG)| \;,
\end{equation}
where $\oI(D,\cG):= {\rm argmax}\,\{|I_j|: I_j\in \cT(\cG),\, D\cap I_j\neq \emptyset\}$, then
split $D$ by halving it and replace it by its two children $D', D''$, $D=D'\cup D''$, i.e.
\begin{equation*}
(\cD\setminus \{D\} )\cup \{D',D''\}\to \cD.
\end{equation*}
}
\end{itemize}

Such a dyadic mesh generator, applied to the sequence of
LGL grids $\cG_p=\cG_p(I)$, produces grids $\cD_p^* :={\bf Dyadic}\,[\cG_p,\{ I \},\alpha]$
which enjoy a number of useful structural properties such as always containing the midpoint of the interval and being symmetric
around the midpoint. However, the sequence of these grids still fails to be nested although exceptions seem to be very rare (see \cite{brix-thesis}).

Since nestedness is essential we remedy the deficiency of ${\bf Dyadic}\,[\cG_p,\{ I \},\alpha]$ by
employing the following recursive definition:\\

\noindent
{\bf NestedDyadic}$\,[\cG_p,\{ I \},\alpha]\ \to \cD_p$:
\begin{itemize}
{\it
\item[(i)] Given $\alpha >0$, set
$
\cD_1:= {\bf Dyadic}[\cG_1,\{ I \},\alpha]$.
\item[(ii)]
For $p>1$, given $\cD_{p-1}$, set
$
\cD_{p}:= {\bf Dyadic}[\cG_{p},\cD_{p-1},\alpha]$.
}
\end{itemize}

The following result from \cite{BCD2013} will be essential for the realization of ASM-conditions.
\begin{theorem}
\label{lemdyadic-2}
For all $p \geq 1$, the dyadic meshes $\cD_{p}$ are nested, locally quasi-uniform,
and locally $(A,B)$-uniformly equivalent to
$\cG_p$ with constants $A,B$ specified as follows:
\begin{equation}
\label{dyeq2}
\forall D \in \cD_p \;, \ \ \forall I_j \in \cT(\cG_p) \;, \qquad I_j \cap D \neq \emptyset
~\implies~
A:=\alpha^{-1}\leq \frac{|I_j|}{|D|} \leq \frac{2C_g}{\min\{\alpha C_g^{-1},1\}}=:B \;.
\end{equation}
Furthermore,
\begin{equation}
\label{dyeq2bis}
\ \qquad
{\rm card}\, \cD_p \simeq {\rm card}\, \cG_p .
\end{equation}
In addition, if $1 \leq \alpha \leq 1.25$, for all values of $p$ such that the sequence of LGL interval lengths
is such that the quotients $ \frac{ \abs{I_{j+1}} }{ \abs{I_j} }$ are monotonically decreasing from the left-end to the center
of the interval $I$, then the grids $\cD_p$ are {\em graded}, i.e., any two adjacent intervals differ in subdivision
generation by at most one.
\end{theorem}

In order to keep the cardinality of $\cD_p$ as close as possible to that of $\cG_p$, numerical evidence suggests to choose
the parameter $\alpha$ close to 1,
see \cite{brix-thesis} for more details. The above condition on $p$ implying gradedness has been numerically checked for
$p$ up to 2000.

\smallskip
With the above construction of dyadic grids in intervals at hand, dyadic grids $\cD_\bp(R)$, $\bp=\bp(R)$,
on hyper-rectangles $R$
are obtained in a straightforward manner by tensorization. Just as the LGL grids $\cG_\bp(R)$, the dyadic grids $\cD_\bp(R)$ may be viewed as {\em sub-grids} on the (macro-) elements $R\in\cR$.
Let $\cT_{D,\bp}(R)=\cT(\cD_{\bp}(R))$ be the partition of $R$ into dyadic (hyper-)rectangles
$E=E_{\bf m}=\bigtimes_{k=1}^d D_{k,m_k}$, where each $D_{k,m_k}$ is a dyadic interval in the partition
$\cT_{D,p_k}(I_k)= \cT(\cD_{p_k}(I_k))$ of $I_k$; let $V_{h,D,p_k}(I_k)$ be the space of continuous piecewise linear functions subordinate to this partition.

We are now prepared to define the auxiliary space for $V_\delta^c$ as follows:
\begin{equation}
\label{def-tens-interp5}
V_{h,D,\bp}:= \{v\in H^1_0(\Omega): v\!\mid_{R}\in V_{h,D,\bp}(R),\, R\in \cR\},
\end{equation}
where
\begin{equation}
\label{def-tens-interp4}
V_{h,D,\bp}(R)=\{v \in C^0(R) \, : \, v\!\mid_{E} \in \mathbb{Q}_1 \ \ \forall \,E \in \cT_{D,\bp}(R) \}
 = \bigotimes_{k=1}^d V_{h,D,p_k}(I_k) .
\end{equation}
Thus, $V_{h,D,\bp}$ consists of globally conforming piecewise multi-linear functions
on a composite dyadic partition of $\Omega$.
In sharp contrast to the spaces on analogous composite LGL grids, due to the nestedness of the dyadic grid hierarchies,
the traces of functions in $V_{h,D,\bp}$ on
an element interface are now arbitrary continuous piecewise linear functions on the coarser one of the two adjacent ``glueable'' dyadic partitions.

\subsection{The auxiliary bilinear forms} \label{sect: auxbilforms}
Since in terms of the auxiliary space method $V = V_\delta^c$ and $\tilde V = V_{h,D,\bp}$ are both conforming spaces we
can simply take $a(\cdot,\cdot) = \tilde a(\cdot,\cdot) = \hat a(\cdot,\cdot)$. A little more care is required to
identify a suitable form $b(\cdot,\cdot)$. Unfortunately, it turns out that simply retaining the definition
\eqref{def:b_delta} will not work. In fact, using an inverse estimate throughout an element is too strong to
allow the direct estimates in {\bf ASM3} to hold since $ \hat a(\cdot,\cdot)$ no longer involves jump terms which in the
first stage of reduction to a conforming problem had to be controlled by the $b_\delt(\cdot,\cdot)$ form in {\bf ASM2}.
\begin{remark}
\label{rem-impossible}
The latter comments hint at the fact that when
going directly from the high-order DG-discretization to a globally conforming
auxiliary low-order discretization, it is difficult to identify a suitable auxiliary form $b(\cdot,\cdot)$.
Our approach of splitting the construction of an auxiliary space into two stages will eventually
provide a suitable preconditioner ${\bf C_B}$ by composition.
\end{remark}

As a consequence we will propose next a ``tamed'' version of $b_\delt(\cdot,\cdot)$ from \eqref{def:b_delta}
using the inverse inequality only in those subcells of an element $R$ which are not ``too anisotropic''.

Observing that
$a_R(u,v)=\sum_{k=1}^d a_{R,k}(u,v)=\sum_{k=1}^d\int_R \partial_{x_{k}} u \, \partial_{x_{k}} v \, dx$, we make the ansatz
\begin{equation}
\label{eq:def_b_tens2}
b_2(u,v)= \sum_{R \in {\cal R}} b_R(u,v) \;, \qquad \
b_R(u,v)= \sum_{k=1}^d b_{R,k}(u,v) ,
\end{equation}
and proceed to define $b_{R,k}(u ,v )$ in a cell-wise manner for any $u, v\in V+\tilde V$ so as to ensure, in view of {\bf ASM2}, that still
$a_{R,k}(v,v) \lsim b_{R,k}(v ,v )$.
To that end, we consider the partition $\cT_\bp(R)$ of $R$ into LGL subcells $S_{\boldsymbol{\ell}} = S_{\boldsymbol{\ell}}(R) = \prod_{k=1}^d I_{k,\ell_k}$ already introduced at the beginning of Section~\ref{sec:conf}.
Next, fixing a constant $C_\textnormal{aspect}>0$, for each $1 \leq k \leq d$, we decompose $\cT_\bp (R)$ into
two parts $\cT_{\bp,k}^{(0)}(R)$, $\cT_{\bp,k}^{(1)}(R)$ defined as follows.
Setting $h_{l,\ell_l}:=\abs{I_{l,\ell_l}}$, we let
\begin{equation}
\label{aspect}
S_{\boldsymbol{\ell}}\in \cT_{\bp,k}^{(0)}(R)\,\quad \mbox{if}\,\quad \frac{ \max_{l\neq k} h_{l,\ell_l} }{h_{k,\ell_k}} > C_\textnormal{aspect},
\qquad S_{\boldsymbol{\ell}}\in \cT_{\bp,k}^{(1)}(R)\,\,
\mbox{otherwise}.
\end{equation}
Thus $\cT_{\bp,k}^{(0)}(R)$ and $\cT_{\bp,k}^{(1)}(R)$ are comprised of ``strongly anisotropic'' and
``sufficiently isotropic" cells, respectively.
We now define
\begin{equation}
\label{eq:defb.b0.b1}
b_{R,k}(u ,v ) = \sum_{S_{\boldsymbol{\ell}} \in \cT_{\bp,k}^{(0)}(R)} b^{(0)}_{R,k,S_{\boldsymbol{\ell}}}(u ,v )
+ \sum_{S_{\boldsymbol{\ell}} \in \cT_{\bp,k}^{(1)}(R)} b^{(1)}_{R,k,S_{\boldsymbol{\ell}}}(u ,v )
\ =: \ b_{R,k}^{(0)}(u ,v ) +b_{R,k}^{(1)}(u ,v ) ,
\end{equation}
where we retain integration in the coordinate with the derivative on anisotropic cells by setting
\begin{align}
b^{(0)}_{R,k,S_{\boldsymbol{\ell}}}(u ,v ):=\sum_{\xi' \in \cF_0(S_{\boldsymbol{\ell},k}')} \omega_{\boldsymbol{\ell},k}'
 \int_{I_{k,\ell_k}} \partial_{x_{k}} u_h (x_k,\xi') \partial_{x_{k}} v_h(x_k,\xi') \, dx_k ,\quad S_{\boldsymbol{\ell}} \in \cT_{\bp,k}^{(0)}(R) \;,
\label{eq:def_b_RkSl_0}
\end{align}
where $u_h:=\cI_{h,\bp}^R u$, $v_h:=\cI_{h,\bp}^R v$ as in Property~\ref{prop:Nhequivalence-multi},
$S_{\boldsymbol{\ell},k}' := \prod_{l=1, l \not = k}^d I_{l, h_l}$,
the weight
\begin{equation}\label{eq:def.omegaprime}
\omega_{\boldsymbol{\ell},k}':= \prod_{l=1, l \not = k}^{d} h_{l,\ell_l} = \vol_{d-1}(S_{\boldsymbol{\ell},k}')
\end{equation}
is equal to the $(d-1)$-dimensional volume of $S_{\boldsymbol{\ell},k}'$, and as before, $\cF_0(S_{\boldsymbol{\ell},k}')$
is the set of all its vertices. Instead, on isotropic cells we employ an inverse inequality to obtain
\begin{align}
b^{(1)}_{R,k,S_{\boldsymbol{\ell}}}(u ,v )
:=\sum_{\xi' \in \cF_0(S_{\boldsymbol{\ell},k}')} \sum_{\xi \in \cF_0(I_{k,\ell_k})}c_{\xi,\xi'}
\frac{\omega_{\boldsymbol{\ell},k}'}{ h_{k,\ell_k}} \, u_h(\xi,\xi') \, v_h(\xi,\xi') ,\quad S_{\boldsymbol{\ell}}\in \cT_{\bp,k}^{(1)}(R),
\label{eq:def_b_RkSl_1}
\end{align}
where as before the $c_{\xi,\xi'}$ are tuning constants of order one which have to be chosen judiciously
in practical applications. To simplify the exposition we shall suppress them in what follows.

In summary, we have
\begin{equation}
\label{def:b_new}
b_2(u,v)=\sum_{R\in \cR}\sum_{k=1}^d\Big(\sum_{S_{\boldsymbol{\ell}} \in \cT_{\bp,k}^{(0)}(R)} b^{(0)}_{R,k,S_{\boldsymbol{\ell}}}(u,v)
+ \sum_{S_{\boldsymbol{\ell}}\in \cT_{\bp,k}^{(1)}(R)}b^{(1)}_{R,k,S_{\boldsymbol{\ell}}}(u,v)\Big),
\end{equation}
where $b^{(i)}_{R,k,S_{\boldsymbol{\ell}}}(u,v)$, $i=0,1$, are defined by \eqref{eq:def_b_RkSl_0}, \eqref{eq:def_b_RkSl_1}, respectively.

\subsection[The operator Q]{The operator $Q$}
It remains to construct a suitable operator $Q: \tilde V = V_{h,D,\bp} \to V= V_\delta^c$ which is now non-trivial because $\tilde V\not\subseteq V$.
Given any $R\in\cR$, it is clear that the trace of a high-order function
from $V_\delta^c$ on any face $F$ of $R$ must have (componentwise) the minimal degree among all the
$\bp(F,R')$ with $R' \in \cR(F)$. To satisfy this condition, we proceed vertex-wise. Precisely, for any vertex $z\in \cF_0(R)$,
let $E_k, k=1,\ldots,d$, be the edges emanating from $z$ and let the {\em minimal vertex degree} be defined as
\begin{equation*}
\bp^*_z := (p^*_1,\ldots,p^*_d)\in \N^d, \qquad \text{with \ } p^*_k = p^*_k(E_k):= \min_{R' \supset E_k} p_k(R') \;.
\end{equation*}
Note that for a given edge $E=[z,y]\in \cF_1(R)$ the degrees
$\bp^*_z$ and $\bp^*_y$ may differ. A proper ``fusion'' can be realized with the aid of the
the {\em element shape functions} $\Phi_z= \Phi_z^R\in \mathbb{Q}_1(R)$, defined by $\Phi_z(z')=\delta_{z,z'}$, $z,z'\in \cF_0(R)$,
that allow us to disentangle first the minimal vertex degrees. In fact, introducing the interpolation operators associated with
the dyadic grids $\cD_p(I)$
\begin{equation}
\label{eq:basic.66}
\cI_{h,D,p}^I : C^0(I) \to V_{h,D,p}(I) \;, \qquad (\cI_{h,D,p}^I v)(\zeta)=v(\zeta)
\quad \forall \zeta \in \cD_p(I) \;, \ \ \forall v \in C^0(I)\;,
\end{equation}
along with their element-wise tensor product
\begin{equation}
\label{def-tens-interp5a}
\cI_{h,D,\bp} = \cI_{h,D,\bp}^R \ : \ C^0(R) \to V_{h,D,\bp}(R) \;, \qquad \cI_{h,D,\bp}^R = \bigotimes_{k=1}^d \cI_{h,D,p_k}^{I_k} ,
\end{equation}
we readily see that if $\tilde v_R \in V_{h,D,\bp}(R)$, then
\begin{equation}
\label{eq:def.vstarz}
\tilde{v}_z^* := \cI^R_{h,D,\bp_z^*} \left( \Phi_z \tilde{v}_R \right) \in V_{h,D,\bp_z^*}(R)
\quad \text{and} \quad \tilde{v}_R^* := \sum_{z \in \cF_0(R)} \tilde{v}_z^* \in V_{h,D,\bp}(R).
\end{equation}
Similarly,
\begin{equation}
\label{eq:def.vstarR}
{v}_z^* = \cI^R_{\bp_z^*} \, \tilde{v}_z^* \in \mathbb{Q}_{\bp_z^*}(R)
\quad \text{and} \quad
Q_R \tilde{v}_R := {v}_R^* := \sum_{z \in \cF_0(R)} {v}_z^* \in \mathbb{Q}_{\bp}(R)
\end{equation}
The desired global operator $Q: \tilde V = V_{h,D,\bp} \to V= V_\delta^c$ can now be defined as
\begin{equation}
\label{def:operQ}
(Q \tilde{v})\mid_{R} :=Q_R \tilde{v}_R = v_R^* \qquad \quad \forall \, R \in \cR \; \qquad \forall \, \tilde{v} \in \tilde{V} .
\end{equation}

\begin{proposition}
\label{prop:continuous}
$Q \tilde{v}$, defined by \eqref{def:operQ}, belongs to $V^c_\delta$.
\end{proposition}
\begin{proof}
First note that for any edge $E \in \cF_1(R)$, one has
\begin{equation}
\label{prop:edgeidentity}
(\tilde{v}_R^*)_{|E}= (\tilde{v}_R)_{|E} .
\end{equation}
In fact, if $E$ has vertices $z_1$ and $z_2$, then
\begin{equation*}
\left( \tilde{v}_{z_1}^* + \tilde{v}_{z_2}^* \right)_{|E} \cI^E_{h,D,\bp^*(E)} \left((\Phi_{z_1}+\Phi_{z_2})\tilde{v}_R \right)_{|E} \cI^E_{h,D,\bp^*(E)} \left(\tilde{v}_R \right)_{|E} = \left(\tilde{v}_R \right)_{|E}
\end{equation*}
since by continuity $(\tilde{v}_R)_{|E} \in V_{h,D,\bp^*(E)}(E)$; on the other hand,
$(\tilde{v}_{y}^*)_{|E} = 0$ for all $y \in \cF_0(R) \setminus \{z_1,z_2 \}$. Thus,
$(\tilde{v}_R^*)_{|E} = \left( \tilde{v}_{z_1}^* + \tilde{v}_{z_2}^* \right)_{|E} =(\tilde{v}_R)_{|E}$
which confirms \eqref{prop:edgeidentity}.

Furthermore, for any interface $F \in \cF_{d-1}$, with $\cR(F)=\{R',R''\}$, one has
\begin{equation}
\label{prop:facecont}
(\tilde{v}_{R'}^*)_{|F}= (\tilde{v}_{R''}^*)_{|F} \qquad \text{and} \qquad
({v}_{R'}^*)_{|F}= ({v}_{R''}^*)_{|F}.
\end{equation}
To see this, let $\cF_0(F)=\cF_0(R')\cap\cF_0(R'')$
be the set of vertices of $F$. Since, by continuity of $\tilde v$, $\tilde{v}_{F} :=(\tilde{v}_{R'})_{|F}=(\tilde{v}_{R''})_{|F}$,
one has $ \big(\Phi_z \tilde{v}_{R'}\big)_{|F}= \Phi_z^F \tilde{v}_{F} = \big(\Phi_z \tilde{v}_{R''}\big)_{|F}$, $z\in \cF_0(F)$.
Hence, denoting by $\bq_z^*=(\bp_z^*)' \in \mathbb{N}^{d-1}$ the reduced degree vector obtained from $\bp_z^*$ by dropping the component in the direction orthogonal to $F$, one concludes that
\begin{align*}
(\tilde{v}_{R'}^*)_{|F}
&=\sum_{z \in \cF_0(F)} \cI^F_{h,D,\bq_z^*} \left( \Phi_z \tilde{v}_{R'} \right)_{|F}
= \sum_{z \in \cF_0(F)} \cI^F_{h,D,\bq_z^*} \left( \Phi_z^F \tilde{v}_{F} \right)\\
&=\sum_{z \in \cF_0(F)} \cI^F_{h,D,\bq_z^*} \left( \Phi_z \tilde{v}_{R''} \right)_{|F}= (\tilde{v}_{R''}^*)_{|F} \;,
\end{align*}
where we have used that $\bq_z^*$ is the same for both $R'$ and $R''$. This confirms the first part of the assertion.
Abbreviating
$\tilde{v}_z^F := \cI^F_{h,D,\bq_z^*} \left( \Phi_z^F \tilde{v}_{F} \right)$,
the second one follows from
$
({v}_{R'}^*)_{|F}= \sum_{z \in \cF_0(F)} \cI^F_{\bq_z^*} \, \tilde{v}_z^F = ({v}_{R''}^*)_{|F}
$,
which completes the proof of \eqref{prop:facecont} and hence of the proposition.
\end{proof}

The main result of this section, whose proof is given in Section~\ref{sec:proof-II}, can be phrased as follows.

\begin{theorem}
\label{th:2}
Let $\mathbf{\tilde A}_2$ be the stiffness matrix with respect to a basis for $\tilde V= V_{h,D,\bp}$ and let $\mathbf{B}_2$ be the
matrix representation of the operator $B_2$ defined by $\langle B_2 v,w\rangle = b_2(v,w)$, $v,w\in V^c_\delta + V_{h,D,\bp}$,
with $b_2(\cdot,\cdot)$ defined by \eqref{def:b_new}.
Finally let $\mathbf{S}_2$ denote the matrix representation of the operator $Q$, defined by \eqref{def:operQ}.
Then, there exists a constant $C_2$ depending only on the grading conditions \eqref{eq:setting.50}, such that,
whenever $\mathbf{C}_{\mathbf{\tilde A}_2}$ is a symmetric preconditioner for $\mathbf{\tilde A}_2$, the matrix
$\mathbf{C}_{\mathbf{\tilde A}_1}:= \mathbf{B}_2^{-1} + \mathbf{S}_2 \mathbf{C}_{\mathbf{\tilde A}_2} \mathbf{S}_2^T$ is a symmetric preconditioner for \eqref{eq:setting.1} satisfying
\begin{equation}
\label{eq:2}
\kappa(\mathbf{C}_{\mathbf{\tilde A}_1} {\mathbf{\tilde A}_1}) \leq C_2 \kappa(\mathbf{C}_{\mathbf{\tilde A}_2} {\mathbf{\tilde A}_2})
\end{equation}
uniformly in $\delta =(\bH,\bp)$, subject to the grading conditions \eqref{eq:setting.50}.
\end{theorem}

Note that, due to the global continuity of the functions in $V^c_\delta + V_{h,D,\bp}$ the matrix
$\mathbf{B}_2$ is no longer diagonal but very sparse. More precisely the sparsity depends quantitatively on
the control of the aspect ratios in \eqref{aspect}. Therefore, $\mathbf{B}_2$ is, of course, not inverted
exactly. A more elaborate discussion of this issue can be found in \cite{brix-thesis,BCD2013}.
The numerical experiments later given in Section~\ref{sec:mumres-II} indicate that the approximate inversion of $\mathbf{B}_2$
is facilitated by efficient smoothing relaxations.

\subsection[A multiwavelet preconditioner for A2]{A multiwavelet preconditioner for $\mathbf{\tilde A}_2$}\label{sec:multiwavelet}

It remains to specify the preconditioner $\mathbf{C}_{\mathbf{\tilde A}_2}$ in Theorem~\ref{th:2}.
Again we stress that such a preconditioner is already of interest for spectral discretization on a single element as well as
for high-order conforming discretizations, independent of the DG context.

On the one hand, we are now in the comfortable situation of dealing with a {\em multilevel hierarchy of nested}
partitions $\cD_\bp(\cR)$ of $\Omega$, obtained as the union of dyadic element partitions $\cD_{\bp(R)}(R)$, $R\in \cR$.
On the other hand, a straightforward application of standard preconditioning concepts is impeded by the fact that the dyadic grids are highly anisotropic.
In fact, the currently known results do not seem to imply that BPX-type techniques would give rise to uniformly bounded
condition numbers.

We therefore propose a strategy that again resorts to the auxiliary space method, this time for the scenario
that $V \subset \tilde V$. The envisaged auxiliary space will be seen to
have uniformly $H^1$-stable splittings in the sense of \cite{GO1995} which, in turn, are obtained in two steps:
\begin{itemize}
\item[(a)]
For each element $R\in \cR$ construct a local $H^1$-stable splitting for the corresponding
finite element space $V_{h,D,\bp}(R)$ subordinate to the dyadic subgrid $\cD_\bp(R)$.
\item[(b)]
Glue the local $H^1$-stable splittings to a global $H^1$-stable splitting for the global conforming
finite element space $V_{h,D,\bp}$.
\end{itemize}
For step (a) we exploit the fact that $V_{h,D,\bp}(R)$ is a tensor product space. For this to work we need a frame or basis for the one-dimensional
factors with the property that properly scaled versions form a frame or basis for $L_2$ and for $H^1$.
To obtain such systems we use certain piecewise linear $L_2$-orthonormal compactly supported multi-wavelets on an interval.
In particular, the boundary conditions of those multi-wavelets allows us to glue them across end points of adjacent intervals in such a
way that they still give rise to globally continuous $L_2$-orthonormal wavelets on the unions of such intervals.
This is the essential prerequisite for step (b).

\subsubsection{Piecewise affine orthogonal multiwavelets on an interval}
We briefly recall first the piecewise affine orthogonal multiwavelets on all of $\R$ constructed in \cite[Example in Section 3.1]{DGH1996}.
For our purpose it is important to relate them to the standard piecewise linear multiresolution spaces $(\check S_{j})_{j \in \Z}$ on the
dyadic grids $2^{-j}\Z$. Specifically, $\check S_{j}$ is obtained by dilating the elements of $\check S_{0}$
by $2^j$, which itself is generated by the integer translates of the standard piecewise linear hat function with integer knots.
An intertwining technique yields then a new multiresolution analysis of subspaces of piecewise linear functions $S_i$ that are spanned by three orthogonal scaling functions $\phi^0$, $\phi^1$, and $\phi^2$ in such a way that
\begin{equation}\label{eq:intertwinedMRA}
\check S_{i+1} \subset S_i \subset \check S_{i+2}.
\end{equation}
The orthogonal complements $W_j := S_{j+1}\ominus S_j$ can be shown to be spanned
by translated and dilated versions of three orthogonal wavelet functions $\psi^0$, $\psi^1$, and $\psi^2$, i.e.,
\begin{equation*}
W_j:=\overline{ \spn \{ \psi^i_{[j,k]} : i \in \{0,1,2\}, k \in \Z \}}.
\end{equation*}
Here we denote the translated and dilated version of a function $f \in L_2(\R)$ by $f_{[j,k]}:=2^{j/2} f(2^j \cdot -k)$.
For convenience we gather in the following the indices of a translated and dilated scaling function or wavelet in a single multilevel index $\lambda=(i,j,k)$, where $\abs{\lambda}:=j$ denotes the level of $\lambda$.
With the convenient definition $W_{-1}:=V_{0}$ we have the orthogonal decomposition $L_2(\R) = \oplus_{j=-1}^{\infty} W_j$.
The wavelets have two vanishing moments which is optimal for piecewise affine orthogonal wavelets.

The scaling functions $\phi^i$ and wavelets $\psi^i$ in the above setting are supported on the interval $[-1,1]$. Their graphs are depicted in Fig.~\ref{fig:DGH1_plot}.

\begin{figure}[ht]
\subfloat{\includegraphics[width=0.3\linewidth]{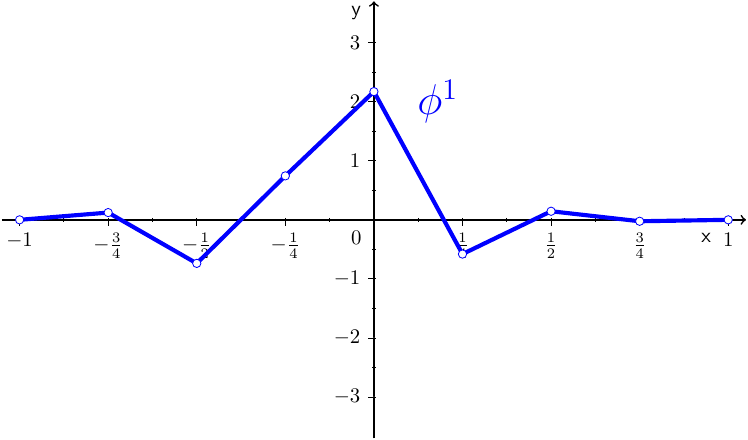}}\hfill
\subfloat{\includegraphics[width=0.3\linewidth]{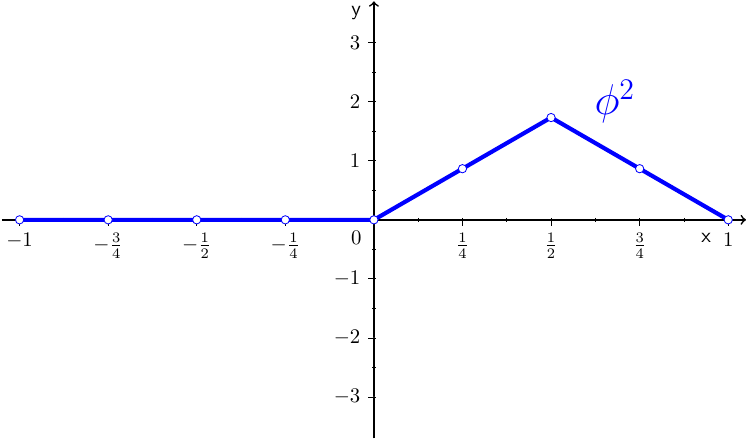}}\hfill
\subfloat{\includegraphics[width=0.3\linewidth]{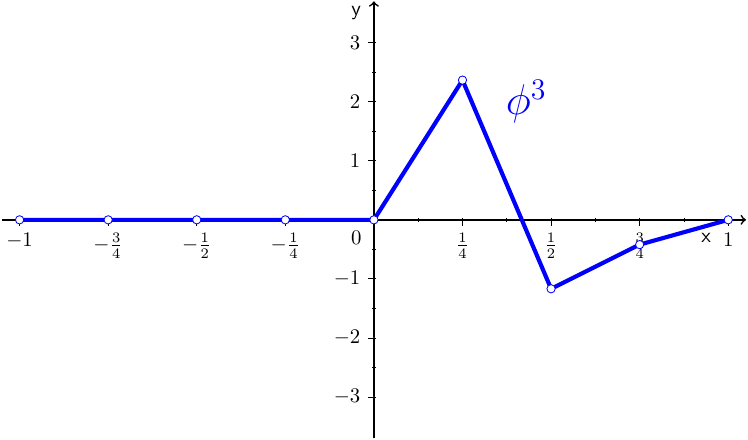}}

\subfloat{\includegraphics[width=0.3\linewidth]{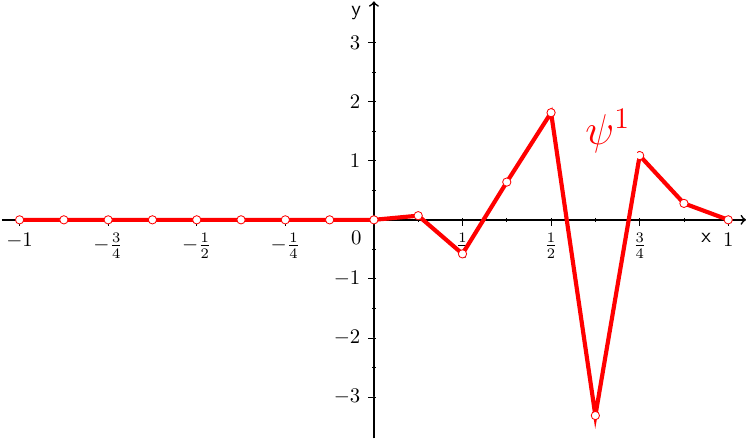}}\hfill
\subfloat{\includegraphics[width=0.3\linewidth]{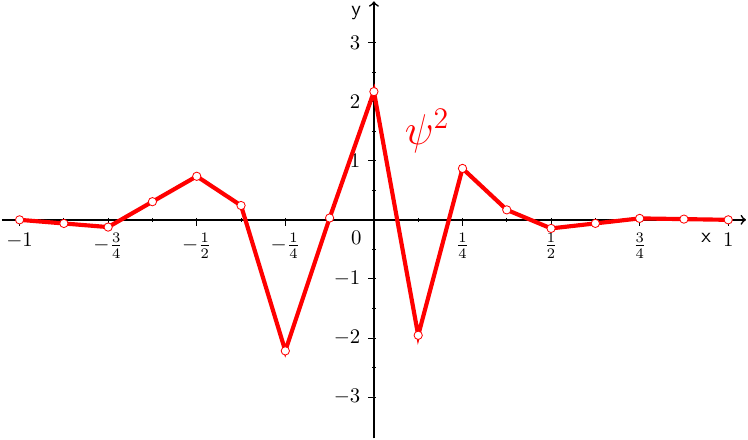}}\hfill
\subfloat{\includegraphics[width=0.3\linewidth]{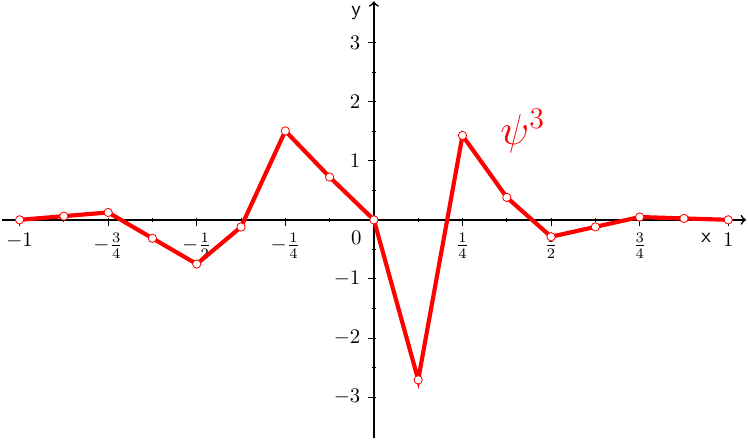}}
\caption{Piecewise affine orthogonal multiwavelets (\cite{DGH1996}, blue: scaling functions, red: wavelets)}
\label{fig:DGH1_plot}
\end{figure}

The above construction is geared to facilitate the construction of an orthonormal wavelet basis $\Psi=\left(\psi_\lambda\right)_{\lambda \in \Delta}$ for $L_2([0,1])$, essentially by restriction to the interval $[0,1] \subset \R$, see~\cite{DGH1999} for details.
The multilevel basis on $[0,1]$ consists essentially of all translated and dilated scaling functions and wavelets with support inside the interval complemented by appropriate boundary scaling functions and wavelets. These boundary functions are obtained as restrictions of linear combinations of the scaling functions and wavelets with support exceeding the interval boundaries
to $[0,1]$.
Moreover, for each end point of the interval there is exactly one boundary scaling function and one wavelet that vanishes at that end point, see Figures~\subref*{fig:DGH2_plot_boundaryN_S} and~\subref{fig:DGH2_plot_boundaryN_W}.
This is important for efficiently ``gluing'' local scaling functions and wavelets on $R\in\cR$ to globally conforming ones on $\Omega$. For those elements $R$
which intersect $\partial\Omega$ one has to incorporate Dirichlet boundary conditions. For the scaling functions this simply amounts
to omitting the single scaling function that does not vanish at the respective end point. The boundary wavelet needs to be modified somewhat
as shown in Fig.~\subref*{fig:DGH2_plot_boundaryD_W}.

With the analogous definition of multiresolution spaces $V_j$ and corresponding orthogonal complement spaces $W_j:=\overline{ \spn \{ \psi_\lambda: \lambda \in \Delta_j\} }$, this time on the
interval $[0,1]$,
we have the orthogonal decomposition $L_2([0,1]) = \bigoplus_{j=-1}^{\infty} W_j$. By a suitable affine change of variables
and adjusted scaling one obtains
orthogonal decompositions of $L_2([a,b])$ for any interval $[a,b] \subset \R$.

\vskip -.6cm
\begin{figure}[h]
\subfloat[\label{fig:DGH2_plot_boundaryN_S}]{\includegraphics[width=0.25\linewidth]{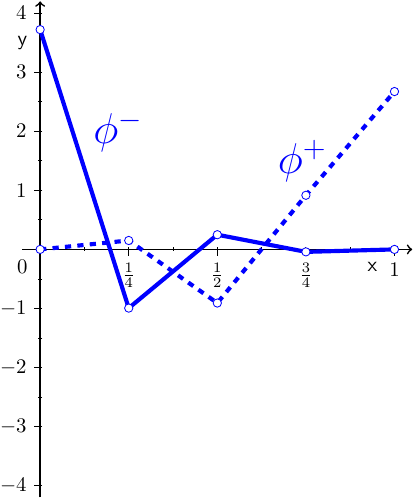}}\hfill
\subfloat[\label{fig:DGH2_plot_boundaryN_W}]{\includegraphics[width=0.25\linewidth]{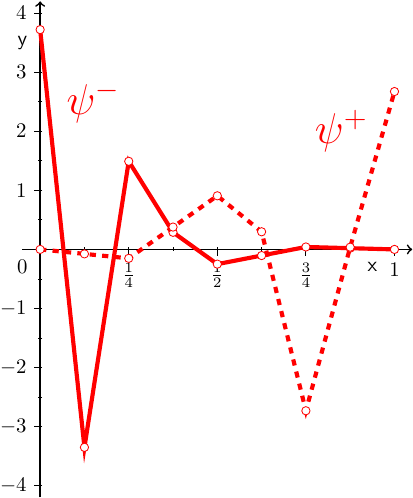}}\hfill
\subfloat[\label{fig:DGH2_plot_boundaryD_W}]{\includegraphics[width=0.25\linewidth]{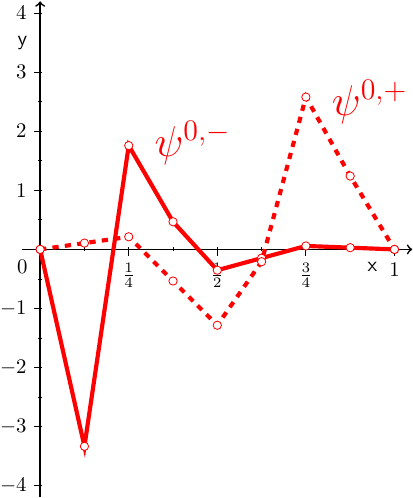}}
\caption{Boundary scaling functions and wavelets (\cite{DGH1999}); blue: scaling functions, red: wavelets; solid line: left boundary, dashed line: right boundary); (a) and (b): Neumann boundary conditions, (c): zero Dirichlet boundary conditions.}
\label{fig:DGH2_plot_boundary}
\end{figure}

\subsubsection{Constructing an auxiliary multilevel space}\label{sect:aux-multilevel}

We construct next an auxiliary space for $V_{h,D,\bp}$ that permits $H^1$-stable multilevel decompositions.
According to (a) it is based on local tensor product spaces. To define these local tensor product spaces consider
first an arbitrary interval $I=[a,b]\subset \R$ along with
a dyadic meshes $\cD_p={\bf NestedDyadic}[\cG_p,\{ I \},\alpha]$, which are nested when $p$ increases and graded.
Recall that $V_{h,D,p}(I)$ denotes the corresponding piecewise linear finite element space spanned by
the piecewise affine hat functions $(\theta_\zeta)_{\zeta\in \cD_p(I)}$, i.e., $\theta_\zeta(\zeta')=\delta_{\zeta,\zeta'}$, $\zeta, \zeta'\in\cD_p(I)$.

In view of \eqref{eq:intertwinedMRA},
$V_{h,D,p}(I)$ does not equal exactly a span of scaling functions from the spaces $V_j$.
Also the grid $\cD_p(I)$ is non-uniform. To obtain a possibly small space $W_{h,D,p}(I)$, spanned by wavelets in $\Psi$, that {\em contains} $V_{h,D,p}(I)$,
we consider first the set
\begin{equation}
\hat \Lambda := \left\{\lambda \in \Delta: \text{there is a} \ \zeta \in \cD_p(I)
\ \text{such that} \ (\psi_\lambda, \theta_\zeta)_{L_2(0,1)} \ne 0 \right\},
\end{equation}
that need to be activated in order to span $V_{h,D,p}(I)$.
Note that because of \eqref{eq:intertwinedMRA}
the set $\hat \Lambda$ is finite. Moreover, because all $\theta_\zeta$ are piecewise affine functions and the wavelets are compactly supported
with two vanishing moments, $\hat\Lambda$ ``adapts'' the non-uniform nature of the grid $\cD_p(I)$. In particular,
the maximal level occurring in $\hat \Lambda$ is finite, i.e.,
$ {|\lambda|} \le j_\text{max}$ {for all $\lambda \in \hat \Lambda$}, where $j_\text{max}=j_\text{max}(p,\alpha)$.

In order to facilitate an efficient transformation of linear combinations of the hat functions $(\theta_\zeta)_{\zeta\in \cD_p(I)}$
in $V_{h,D,p}(I)$
into a linear combination of wavelets one may possibly have to augment $\hat\Lambda$ somewhat
to a set $\Lambda$ which is the smallest ``multilevel tree'' that contains $\hat \Lambda$ and set
$\Lambda_j:=\{\lambda \in \Lambda: \abs{\lambda}=j\}$.
By this we mean that, when representing all wavelets with indices in $\Lambda$ in terms of scaling functions, the
encountered scaling functions on level $j$ all appear in the two-scale relation of the active scaling functions on level $j-1$.
The corresponding wavelet-transform then exhibits linear complexity which is important for the efficiency of the
preconditioner based on wavelet representation of the stiffness matrix.

Defining for $-1 \le j \le j_\text{max}$ the spaces $W_{h,D,p,j}(I):=\spn\{ \psi_\lambda : \lambda \in \Lambda_j \}$,
we obtain the $L_2$-orthogonal decomposition
\begin{equation*}
W_{h,D,p}(I):=\spn\{ \psi_\lambda : \lambda \in \Lambda \}= \bigoplus_{j=-1}^{\infty} W_{h,D,p,j}(I).
\end{equation*}
Note that since $\cD_p$ is graded and because of \eqref{eq:intertwinedMRA}, the grid underlying $W_{h,D,p}(I)$ is a local refinement of $\cD_p(I)$ by at most one level.
Furthermore, nestedness is preserved, i.e., $W_{h,D,p}(I) \subseteq W_{h,D,p+1}(I)$ for all $p$.

Given any $R = \bigtimes_{k=1}^d I_k\in \cR$, we define, according to step (a), the local
auxiliary space as
\begin{equation}
\label{def-tens-interp4a}
\tilde V(R) = W_{h,D,\bp}(R) := \bigotimes_{k=1}^d W_{h,D,p_k}(I_k),
\end{equation}
which leads to
\begin{equation}
\label{def-tens-interp5b}
\tilde V = W_{h,D,\bp}:= \{w\in H^1_0(\Omega): w\!\mid_{R}\in W_{h,D,\bp}(R),\, R\in \cR\},
\end{equation}
as the auxiliary space $\tilde V \supset V = V_{h,D,\bp}$ used in the ASM.
By construction, the dimension of $\tilde V$ remains uniformly proportional to ${\rm dim}\,V_{h,D,\bp}$.

For the construction of $H^1$-stable splittings for $\tilde V$ we wish to identify next an $L_2$-orthogonal wavelet basis
along with the corresponding scaling function basis for $\tilde V= W_{h,D,\bp}$ composed of the local wavelet bases for each $W_{h,D,\bp}(R)$. The basic principle is to properly
``glue'' those scaling functions and wavelets across an element interface which do not vanish on that interface.
Due to the properties of the interval-adapted scaling functions and wavelets, mentioned above, one can follow essentially the lines of
\cite{CTU1999,DS1999}. A small difference to be perhaps addressed is the fact
that the dyadic grids may vary from element to element and hence the restrictions of dyadic grids at an (inner) element interface
$F\subset \Omega$ do not agree.
However, since grids are nested the common minimal grid -- the intersection -- $\cD^*(F)$ belongs to all participating grids. Moreover, global continuity
of the functions in $V_{h,D,\bp}$ implies that the traces of the local functions from the adjacent elements agree on $F$ and belong to the
finite element space induced by $\cD^*(F)$. Let us denote by $W(\cD^*(F))$ the trace space obtained by the intersection of the spaces
$W_{h,D,\bp}(R)$ such that $F\in \cF(R)$.
The effect of the gluing process is now as follows. Any global scaling basis function or wavelet in $W_{h,D,\bp}$
which does not vanish on $F$, and hence has part of its support in all adjacent elements, has the following property: its restriction to
any of the adjacent elements, whenever being nontrivial, is a tensor product of a scaling function or wavelet belonging to the trace space
$W(\cD^*(F))$ and a boundary scaling function or wavelet in the complementary variables which do not vanish on $F$.
In short, the global basis, which we refer to as {\em composite basis},
consists of basis functions supported in a single element and on
basis functions that are glued across element interfaces with factors from the intersection spaces $W(\cD^*(F))$. Moreover, it is important to note
that the restriction of the global wavelets to any element $R\in\cR$ is still an $L_2$-orthogonal
basis for $L_2(R)$ where the basis functions affected by the restriction still have norms that are uniformly equivalent to one.
Finally, the (local and global) transformation from elements in $V_{h,D,\bp}$ into wavelet representation has linear complexity.
It remains to confirm that $\Psi_\Omega$ gives rise to uniformly $H^1$-stable splittings for $\tilde V$. On account of the
above restriction properties and the set-additivity of the $H^1$-norm it suffices to show that the local wavelet bases for
the tensor-product spaces $W_{h,D,\bp}(R)$ give rise to uniformly $H^1(R)$-stable splittings.

To that end, we can now apply the theory of tensorial subspace splittings in its anisotropic variant. In fact, suitably scaled versions
of the univariate wavelet bases on an interval are uniformly $L_2$- respectively $H^1$-stable.
Thus, by Theorem 1 in \cite{GO1995a} we know that on any $R = \bigtimes_{k=1}^d I_k \in \cR$
we have an $H^1$-stable splitting {of $W_{h,D,\bp}(R)$ into the components }
$W_{j_1,\ldots,j_d} := \bigotimes_{k=1}^d W_{h,D,p_k,j_k}(I_k)$
\begin{align}
\label{eq:scaling}
a_R(u,u) \ \simeq \
\inf
\sum_{\ j_1,\ldots,j_d \ge -1} (2^{2j_1}+\cdots + 2^{2j_d}) (u_{j_1\ldots,j_d}, u_{j_1\ldots,j_d})_R \;,
\end{align}
where the infimum is taken over all representations of $u$ as $u = \displaystyle{\sum_{j_1,\ldots,j_d \ge -1}} u_{j_1,\ldots,j_d}$
with $u_{j_1,\ldots,j_d} \in W_{j_1,\ldots,j_d}$.
Since the wavelets form a basis the representations in the splittings are actually unique therefore providing
an $H^1(R)$-stable splitting on each element $R\in \cR$.

We again apply the auxiliary space method.
In fact, in terms of the terminology of Section~\ref{sec:asm} we now have $\tilde V=\hat V$, $a=\tilde a=\hat a$ and $\tilde Q: V \rightarrow \tilde V$
can be taken as the canonical embedding of $V$ into $\tilde V$. Then \eqref{eq:asm_a_equiv_ahat} and the inequalities in \eqref{eq:asm_bound_Qtilde} and \eqref{eq:asm_approx_Qtilde} concerning $\tilde Q$ are obviously satisfied.

Since $V=V_{h,D,\bp}\subset \tilde V= W_{h,D,\bp}$ we could
look for an appropriate operator $Q$ which possesses a right inverse. This would allow us to dispense with an additional smoothing bilinear form in this stage.
For simplicity of exposition we instead apply again
Proposition~\ref{prop:asm1}, i.e. we define the form $b_3(\cdot,\cdot)$ in analogy to
\eqref{eq:def_b_tens2}-\eqref{def:b_new},
(i) replacing $V_{\bp}(R)$ by $V_{h,D,\bp}(R)$ and
(ii) substituting the underlying LGL grid by the dyadic grid and
introducing the corresponding intervals $S_{D,\ell}$, subcells $I_{D,k,\ell_k}$ and, in analogy to \eqref{aspect},
sets $\cT_{D,\bp}^{(i)}$ for $i \in \{0,1\}$ .

Moreover, also $Q$ can be defined in complete analogy to \eqref{def:operQ},
(i) replacing $V_{h,D,\bp}(R)$ in \eqref{eq:def.vstarz} by $W_{h,D,\bp}(R)$ with the same meaning of the degree vectors $\bp_z^*$, $z\in \cF_0(R)$;
(ii) employing hat function representations of the elements of $W_{h,D,\bp}(R)$;
(iii) replacing the operator $\cI^R_{\bp_z^*}$ in \eqref{eq:def.vstarR} by the dyadic interpolation operator $\cI^R_{h,D,\bp_z^*}$.
Since the elements of the composite space $W_{h,D,\bp}$ are continuous over $\Omega$
one shows in exactly the same way as in Proposition~\ref{prop:continuous} that $Q\tilde v \in H^1(\Omega)$ for any $\tilde v\in \tilde V=W_{h,D,\bp}$.

The following lemma asserts the uniform $H^1$-stability of $Q$ and the corresponding direct estimate, hence the validity of the ASM conditions.

\begin{lemma}
\label{lem:Q}
Let $Q:\tilde V \rightarrow V$ be the operator defined above.
Then one has
\begin{align}\label{eq:Q3bound}
\norm{Q \tilde v}_{H^1} \lesssim \norm{\tilde v}_{H^1} \quad \text{for all} \quad \tilde v \in \tilde V
\end{align}
and
\begin{align}\label{eq:Q3approx}
b(\tilde v - Q \tilde v, \tilde v - Q \tilde v) \lesssim \hat a(\tilde v, \tilde v) \quad \text{for all} \quad \tilde v \in \tilde V,
\end{align}
where the constants are independent of the discretization parameters $\delta=(\bH,\bp)$.
\end{lemma}

The proof is a consequence of various stability estimates provided in Section~\ref{sec:proof-II}
and is therefore deferred to that section (see Remarks~\ref{rem:Q} and~\ref{rem:Qapprox}).

To build a preconditioner $\mathbf{C}_{\mathbf{\tilde A}_2}$ for the conforming problem over $V_{h,D,\bp}$ let $\mathbf{S}_3$ and $\mathbf{B}_3$ denote the matrix representations of the operator $Q$ and the auxiliary bilinear form $b_3$, referred to in Lemma~\ref{lem:Q}, respectively.
In the iterative solver we need to efficiently apply the stiffness matrix with respect to a {properly scaled wavelet basis for $W_{h,D,\bp}$.
Specifically, the wavelets are normalized so that their $H^1$-norms are uniformly equivalent to one.
However, this matrix is not sparse and is therefore never assembled. Instead, we only assemble the stiffness matrix with respect to
the corresponding scaling function basis which, in turn, can be efficiently obtained from the standard stiffness matrix with respect to a (slightly refined)
hat-function basis. As usual, the application of the wavelet representation is then realized by first applying the fast wavelet transform
in cascadic fashion, then applying the sparse scaling function representation followed by the inverse wavelet transform.
The overall complexity remains uniformly proportional to the dimension of the original problem.
To be more precise}, let $\mathbf{T}$ be the matrix representing the inverse multiwavelet transform taking the scaled wavelet representation into a scaling function representation, see \eqref{eq:scaling}.
Note that {the restriction of $\mathbf{T}$ to an element $R\in\cR$} is a tensorial operator, that can be applied dimension-wise. Moreover,
the {transformation} matrix
$\mathbf{T}$ itself is not assembled but, as indicated above, applied in cascadic fashion which {overall exhibits the linear complexity of the fast wavelet transform}.
Let us denote by $\mathbf{\tilde A}_2^\Phi$ the stiffness matrix with respect to the local scaling function basis of a space $\hat W_{h,D,\bp}$, that is slightly larger than $W_{h,D,\bp}$.
Note that due to the grading of $\cD_\bp$, the dyadic grid underlying $\hat W_{h,D,\bp}$ is a local refinement of that of $W_{h,D,\bp}$ by finitely many levels.
Since the scaling functions are piecewise (multi-)affine functions on a uniform grid, where the nodal values are known, the sparse stiffness matrix $\mathbf{\tilde A}_2^\Phi$ can easily be {assembled}.
Then the stiffness matrix $\mathbf{\tilde A}_2^\Psi$ with respect to the corresponding wavelet basis of $W_{h,D,\bp}$
is given as $\mathbf{\tilde A}_2^\Psi=\mathbf{T}^T \mathbf{\tilde A}_2^\Phi \mathbf{T}$ and can be {applied} in the way described above at the expense of
$\cO({\rm dim}\,(V_{h,D,\bp}))$ operations.
Finally let ${\sf CG}$ stand for a fixed number of conjugate gradient iterations applied to $\mathbf{\tilde A}_2^\Psi$.

As before for $\mathbf{B}_2$, the matrix $\mathbf{B}_3$ of the smoothing operator is not inverted exactly, but we shall approximate the inverse of $\mathbf{B}_3$ along the same lines as the inverse of $\mathbf{B}_2$.

The above findings can now be summarized as follows.

\begin{theorem}
\label{thm:tildeA2-prec}
Adhering to the above notation the composite preconditioner
\begin{equation}
\label{eq:tA2-prec}
\mathbf{C}_{\mathbf{\tilde A}_2} := \mathbf{B}_3^{-1} + \mathbf{S}_3 \ {\sf CG}\big(\mathbf{T}^T \mathbf{\tilde A}_2^\Phi \mathbf{T}\big) \ \mathbf{S}_3^T
\end{equation}
satisfies $\kappa(\mathbf{C}_{\mathbf{\tilde A}_2} {\mathbf{\tilde A}_2})= \cO(1)$,
uniformly in the discretization parameters $\delta=(\bH,\bp)$.
\end{theorem}

Employing the approximate inversion of $\mathbf{B}_3$ in analogy to the treatment of $\mathbf{B}_2$ the application of $\mathbf{C}_{\mathbf{\tilde A}_2}$ requires $\cO({\rm dim}\,(V_{h,D,\bp}))$ operations.

\subsection[Numerical experiments concerning the preconditioner C A1]{Numerical experiments concerning the preconditioner $\mathbf{C}_{\mathbf{\tilde A}_1}$}\label{sec:mumres-II}

We discuss next efficient strategies for applying
$\mathbf{C}_{\mathbf{\tilde A}_1}= \mathbf{B}_2^{-1} + \mathbf{S}_2 \mathbf{C}_{\mathbf{\tilde A}_2} \mathbf{S}_2^T$
introduced in Theorem~\ref{th:2}.

\subsubsection[Approximate inversion of B2 with C A2 = inverse(A2)]{Approximate inversion of $\mathbf{B}_2$ with $\mathbf{C}_{\mathbf{\tilde A}_2} = \mathbf{\tilde A}_2^{-1}$}\label{subsec:ApproxInvB2}
The matrix $\mathbf{B}_2$ for the smoothing operator is very sparse but no longer diagonal.
In Fig.~\subref*{fig:SparsityB2singlepatch}
this is illustrated for $d=2$ and a patch with polynomial degree $p=25$ in both directions. For a detailed
discussion of exploiting the sparsity we refer to \cite{brix-thesis,BCD2013}.
\begin{figure}[ht]
\subfloat[before reordering]{\includegraphics[width=0.24\linewidth]{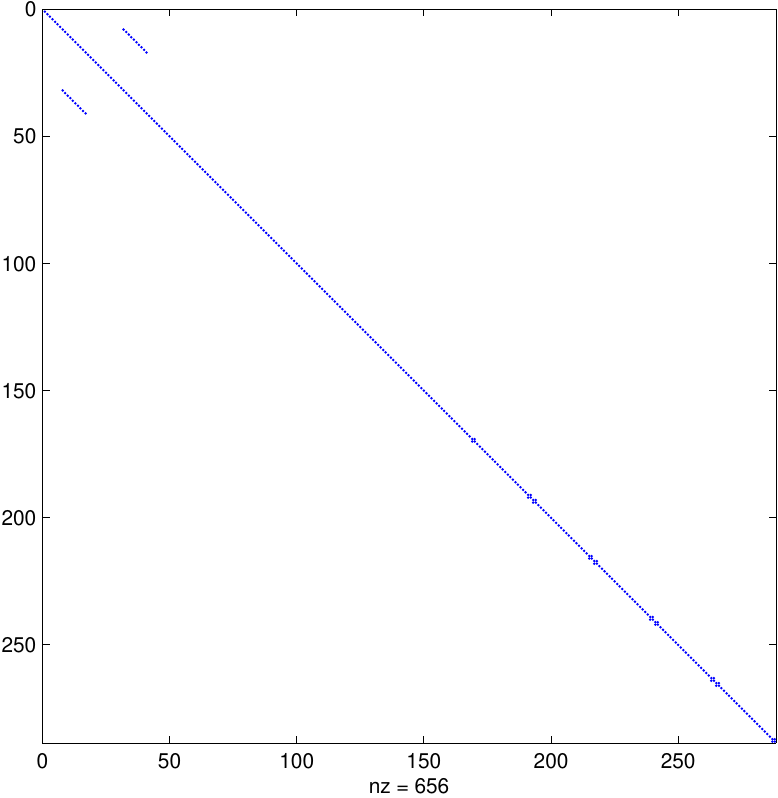}\label{fig:SparsityB2singlepatch}}\hfill
\subfloat[after reordering]{\includegraphics[width=0.24\linewidth]{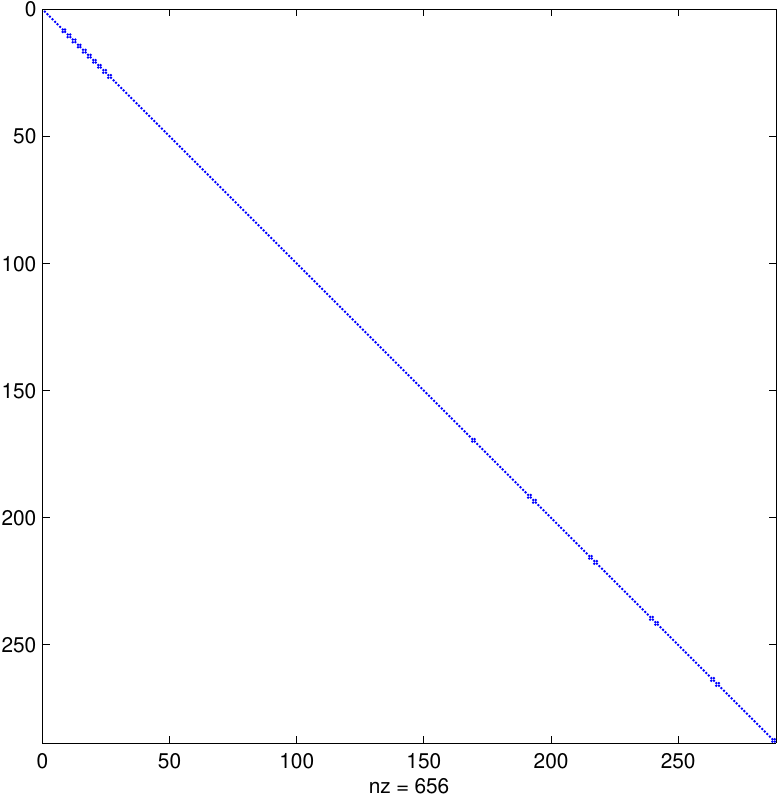}\label{fig:SparsityB2singlepatch_ordered}}\hfill
\subfloat[before reordering]{\includegraphics[width=0.24\linewidth]{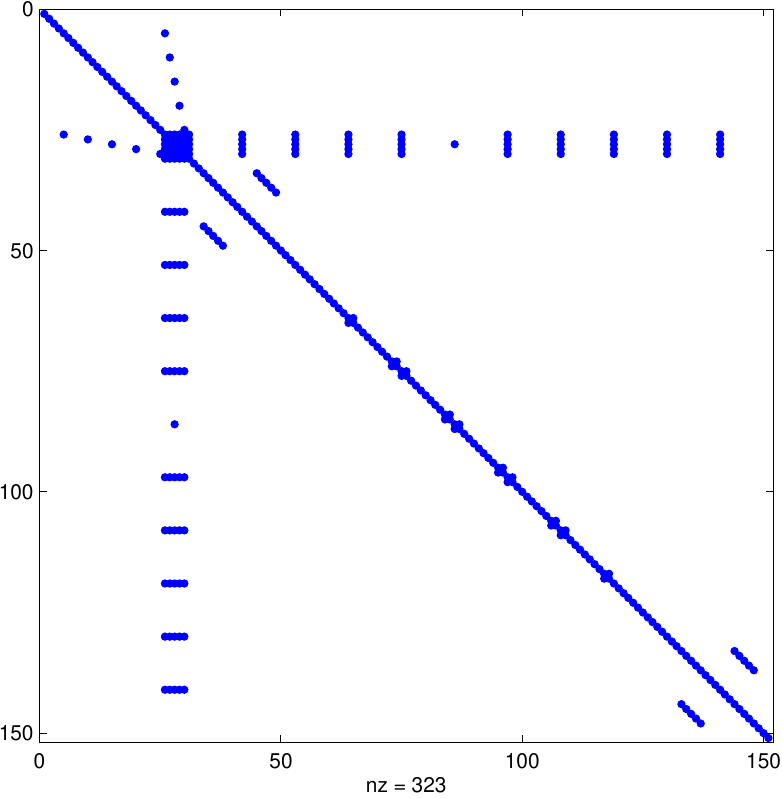}\label{fig:SparsityB2multipatch}}\hfill
\subfloat[after reordering]{\includegraphics[width=0.24\linewidth]{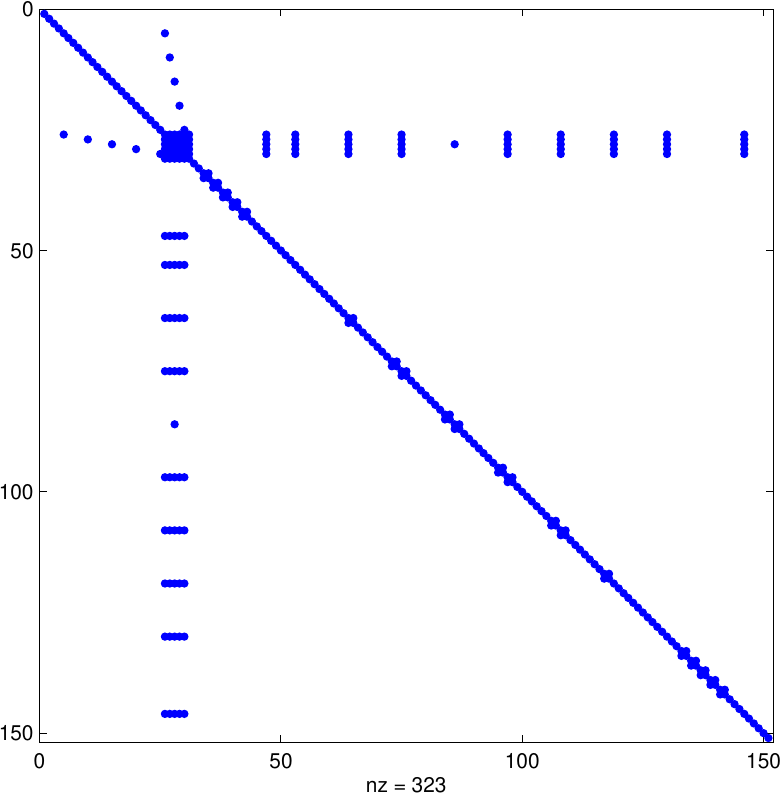}\label{fig:SparsityB2multipatch_ordered}}
\caption{Sparsity patterns for patch-inner nodal functions:
(a)-(b) top left quarter of the matrix $\mathbf{B}_2$ of $p=25$ on a single patch.
(c)-(d) Matrix $\mathbf{B}_2$ on two patches with polynomial degrees $6 \times 6$ and $12 \times 12$.}
\label{fig:SparsityB2}
\end{figure}

We compare below the exact inversion of $\mathbf{B}_2$ and an alternating Gauss-Seidel relaxation
over the skeleton of $\cR$ and a block elimination for each $R\in \cR$ in the spirit of substructuring methods.
Before we need to fix the parameter $\alpha$ in the dyadic grid generation, the aspect ratio control $C_{\rm aspect}$
in the definition of $b_2(\cdot,\cdot)$, monitoring the influence of the inverse estimates, and the tuning constant $c_{\xi, \xi'} \sim 1$.
$\alpha=1.2$ appears to be a good compromise
yielding sufficiently rich auxiliary spaces $V_{h,D,\bp}$ while keeping their dimension close to that of $V_\delta$.
In subsequent experiments we use $C_{\rm aspect}=2$ since larger values turn out to more oscillations in the estimates for
the condition numbers. Due to a somewhat stronger variation in the condition numbers, the choice of $c_{\xi, \xi'}$ is less clear.
Based on extensive experiments we set $c_{\xi, \xi'}=0.6$.

\begin{figure}[b!]
\centering
\subfloat[First test scenario\label{fig:resultsASM2a}]{\includegraphics[height=0.375\linewidth]{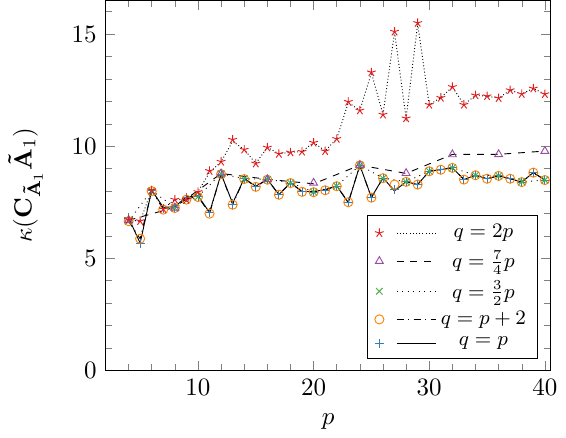}}\hspace{0.01\linewidth}
\subfloat[Second test scenario\label{fig:resultsASM2b}]{\includegraphics[height=0.375\linewidth]{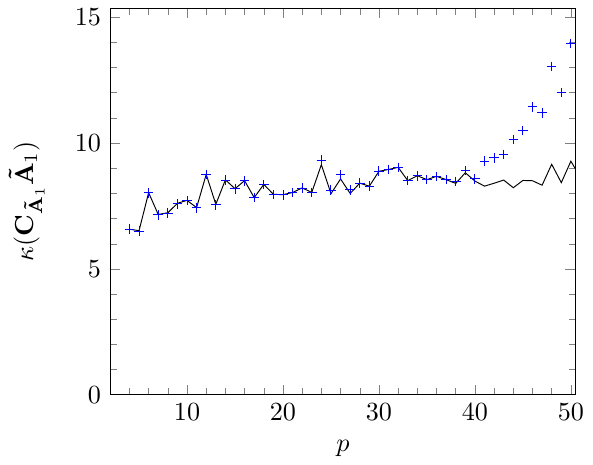}}
\caption{Condition numbers $\kappa({\bf C_{\tilde A_1}} {\bf \tilde A_1})$ obtained for ${\bf C_{\tilde A_1}}= \mathbf{B}_2^{-1} + \mathbf{S}_2 \mathbf{C}_{\mathbf{\tilde A}_2} \mathbf{S}_2^T$ with $\mathbf{C}_{\mathbf{\tilde A}_2} = \mathbf{\tilde A}_2^{-1}$.}
\label{fig:resultsASM2}
\end{figure}

Figures~\subref*{fig:resultsASM2a} and~\subref*{fig:resultsASM2b} show the condition numbers $\kappa({\bf C_{\tilde A_1}} {\bf \tilde A_1})$ for the first and second test scenario, respectively, when $\mathbf{\tilde A}_2$ is inverted exactly.
The plot marks indicate the condition numbers obtained by approximately
solving {the smoothing problem} with the aid of $7$ iterations of the substructuring method.
In contrast, in both plots the solid lines represent the results when the smoothing problem is solved exactly using a direct method.

For the first test scenario the condition numbers are presented for $p$ ranging from $4$ to $40$.
In the cases (i) $q=p$, (ii) $q=p+2$ and (iii) $q=3/2p$ the condition numbers vary only mildly, probably
due to the
discrete process of dyadic grid generation. In contrast, we observe larger condition numbers for larger jumps
between the polynomial degrees on adjacent elements , see the cases (iv) $q=7/4p$ and, in particular, (v) $q=2p$.
For a more thorough analysis and an explanation for these effects we refer
to~\cite{BCD2012}.

One observes enhanced oscillations when specific polynomial degrees are present in the grid, e.g. for odd polynomial degrees close to $27$ which we attribute to a resonance effect between the LGL and dyadic grids. More detailed investigations reveal that similar oscillations are also observed for some even polynomial degrees, e.g. $50$. Moreover, the resonance can be shifted to nearby polynomial degrees when the parameter $\alpha$, controlling the associated dyadic grids, is varied.
Therefore, our tests should be viewed as a guide for favorable choices of the dyadic grid generation parameters $\alpha$
depending on the employed polynomial degrees.

For the second test scenario, which is perhaps more relevant for practical applications than the first one,
we investigate $p$ in the range of $4$ to $50$.
When the smoothing problem is solved exactly, we observe that,
aside from some oscillations due to the grid resonance effect, the condition numbers grow in essence mildly for small polynomial degrees and quickly tend to a limit below $10$.
When solving the smoothing problem only approximately by a fixed number of substructuring iterations, we observe
increasing condition numbers for polynomial degrees larger than $41$.
This issue will be examined more closely in forthcoming work.

\subsubsection[The multi-wavelet preconditioner C A2]{The multi-wavelet preconditioner $\mathbf{C}_{\mathbf{\tilde A}_2}$}

In the more relevant second test scenario we compare the effect of an exact solution of the smoothing problem with a substructuring method for $\mathbf{B}_3$ performing alternating Gauss-Seidel relaxation over the skeleton of $\cR$ and a block elimination for each $R\in \cR$. We always apply $7$ iterations in the substructuring scheme and set $C_\textnormal{aspect}=2.0$ and $c_{\xi, \xi'}=0.6$.

Our first observation is that the condition number of the multi-wavelet stiffness matrix $\mathbf{\tilde A}_2^\Psi$ is approximately $\kappa(\mathbf{\tilde A}_2^\Psi) \approx 60$, independent of the underlying polynomial degree.
We observe that
after approximately $30$ iterations of the CG algorithm a desired absolute residual tolerance of $10^{-6}$ is reached.

Fig.~\ref{fig:resultsASM3} shows that the condition numbers produced by $\mathbf{C}_{\mathbf{\tilde A}_2}$ change abruptly when the polynomial degree varies. This seems to be again due to the discontinuous nature of dyadic grid generation.

Moreover, we see that for $4 \le p \le 50$ the condition numbers are below $12.5$ in all cases.
Finally, note that the substructuring method works better on $\mathbf{B}_3$ than on $\mathbf{B}_2$. Using the same number of iterations as before, there is no visible difference in the results between exact and approximate inversion of $\mathbf{B}_3$, which is probably due to the much more local coupling of the macro elements in $\mathbf{B}_3$.

\begin{figure}[t!]
\centering
\includegraphics[height=0.375\linewidth]{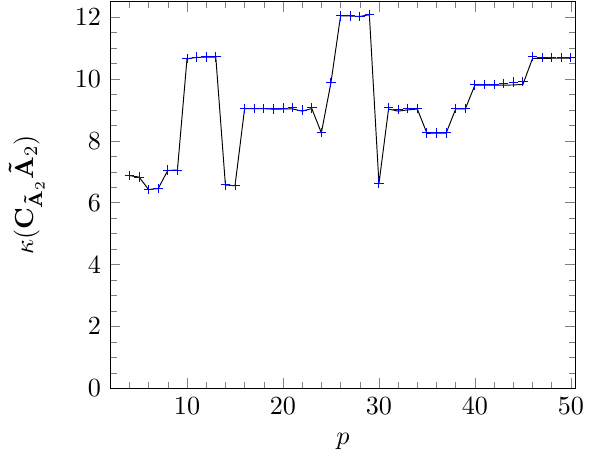}
\caption{Condition numbers $\kappa(\mathbf{C}_{\mathbf{\tilde A}_2} \mathbf{\tilde A}_2)$
produced by the multi-wavelet preconditioner $\mathbf{C}_{\mathbf{\tilde A}_2}$ in the second test scenario, where the smoothing problem is solved exactly (solid line) or approximately using the substructuring method (blue marks).
\label{fig:resultsASM3}}
\end{figure}

\section{The composite preconditioner}\label{sec:comp}

\begin{figure}[t!]
\centering
\subfloat[First test scenario\label{fig:resultsASM12a}]{\includegraphics[height=0.375\linewidth]{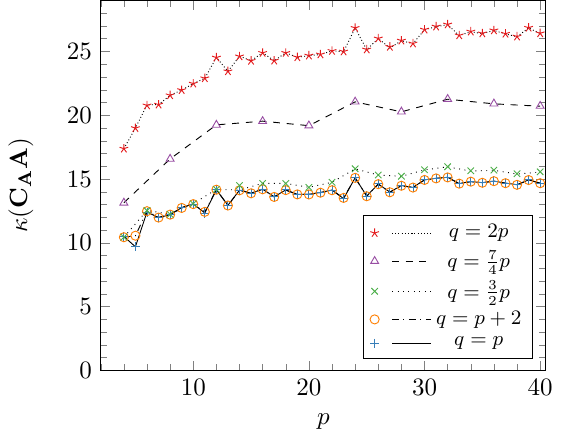}}\hspace{0.01\linewidth}
\subfloat[Second test scenario\label{fig:resultsASM12b}]{\includegraphics[height=0.375\linewidth]{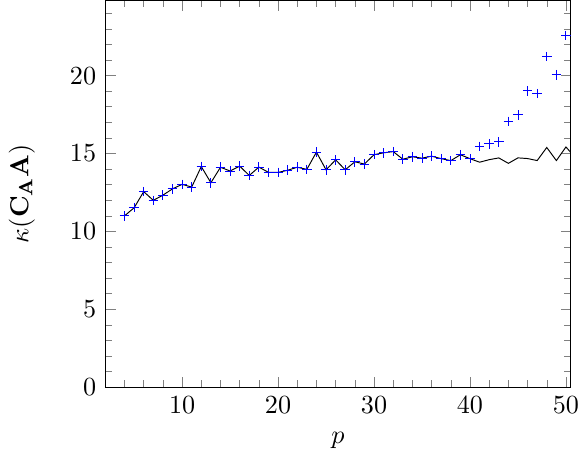}}
\caption{Condition numbers $\kappa(\mathbf{C}_\mathbf{A} \mathbf{A})$ obtained by \eqref{combined} with $\mathbf{C}_{\mathbf{\tilde A}_2} = \mathbf{\tilde A}_2^{-1}$.}
\label{fig:resultsASM12}
\end{figure}

We adhere to the previous notation concerning the matrices $\mathbf{C}_{\mathbf{\tilde A}_1}, \mathbf{\tilde A}_1, \mathbf{C}_{\mathbf{\tilde A}_2}, \mathbf{\tilde A}_2$, $\mathbf{B}_i$, and $\mathbf{S}_i$ for $i=1,2$. Combining Theorems~\ref{th:stageI}, \ref{th:2}, and~\ref{thm:tildeA2-prec}, we obtain the following result.
\begin{theorem}
\label{thm:combined}
Assume that the grading conditions \eqref{eq:setting.50} hold. Then the
composite preconditioner
\begin{equation}
\label{combined}
\mathbf{C}_\mathbf{A}:= \mathbf{B}_1^{-1} + \mathbf{S}_1 \big(\mathbf{B}_2^{-1} + \mathbf{S}_2 \mathbf{C}_{\mathbf{\tilde A}_2}\mathbf{S}_2^T\big) \mathbf{S}_1^T
\end{equation}
for the DG-system \eqref{eq:conf.1} satisfies $\kappa(\mathbf{C}_\mathbf{A} \mathbf{A})=\cO(1)$ uniformly in
the discretization parameters $\delta= (\bH,\bp)$.
\end{theorem}

Moreover, safe for the exact inversion of $\mathbf{B}_2$, each application requires
a number of operations that stays uniformly proportional to ${\rm dim}\,V_\delta$ with respect to the discretization parameters $\delta=(\bH,\bp)$.

The quantitative behavior of the combined preconditioner is illustrated by the subsequent experiments.
First, we set $\mathbf{C}_{\mathbf{\tilde A}_2}=\mathbf{\tilde A}_2^{-1}$ in \eqref{combined}, i.e., we invert $\mathbf{\tilde A}_2$ exactly.
The numerical results for both test scenarios are depicted in Fig.~\ref{fig:resultsASM12}.

For both scenarios, the numerical effects observed earlier for the single stages are still present.
The oscillations that have been striking in the second stage for the first test scenario now only have a mild effect, but their dependence on the ratio $q/p$ is still clearly visible.
In the second test scenario, for $p \le 40$ the condition numbers for the approximate and exact solutions of the smoothing problem almost agree.
For larger $p$ the insufficiently accurate solutions of the smoothing problem by the substructuring method in the second stage lead to a deterioration of the overall condition number. When $\mathbf{B}_2$ is inverted exactly, the condition number is bounded by $17$.

Fig.~\ref{fig:resultsASM123} shows the condition numbers $\kappa(\mathbf{C}_\mathbf{A} \mathbf{A})$ obtained
for the second test scenario when $\mathbf{C}_{\mathbf{\tilde A}_2}$ is the multi-wavelet preconditioner defined in \eqref{eq:tA2-prec};
the polynomial degree $p$ ranges from $4$ to $50$.
In Fig.~\subref*{fig:resultsASM123a} these condition numbers (solid line) are compared to those obtained for $\mathbf{C}_{\mathbf{\tilde A}_2}=\mathbf{\tilde A}_2^{-1}$ (dashed line), see also Fig.~\subref*{fig:resultsASM2b}.

When the multi-wavelet preconditioner is applied, we again observe that, besides the oscillations, the condition numbers essentially tend to a limit smaller than $18$.
In comparison with the condition numbers obtained for the variant where $\mathbf{\tilde A}_2$ is inverted exactly,
we observe that both curves are nearly parallel. The oscillations, which are due to the LGL/dyadic grid resonance,
in both curves appear for the same polynomial degrees, but their amplitudes are multiplied by a factor in the range from $2$ to $3$ in the variant using the multi-wavelet preconditioner.

The effect of the substructuring method is visualized as in Fig.~\ref{fig:resultsASM2}, namely the blue plot marks indicate the condition numbers obtained when ${\bf B_{2}}$ is approximated by $7$ iterations of the substructuring method while the solid line represents the condition numbers when the smoothing problem is solved exactly by a direct method.

The results obtained with the composite preconditioner are qualitatively comparable
to the ones obtained in Subsection~\ref{subsec:ApproxInvB2}, see also Fig.~\subref*{fig:resultsASM2b}. For $p \le 41 $ we observe that the substructuring method performs as well as exact solution of the smoothing problem.
In comparison with the variant where $\mathbf{\tilde A}_2$ is inverted exactly the condition numbers are larger by a small factor below 2.

\begin{figure}
\centering
\subfloat[Using smoothing operator $\mathbf{B}_3$.\label{fig:resultsASM123a}]{\includegraphics[height=0.375\linewidth]{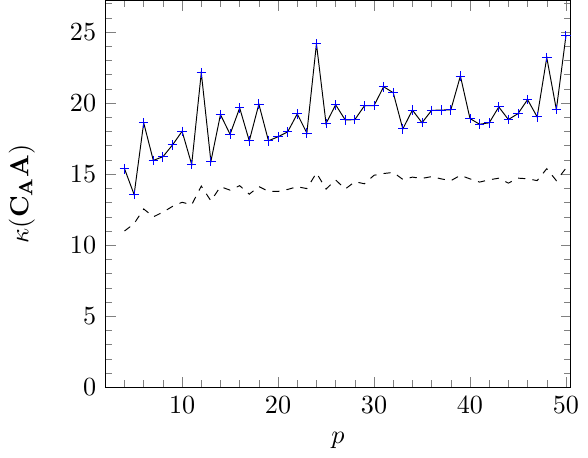}}\hspace{0.01\linewidth}
\subfloat[Without smoothing operator $\mathbf{B}_3$. \label{fig:resultsASM123b}]{\includegraphics[height=0.375\linewidth]{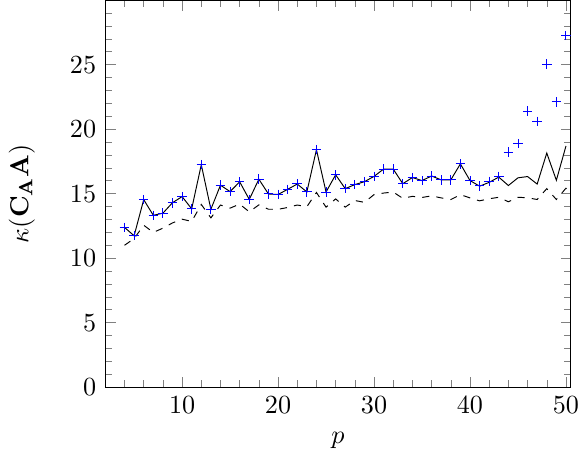}}
\caption{Condition numbers $\kappa(\mathbf{C}_\mathbf{A} \mathbf{A})$ obtained by the combined preconditioner \eqref{combined} in the second test scenario.
We display the condition numbers obtained using the multi-wavelet preconditioner for $\mathbf{C}_{\mathbf{\tilde A}_2}$ when the inverse of $\mathbf{B}_2$ (and $\mathbf{B}_3$ in Fig.~\protect\subref*{fig:resultsASM123a}) is either computed exactly (solid line) or is approximated using the substructuring method (blue crosses).
For a comparison we also give the results obtained using exact inverse of $\mathbf{\tilde A}_2$ (dashed line; see also Fig.~\protect\subref*{fig:resultsASM12b}).}
\label{fig:resultsASM123}
\end{figure}

Finally, we numerically investigate a variant of the concatenated preconditioner \eqref{combined}, where in
the definition of the multi-wavelet preconditioner \eqref{eq:tA2-prec} the inverse of the smoother $\mathbf{B}_3$ is omitted. This can be interpreted as a single auxiliary space method with the concatenated operators $Q_2 \circ Q_3$
and $\tilde Q_3 \circ \tilde Q_2$ and the auxiliary bilinear form $b_2$.

In the numerical results displayed in Fig.~\subref*{fig:resultsASM123b} we observe that the condition number is bounded independently of the discretization parameters $\delta=(\bH,\bp)$. Quantitatively, the condition numbers are somewhat lower than in the first variant, i.e., we can save computational cost and obtain a better preconditioner when the summand $\mathbf{B}_3^{-1}$ is omitted.
On the other hand, when an approximate inversion of $\mathbf{B}_3$ is used the need for a more accurate inversion of $\mathbf{B}_2$ seems to be reduced.

\section{Proof of Theorem~\ref{th:stageI}}
\label{sec:proof-I}

The proof of Theorem~\ref{th:stageI} follows from Corollary~\ref{cor:asm1} once we have verified the validity of the
ASM-conditions in Section~\ref{sec:asm}. Since here $\tilde V =V^c_\delta \subset V=V_\delta$ Proposition~\ref{prop:asm2}
applies, i.e., we have to verify {\bf ASM2} and find a suitable $\tilde Q : V\to \tilde V$, satisfying the direct estimate
\eqref{eq:asm6}.

To this end, we collect first some simple preparatory facts about product LGL grids and the structure of
corresponding quadrature weights. To organize this it is convenient to localize the facet complexes $\cF_l$.
Given two integers $0 \leq l,m \leq d$ and a facet $F \in \cF_l$ (see \eqref{facetcomplex}), we set
\begin{equation}
\label{eq:facets1}
\cF_m(F)=\{ G \in \cF_m \, : \, G \subset F \text{ if } m<l \;, \ G = F \text{ if } m=l \;, \
G \supset F \text{ if } m>l \}\;.
\end{equation}
This new definition consistently extends the ones given in Section~\ref{sect2.1}, where
in particular $\cF_l(R)$ has already been defined as the set of all $l$-facets of an element $R \in \cR$, whereas $\cR(F)=\cF_{d}(F)$
is the set of all elements $R$ containing the facet $F$.

Given any $R \in \cR$ and the corresponding LGL grid $\cG_\bp(R)$, if $F \in \cF_l(R)$ for some $1 \leq l \leq d-1$,
then $\cG_\bp(R) \cap F$ is the LGL grid of the induced quadrature formula on $F$, which will be denoted by $\cG_\bp(F,R)$.
The order of the formula, denoted by $\bp_*=\bp(F,R) \in \mathbb{N}^l$, is obtained from the vector $\bp=\bp(R)$
by deleting the $d-l$ components
corresponding to the frozen coordinates of $F$. The weights of the formula, denoted by $w^{F,R}_\xi$, are connected to the weights
$w_\xi$ of the original formula on $R$ by the relation
\begin{equation}\label{eq:facets2}
w_\xi=w_{f_1}\cdots w_{f_{d-l}} w^{F,R}_\xi \qquad \forall \xi \in \cG_\bp(F,R) \,,
\end{equation}
where each $w_{f_i}$ is a boundary weight of the univariate LGL formula along one of the frozen coordinates of $F$.
This implies that if the facet $F'$ belongs to $\cF_{l-1}(F) \subset \cF_{l-1}(R)$, then
\begin{equation}\label{eq:facets3}
w^{F',R}_\xi=w_{f}w^{F,R}_\xi \qquad \forall \xi \in \cG_\bp(F',R) \,,
\end{equation}
for some boundary weight $w_{f}$.

Next, let us draw some consequences from the assumptions \eqref{eq:setting.50}
of quasi-uniformity of the mesh and polynomial degree grading. Under these assumptions, it is easily seen that
given any $R \in \cR$, there exists a real ${\sf H}_R>0$ and an integer ${\sf p}_R>0$, such that for all $R' \in \cR$ satisfying
$R' \cap R \not = \emptyset$, it holds
\begin{equation}\label{eq:facets4}
H_k(R') \simeq {\sf H}_R \;, \qquad p_k(R') \simeq {\sf p}_R \;, \qquad \quad 1 \leq k \leq d \;,
\end{equation}
i.e., lengths, as well as polynomial degrees are locally comparable. More generally, for any
facet $F\in \cF_l$, $l<d$, one can define analogous ``representative parameters'' ${\sf H}_F, {\sf p}_F$, for instance,
by averaging corresponding parameters from intersecting elements. This implies
\begin{equation}\label{eq:facets4half}
{\sf H}_F \simeq {\sf H}_G \;, \qquad {\sf p}_F \simeq {\sf p}_G \;, \qquad \text{for any two faces }F, \ G \text{ such that }
F \cap G \not = \emptyset \;.
\end{equation}
Furthermore, for each boundary weight $w_j$ of any univariate LGL formula used in the definition of the tensorial LGL formula on $R'$, and for any face $F \in \cF_{d-1}(R')$, one has
\begin{equation}
\label{eq:facets5}
w_j \simeq \frac{{\sf H}_R}{\, {\sf p}_R^2} ,\quad j \in \{0,p_k(R')\}, \qquad \qquad \omega_F \simeq \frac{\, {\sf p}_F^2}{{\sf H}_F},
\quad F \in \cF_{d-1}(R').
\end{equation}
In particular, all the faces $F \in \cF_{d-1}$ having a nonempty intersection with an element $R$ carry weights
$\omega_F$ of comparable magnitude.

As a final prerequisite recall the fundamental property
that the bilinear form
\begin{equation}
\label{eq:basic.2}
({u},{v})_{0,{I},p} =
\sum_{j=0}^p {u}({\xi}_j) {v}({\xi}_j)\, {w}_j
\end{equation}
is an inner product in $\mathbb{P}_p(I)$, which induces a norm that is uniformly equivalent
to the $L^2$-norm, namely
$
\Vert v \Vert_{0,I} \leq \Vert v \Vert_{0,{I},p} \leq \sqrt{3} \Vert v \Vert_{0,I}$
for all $ v \in \mathbb{P}_p(I)$, see e.g. \cite[(5.3.2) on p. 280]{CaHuQuZa06}.
Tensorization yields that
\begin{equation}
\label{eq:basic.12}
({u},{v})_{0,{R},\bp} =
\sum_{\xi \in \cG_\bp(R)} {u}({\xi}) {v}({\xi}) \, {w}_{\xi}
\end{equation}
is a discrete inner product in $\mathbb{Q}_\bp(R)$, defining a norm which is uniformly equivalent to the $L^2$-norm, i.e.,
\begin{equation}
\label{eq:basic.13}
\Vert {v} \Vert_{0,R} \leq \Vert {v} \Vert_{0,{R},\bp} \leq \big(\sqrt{3}\big)^d
\Vert {v} \Vert_{0,R} \qquad \forall {v} \in \mathbb{Q}_\bp(R) \;.
\end{equation}
LGL grids as well as discrete inner products and norms restrict to any facet of $R$ in the obvious way,
yielding objects with analogous properties.

We shall now verify condition {\bf ASM2}, i.e., the form $b_\delt(\cdot,\cdot)$ dominates $a_\delta(\cdot,\cdot)$.
To this end, the following
immediate consequences of \eqref{eq:facets5}
concerning the weights $W_\xi$ introduced in \eqref{eq:ascs1} will be useful.

\begin{lemma}\label{lem:WF}
Let $R \in \cR$ be arbitrary. For any face $F \in \cF_{d-1}$ and any $\xi \in \cG_\bp(F,R)$, one has
\begin{equation*}
W_\xi \simeq \omega_F w^{F,R}_\xi \;.
\end{equation*}
\end{lemma}
\begin{proposition}
\label{prop:b_dominates_a}
For $b_\delt$ given in Definition~\ref{def:b_delta} one has
\begin{equation*}
a_\delta(v,v) \lsim b_\delt(v,v) \qquad \forall v \in V_\delta \;.
\end{equation*}
\end{proposition}
\begin{proof}
For any $R \in \cR$, using \eqref{eq:basic.2}, \eqref{eq:basic.13} and \eqref{eq:basic.9ter} together with tensorization, one obtains
\begin{equation*}
\Vert \nabla v \Vert_{0,R}^2 =
\sum_{k=1}^d \Vert \partial_k v \Vert_{0,R}^2 \lsim
\sum_{k=1}^d \sum_{\xi \in \cG_\bp(R)} v(\xi)^2 \, w_{\xi,k}^{-1}\Big(\prod_{j\neq k}w_{\xi,j}\Big)=
\sum_{\xi \in \cG_\bp(R)} v(\xi)^2 \, W_\xi \;,
\end{equation*}
whence
\begin{equation*}
\sum_{R \in \cR} \Vert \nabla v \Vert_{0,R}^2 \lsim b_\delt(v,v) \;.
\end{equation*}
On the other hand, let $F \in \cF_{d-1}$ be a face shared by two elements, say
$R^\pm$, and let $\bp_*^\pm=\bp(F,R^\pm)$.
We have
\begin{equation*}
\Vert \, [v] \, \Vert_{0,F}^2 \lsim \sum_\pm \Vert v^\pm \Vert_{0,F}^2
\lsim \sum_\pm \Vert v^\pm \Vert_{0,F,\bp_*^\pm}^2
= \sum_\pm \sum_{\xi \in \cG_{\bp}(F,R^\pm)} (v^\pm(\xi))^2 w^{F,R^\pm}_\xi \;.
\end{equation*}
Multiplying by $\omega_F$ and using Lemma~\ref{lem:WF} for both $R=R^\pm$, we obtain
\begin{equation*}
\omega_F \Vert \, [v] \, \Vert_{0,F}^2
\lsim \sum_\pm \sum_{\xi \in \cG_{\bp}(F,R^\pm)} (v^\pm(\xi))^2 \, W_\xi^\pm
\leq \sum_\pm \sum_{\xi \in \cG_{\bp}(R^\pm)} (v^\pm(\xi))^2 \, W_\xi^\pm \;.
\end{equation*}
A similar result holds for the faces $F\in \cF_{d-1}$ sitting on the boundary of $\Omega$. Hence,
\begin{equation*}
\sum_{F \in \cF} \omega_F \Vert \, [v] \, \Vert_{0,F}^2 \lsim b_\delt(v,v) \;,
\end{equation*}
which proves our claim.
\end{proof}

To verify the remaining ASM conditions we need to identify a suitable operator $\tilde Q : V_\delta \to V^c_\delta$.
First recall that, on account of Assumption \eqref{ass:locdeg} (which indeed poses a restriction
only for $d \geq 3$),
we are for each $F \in \cF_l$ with $1 \leq l \leq d$, entitled to select, once and for all, an element $R^\newsharp(F)$
such that
\begin{equation}\label{eq:facets7}
R^\newsharp(F):= {\rm argmin}_{R' \in \cR(F)}\bp(F,R'),\qquad \bp^\newsharp(F) := \bp(F,R^\newsharp(F)).
\end{equation}
The LGL grid $\cG_\bp(F,R^\newsharp(F))$ will be denoted by $\cG_{\bp^\newsharp}(F)$. Finally, for each vertex $z \in \cF_0$
we select an arbitrary element $R^\newsharp(z)$.

Given $v\in V_\delta$, the construction of $\tilde Q v =: \tilde v$ is based on the following simple idea.
Ascending from lower to higher dimensional facets, first at each vertex $z$
of $\cR$ we assign the value $\tilde v(z) := v_{R^\newsharp(z)}(z)$. Given $F\in \cF_l$ for some $l\geq 1$ and having fixed the values of
$\tilde v$ at the LGL nodes in the boundary $\partial_l F$ (viewing $F$ as an $l$-dimensional manifold),
we prescribe for $\tilde v$ at the LGL nodes in the relative interior of $F$
to be the corresponding values of $v_{R^\newsharp(F)}$. Finally, $\tilde v$ agrees with $v_R$ at the interior LGL nodes of $\cG_\bp(R)$.
Formally, this can be described by the following recursive procedure:

\begin{definition}
\label{def:Qtilde-scg}
For $0 \leq l \leq d$, set
$
F_l =\bigcup_{F \in \cF_l} F.
$
Given any $v \in V_\delta$, we define the sequence
of piecewise polynomial functions $\tilde{q}_l \, : \, F_l \to \mathbb{R}$ by the following recursion:
\begin{itemize}
\item[i)] For any $x \in \cF_0$, set $\tilde{q}_0(x)= 0$ if $x \in \partial\Omega$ and
$\tilde{q}_0(x)= v_{R^\newsharp(x)}(x)$ if $x \in \Omega$.

\item[ii)] For $l=1, \dots , d$, define $\tilde{q}_l$ on $F_l$ as follows: for any $F \in \cF_l$,
set $\tilde{q}_l\!\!\mid_{F}=0$ if $F \subset \partial\Omega$, otherwise define $\tilde{q}_l\!\mid_{F} \in \mathbb{Q}_{\bp^\newsharp}(F) $ by the conditions
\begin{equation}
\label{eq:Qtilde-scg.1}
\begin{split}
& \tilde{q}_l\!\mid_{F}(\xi)=v_{R^\newsharp(F)}(\xi) \qquad \forall \xi \in \cG_{\bp^\newsharp}(F) \setminus \partial_l F \;, \\
& \tilde{q}_l\!\mid_{F}(\xi)=\tilde{q}_{l-1}(\xi) \quad \qquad \forall \xi \in \cG_{\bp^\newsharp}(F) \cap \partial_l F .
\end{split}
\end{equation}
\end{itemize}
Finally, we set
\begin{equation}\label{eq:Qtilde-scg.2}
\tilde{Q} v = \tilde{q}_d \;.
\end{equation}
\end{definition}

\begin{remark}
\label{remQdelta}
The above recursion defines a linear operator $\tilde Q$ from $V_\delta$ into $\tilde V_\delta$.
\end{remark}
In fact, the linearity of $\tilde Q$ is obvious. Moreover,
note that $\tilde{Q} v \in V_\delta$ since $\bp^\newsharp(R) =\bp(R)$ for all $R \in \cR$. Furthermore, one has the property
\begin{equation}\label{eq:Qtilde-scg.3}
\tilde{q}_l\!\mid_{F_{l-1}} = \tilde{q}_{l-1} \qquad 1 \leq l \leq d \;.
\end{equation}
Indeed, if $F \in \cF_l$ and $F' \in \cF_{l-1}(F)$, then $\tilde{q}_l\!\mid_{F'} \in \mathbb{Q}_{\bp(F',R^\newsharp(F))}(F')$ whereas
\begin{equation*}
\tilde{q}_{l-1|F'} \in \mathbb{Q}_{\bp(F',R^\newsharp(F'))}(F').
\end{equation*}
The minimality property \eqref{eq:facets7} applied to $F'$,
together with the last set of conditions in \eqref{eq:Qtilde-scg.1}, imply
that $\tilde{q}_l\!\mid_{F'}=\tilde{q}_{l-1}\!\mid_{F'}$, whence \eqref{eq:Qtilde-scg.3} follows. This property implies that
$\tilde{Q} v$ is continuous throughout $\overline{\Omega}$, and vanishes on $\partial\Omega$.
We conclude that $\tilde{Q} v$ belongs to $\tilde{V}_\delta$.

We now turn to the verification of
\eqref{eq:asm6} in Proposition~\ref{prop:asm2}. This will be accomplished through a sequence of intermediate
results.

\begin{lemma}
\label{lem:boundb1}
For any $v \in V_\delta$, let its restriction to $R\in \cR$ be denoted by $v_R$.
The following bound holds
\begin{equation*}
b_\delt (v-\tilde{Q} v,v-\tilde{Q} v) \lsim \sum_{F \in \cF_{d-1}} \omega_F
\sum_{R \in \cR(F)} \Vert v_R - \tilde{q}_{d-1} \Vert_{0,F}^2, \quad v\in V_\delta ,
\end{equation*}
where $\tilde{q}_{d-1}$ is constructed in Definition~\ref{def:Qtilde-scg} in order to define $\tilde{Q} v$.
\end{lemma}
\begin{proof}
By definition of $b_\delt$, one has
\begin{equation*}
b_\delt (v-\tilde{Q} v,v-\tilde{Q} v) = \sum_{R \in \cR}
\sum_{\ \xi \in \cG_\bp(R)} |(v_R-\tilde{Q} v)(\xi)|^2 \, c_\xi W_\xi \;.
\end{equation*}
In view of \eqref{eq:Qtilde-scg.2}, $(\tilde{Q} v)\!\mid_{R}=\tilde{q}_d\!\mid_{R}$ and $\tilde{q}_{d}$ coincides with $v_R$ on $\cG_\bp(R)\setminus\partial R$
and with $\tilde{q}_{d-1}$ on $\cG_\bp(R)\cap\partial R$. Thus, since $c_\xi \simeq 1$
\begin{eqnarray*}
b_\delt (v-\tilde{Q} v,v-\tilde{Q} v) &\simeq&
\sum_{R \in \cR} \sum_{\ \xi \in \cG_\bp(R)\cap\partial R} |(v_R-\tilde{q}_{d-1})(\xi)|^2 \, W_\xi \\
&\leq & \sum_{R \in \cR} \sum_{\ F \in \cF_{d-1}(R)} \sum_{\ \xi \in \cG_\bp(F,R)} |(v_R-\tilde{q}_{d-1})(\xi)|^2 \, W_\xi \\
&\lsim& \sum_{F \in \cF_{d-1}} \omega_F \sum_{R \in \cR(F)}
 \sum_{\ \xi \in \cG_\bp(F,R)} |(v_R-\tilde{q}_{d-1})(\xi)|^2 \, w^{F,R}_\xi \;,
\end{eqnarray*}
where we have used Lemma~\ref{lem:WF} in the last inequality. Finally, we observe that $v_{R|F} \in \mathbb{Q}_{\bp(F,R)}(F)$,
whereas $\tilde{q}_{d-1|F} \in \mathbb{Q}_{\bp^\newsharp(F)}(F) \subseteq \mathbb{Q}_{\bp(F,R)}(F)$. Thus, \eqref{eq:basic.13} yields
\begin{equation*}
 \sum_{\ \xi \in \cG_\bp(F,R)} |(v_R-\tilde{q}_{d-1})(\xi)|^2 \, w^{F,R}_\xi \simeq \Vert v_R - \tilde{q}_{d-1} \Vert_{0,F}^2
\end{equation*}
proving the assertion.
\end{proof}

\begin{lemma}\label{lem:boundb2}
Let $v \in V_\delta$ be arbitrary, and let $\tilde{q}_l$, $0 \leq l \leq d$, be the sequence built in
Definition~\ref{def:Qtilde-scg} in order to define $\tilde{Q} v$. For $1 \leq l \leq d-1$, given any $F \in \cF_l$ and
any $R \in \cR(F)$, the following bound holds
\begin{equation}\label{eq:boundb1}
\Vert v_R - \tilde{q}_l \Vert_{0,F}^2 \lsim \sum_{G \in \cF_{d-1}(F)} \Vert \, [v]_G \, \Vert_{0,F}^2 \ \ +\ \
{\sf H}_F {\sf p}^{-2}_F \!\!\!\!\!\!\!\! \sum_{F' \in \cF_{l-1}(F)} \Vert v_{R^\newsharp(F)} - \tilde{q}_{l-1} \Vert_{0,F'}^2 \;,
\end{equation}
where $[v]_G$ denotes the jump of $v$ across the internal face $G$, or $v|_G$ when $G\subset \partial\Omega$.
On the other hand, for any vertex $x \in \cF_0$ and any $R \in \cR(x)$, one has
\begin{equation}\label{eq:boundb2}
|(v_R - \tilde{q}_0)(x)|^2 \lsim \sum_{G \in \cF_{d-1}(x)} |\, [v]_G(x) \, |^2 \;.
\end{equation}
\end{lemma}
\begin{proof}
Assume first $l>0$. If $F \subset \partial R\cap \partial\Omega$, then
\begin{equation*}
\Vert v_R - \tilde{q}_l \Vert_{0,F}^2 = \Vert v_R \Vert_{0,F}^2
\lsim \sum_{G \in \cF_{d-1}(F)} \Vert \, [v]_G \, \Vert_{0,F}^2\;,
\end{equation*}
which is a particular instance of \eqref{eq:boundb1}. Let us then assume $F \not \subset \partial \Omega$ and,
for simplicity, let us set $R^\newsharp=R^\newsharp(F)$. Observe that the definition of $\tilde{q}_{l|F}$ on
$\cG_{\bp^\newsharp}(F) \cap \partial_l F$ given in \eqref{eq:Qtilde-scg.1} can be rephrased as
$\tilde{q}_{l|F}(\xi)=v_{R^\newsharp}(\xi)+\big(\tilde{q}_{l-1}(\xi)-v_{R^\newsharp}(\xi)\big)$, or equivalently,
\begin{equation*}
\tilde{q}_{l|F}= v_{R^\newsharp|F}+\sum_{\xi \in \cG_{\bp^\newsharp}(F) \cap \partial_l F}\big(\tilde{q}_{l-1}(\xi)-v_{R^\newsharp}(\xi)\big)
\psi_\xi^F \;,
\end{equation*}
where $\psi_\xi^F \in \mathbb{Q}_{\bp^\newsharp}(F)$ denotes the Lagrangian function associated with the node $\xi$ of the LGL grid
on $F$ of order $\bp^\newsharp$. Thus, using once more \eqref{eq:basic.13}, we obtain
\begin{equation*}
\Vert v_R - \tilde{q}_l \Vert_{0,F}^2 \lsim \Vert v_R - v_{R^\newsharp} \Vert_{0,F}^2 +
\sum_{\xi \in \cG_{\bp^\newsharp}(F) \cap \partial_l F} |\tilde{q}_{l-1}(\xi)-v_{R^\newsharp}(\xi)|^2 w^{F,R^\newsharp}_\xi
=: A^2+B^2 \;.
\end{equation*}
As for the first quantity, there exists a sequence $R^i \in \cR(F)$, $i=0, \dots, m\leq d-l$
such that $R^0=R$, $R^m=R^\newsharp$ and $R^{i-1}\cap R^i \in \cF_{d-1}(F)$ for $i=1, \dots , m$. Thus,
$
A^2 \lsim \sum_{i=1}^m \Vert v_{R^i} - v_{R^{i-1}} \Vert_{0,F}^2 \leq \sum_{G \in \cF_{d-1}(F)} \Vert \, [v]_G \, \Vert_{0,F}^2 \;.
$
On the other hand, using \eqref{eq:facets3} and \eqref{eq:facets5}, we have
\begin{eqnarray*}
B^2 &\leq& \sum_{F' \in \cF_{l-1}(F)}
\sum_{\ \xi \in \cG_{\bp}(F',R^\newsharp)} |\tilde{q}_{l-1}(\xi)-v_{R^\newsharp}(\xi)|^2 w^{F,R^\newsharp}_\xi \\
&\lsim& {\sf H}_F {\sf p}^{-2}_F \!\!\!\!\!\!\!\! \sum_{F' \in \cF_{l-1}(F)}
\sum_{\ \xi \in \cG_{\bp}(F',R^\newsharp)} |\tilde{q}_{l-1}(\xi)-v_{R^\newsharp}(\xi)|^2 w^{F',R^\newsharp}_\xi
\ \ \simeq \ \ {\sf H}_F {\sf p}^{-2}_F \!\!\!\!\!\!\!\!
\sum_{F' \in \cF_{l-1}(F)} \Vert v_{R^\newsharp} - \tilde{q}_{l-1} \Vert_{0,F'}^2 \;,
\end{eqnarray*}
where the last equivalence follows, as in the proof of the previous lemma, by observing that
$v_{R^\newsharp|F'} \in \mathbb{Q}_{\bp(F',R^\newsharp)}(F')$
whereas $\tilde{q}_{l-1|F'} \in \mathbb{Q}_{\bp^\newsharp(F')}(F') \subseteq \mathbb{Q}_{\bp(F',R^\newsharp)}(F')$. Thus,
\eqref{eq:boundb1} is proven. Finally, \eqref{eq:boundb2} is trivial.
\end{proof}

\begin{lemma}\label{lem:boundb3}
Let $F \in \cF_l$ for some $0 \leq l \leq d-2$, and let $G \in \cF_{d-1}(F)$. Define
$\bp^\uparrow=\displaystyle{\max_{R \in \cR(G)}\bp(R)}$ (componentwise). Then,
\begin{equation*}
\left( {\sf H}_G {\sf p}^{-2}_G
\right)^{d-l-1}
 \Vert v \Vert_{0,F}^2 \lsim \Vert v \Vert_{0,G}^2 \qquad
\forall v \in \mathbb{Q}_{\bp^\uparrow}(G) \;.
\end{equation*}
\end{lemma}
\proof Using several times \eqref{eq:facets3} and \eqref{eq:facets5}, as well as \eqref{eq:basic.13}, we have
\begin{eqnarray*}
\Vert v \Vert_{0,G}^2 \ \simeq \sum_{\xi \in \cG_{\bp^\uparrow}(G)} |v(\xi)|^2 w^G_\xi
&\geq& \sum_{\xi \in \cG_{\bp^\uparrow}(G)\cap F} |v(\xi)|^2 w^G_\xi \\
&\simeq& \left( {\sf H}_G {\sf p}^{-2}_G
\right)^{d-l-1} \!\!\!\!\! \sum_{\xi \in \cG_{\bp^\uparrow}(F)} |v(\xi)|^2 w^F_\xi
\ \simeq \ \left( {\sf H}_G {\sf p}^{-2}_G
\right)^{d-l-1} \Vert v \Vert_{0,F}^2 \;. \qquad \endproof
\end{eqnarray*}

We are now ready to prove condition \eqref{eq:asm6}.
\begin{proposition}\label{prop:jackson_b}
For $b_\delt$ given in Definition~\ref{def:b_delta} and for $\tilde{Q}$ given in Definition~\ref{def:Qtilde-scg}, one has
\begin{equation*}
b_\delt (v-\tilde{Q} v,v-\tilde{Q} v) \lsim a_\delta(v,v) \qquad \forall v \in V_\delta \;.
\end{equation*}
\end{proposition}
\proof First, we observe that the cardinality of any set $\cF_m(F)$ defined in \eqref{eq:facets1} is bounded by a quantity
depending only on the dimension $d$. We start from the bound given by Lemma~\ref{lem:boundb1}, and focus on any face
$F \in \cF_{d-1}$. Using inequality \eqref{eq:boundb1} of Lemma~\ref{lem:boundb2} with $l=d-1$, we get
\begin{equation*}
\sum_{R \in \cR(F)} \Vert v_R - \tilde{q}_{d-1} \Vert_{0,F}^2 \ \lsim \
\Vert \, [v] \, \Vert_{0,F}^2 \ \ + \ \
{\sf H}_F {\sf p}^{-2}_F \!\!\!\!\!\!\!\!
\sum_{F' \in \cF_{d-2}(F)} \Vert v_{R^\newsharp(F)} - \tilde{q}_{d-2} \Vert_{0,F'}^2 \;.
\end{equation*}
A further application of Lemma~\ref{lem:boundb2} to each summand on the right hand side of the above inequality, now with $l=d-2$, yields taking into account \eqref{eq:facets4half}
\begin{eqnarray*}
\sum_{R \in \cR(F)} \Vert v_R - \tilde{q}_{d-1} \Vert_{0,F}^2 &\lsim&
\Vert \, [v] \, \Vert_{0,F}^2 \ \ + \ \
{\sf H}_F {\sf p}^{-2}_F \!\!\!\!\!\!\!\!
\sum_{F' \in \cF_{d-2}(F)} \sum_{\ G \in \cF_{d-1}(F)} \Vert \, [v]_G \, \Vert_{0,F'}^2 \\
&& +
\left( {\sf H}_F {\sf p}^{-2}_F
\right)^2 \!\!\!\!\!\!\!\! \sum_{F' \in \cF_{d-2}(F)}
\sum_{F'' \in \cF_{d-3}(F')} \Vert v_{R^\newsharp(F')} - \tilde{q}_{d-3} \Vert_{0,F''}^2 \;.
\end{eqnarray*}
Lemma~\ref{lem:boundb3} with $l=d-2$ and \eqref{eq:facets4half} yield
${\sf H}_F {\sf p}^{-2}_F
\Vert \, [v]_G \, \Vert_{0,F'}^2 \lsim \Vert \, [v] \, \Vert_{0,G}^2$.
At this point, let us introduce the set $\cF^{\, \cap}_{d-1}(F)=\{G \in \cF_{d-1} \, : \, G \cap F \not = \emptyset \}$ of all faces
intersecting the face $F$, and let us observe that
\begin{equation*}
{\sf H}_F {\sf p}^{-2}_F \!\!\!\!\!\!\!\!
\sum_{F' \in \cF_{d-2}(F)} \sum_{\ G \in \cF_{d-1}(F)} \Vert \, [v]_G \, \Vert_{0,F'}^2
 \lsim \sum_{\ G \in \cF^{\, \cap}_{d-1}(F)} \Vert \, [v] \, \Vert_{0,G}^2 \;.
\end{equation*}
On the other hand, we have
\begin{equation*}
\sum_{F' \in \cF_{d-2}(F)}
\sum_{F'' \in \cF_{d-3}(F')} \Vert v_{R^\newsharp(F')} - \tilde{q}_{d-3} \Vert_{0,F''}^2 \lsim
\sum_{F'' \in \cF_{d-3}(F)} \sum_{\ R \in \cR(F'')} \Vert v_{R} - \tilde{q}_{d-3} \Vert_{0,F''}^2 \;.
\end{equation*}
Thus,
\begin{equation*}
\sum_{R \in \cR(F)} \Vert v_R - \tilde{q}_{d-1} \Vert_{0,F}^2 \ \ \lsim
\sum_{\ G \in \cF^{\, \cap}_{d-1}(F)} \Vert \, [v] \, \Vert_{0,G}^2 \ \ +\ \
\left( {\sf H}_F {\sf p}^{-2}_F
\right)^2 \!\!\!\!\!\!\! \sum_{F'' \in \cF_{d-3}(F)} \sum_{\ R \in \cR(F'')} \Vert v_{R} - \tilde{q}_{d-3} \Vert_{0,F''}^2 \;.
\end{equation*}
We now proceed recursively, using Lemmas~\ref{lem:boundb2} and~\ref{lem:boundb3} with $l=d-3, d-4, \dots$. At the $j$-th
stage of recursion, we obtain
\begin{equation*}
\sum_{R \in \cR(F)} \Vert v_R - \tilde{q}_{d-1} \Vert_{0,F}^2 \lsim
\sum_{\ G \in \cF^{\, \cap}_{d-1}(F)} \Vert \, [v] \, \Vert_{0,G}^2 \ \ +\ \
\left( {\sf H}_F {\sf p}^{-2}_F
\right)^{j-1} \!\!\!\!\!\! \sum_{F' \in \cF_{d-j}(F)} \sum_{\ R \in \cR(F')} \Vert v_{R} - \tilde{q}_{d-j} \Vert_{0,F'}^2 \;.
\end{equation*}
When $j=d$, we use \eqref{eq:boundb2} and \eqref{eq:facets5} to finally get
\begin{equation*}
\sum_{R \in \cR(F)} \Vert v_R - \tilde{q}_{d-1} \Vert_{0,F}^2 \lsim
\sum_{\ G \in \cF^{\, \cap}_{d-1}(F)} \Vert \, [v] \, \Vert_{0,G}^2 \;.
\end{equation*}
At last, we observe that $\omega_G \simeq \omega_F$ for all $G \in \cF^{\, \cap}_{d-1}(F)$ by \eqref{eq:facets4half},
so that, going back to Lemma~\ref{lem:boundb1}, we conclude that
\begin{equation*}
b_\delt (v-\tilde{Q} v,v-\tilde{Q} v) \lsim \sum_{G \in \cF_{d-1}} \omega_G \Vert \, [v] \, \Vert_{0,G}^2
\ \leq \ a_\delta(v,v) \;. \qquad \qquad \endproof
\end{equation*}

\section{Proof of Theorem~\ref{th:2}}
\label{sec:proof-II}

The proof of Theorem~\ref{th:2} follows again from Corollary~\ref{cor:asm1} once we have verified the ASM conditions
for the respective ingredients given in Section~\ref{sec:conf}. Since now $\tilde V=V_{h,D,\bp}\not\subset V=V_\delta^c$
we need to address {\bf ASM2-3} in full while {\bf ASM1} is trivial since all spaces are conforming
and the standard energy bilinear form can be used.

We verify first {\bf ASM2}.

\begin{proposition}\label{prop:asimeqb}
One has
$\,a(v,v) \lsim b_2(v,v)$ for all $ v \in V$.
\end{proposition}
\proof
We first observe that $a(v,v) \simeq a(v_h,v_h)$ for all $v \in V_\delta^c$ due to Property~\ref{prop:Nhequivalence-multi}.
With the notation of Section~\ref{sect: auxbilforms}, the idea is to bound the terms
\begin{align*}
a_{R,k,S_\bell}(v_h,v_h):=
\int_{S_{\bell}} \abs{\partial_{x_{k}} v_h}^2 \, dx
=\int_{S_{\bell,k}'} \left( \int_{I_{k,\ell_k}} \abs{\partial_{x_{k}} v_h(x_k,x')}^2 \, dx_k \right) \, dx'
\end{align*}
by using quadrature in all but the $k$-th variable while keeping the integral
with respect to the $k$-th variable for $S_\bell\in \cT_{\bp,k}^{(0)}(R)$, the ``anisotropic sub-cells'', and applying an inverse estimate for $S_\bell\in \cT_{\bp,k}^{(1)}(R)$,
the ``isotropic sub-cells''.
Specifically, using the tensorized trapezoidal rule for integration over $S_{\bell,k}'$, which is the finite-element lumped mass matrix approximation~\cite[(4.4.44) on p. 220]{CaHuQuZa06}, yields the terms $b^{(0)}_{R,k,S_{b\ell}}(v ,v )$ in \eqref{eq:def_b_RkSl_0}.
Hence we still have $a_{R,k,S_{\bell}}(v_h,v_h)\lsim b^{(0)}_{R,k,S_{\bell}}(v ,v )$ for all $S_\bell\in \cT_{\bp,k}^{(0)}(R)$.
The same local relations hold for the sub-cells in $\cT_{\bp,k}^{(1)}(R)$ since we have used an inverse estimate as in the original
version of $b_\delt(\cdot,\cdot)$. \endproof

The remaining ASM conditions involve the operators $Q$, defined by \eqref{def:operQ}, and $\tilde Q$
which is yet to be defined and is only needed for the analysis.

The definition of the operator $\tilde Q$ follows the same lines as the one of $Q$. Given any $v \in V$ and any $R \in \cR$,
we set $v_R = v_{|R} \in \mathbb{Q}_\bp(R)$. Then, the chain \eqref{eq:def.vstarz}-\eqref{eq:def.vstarR} is replaced
by the following one:
\begin{equation}\label{eq:def.vz}
{v}_z^* := \cI^R_{\bp_z^*} \left( \Phi_z {v}_R \right) \in \mathbb{Q}_{\bp_z^*}(R)
\qquad \text{and} \qquad
\tilde{v}_z^* := \cI^R_{h,D,\bp_z^*}( \cI^R_{h,\bp_z^*} \, {v}_z^*) \in V_{h,D,\bp^*_z}(R) \;.
\end{equation}
Summing over the vertices of $R$, we define
\begin{equation}\label{eq:def.vR}
{v}_R^*:= \sum_{z \in \cF_0(R)} {v}_z^* \in \mathbb{Q}_{\bp}(R)
\qquad \text{and} \qquad
\tilde{Q}_R {v}_R := \tilde{v}_R^* = \sum_{z \in \cF_0(R)} \tilde{v}_z^* \in V_{h,D,\bp}(R) \;.
\end{equation}
As above, one easily confirms interelement continuity, which suggests defining
the operator $\tilde{Q} : {V} \to \tilde{V}$ by
\begin{equation}
\label{def:opertildeQ}
(\tilde{Q} {v})_{|R}:=\tilde{Q}_R {v}_R = \tilde{v}_R^* \qquad \quad \forall R \in \cR \; \qquad
\forall {v} \in {V} \;,
\end{equation}
where $\tilde{v}_R^*$ is defined in \eqref{eq:def.vR}.

To complete the proof of Theorem~\ref{th:2} it remains to verify {\bf ASM3} which, in turn,
requires establishing the following two results.

\begin{proposition}\label{propos:forQ.1}
The operators $Q$ and $\tilde Q$ are linear and satisfy the continuity assumption \eqref{eq:asm_bound_Qtilde}.
\end{proposition}

\begin{proposition}\label{propos:forQ.3}
The operator $Q$ and $\tilde Q$ satisfy the Jackson conditions \eqref{eq:asm_approx_Qtilde}, i.e., one has
\begin{equation*}
b_2(\tilde v - Q \tilde v,\tilde v - Q \tilde v) \lsim |\tilde v|_{1,\Omega}^2 \quad \forall \, \tilde v \in \tilde V \;,
\qquad b_2( v - \tilde Q v, v - \tilde Q v) \lsim | v|_{1,\Omega}^2 \qquad \forall\, v \in V \;,
\end{equation*}
where the multiplicative constant in the above estimate depends on the constant $C_\textnormal{aspect}$ in \eqref{aspect}.
\end{proposition}

The remainder of this section is devoted to the proofs of Propositions~\ref{propos:forQ.1} and~\ref{propos:forQ.3}
which require several further technical prerequisites.

\subsection{The role of Property~\ref{prop:Nhequivalence-multi} and related facts}\label{ssec:stab}
The following stability estimates draw in essential way on Property~\ref{prop:Nhequivalence-multi}
and its univariate counterpart. A major issue is to interrelate grids of different orders.

We begin with
a relevant property that concerns the locally uniform equivalence of LGL grids of comparable order and
refer to \cite{BCD2013} for the proof.
\begin{property}
\label{prop:lgl-uniequiv}
Assume that $c\, p \leq q \leq p$ for some fixed constant $c>0$. Let $I=[a,b]$. Then, $\cG_q(I)$ and $\cG_p(I)$ are
locally $(A,B)$-uniformly equivalent, with $A$ and $B$ depending on the proportionality factor $c$ but independent
of $q$, $p$ and $H=b-a$.
\endproof
\end{property}

As a consequence of Properties~\ref{prop:lgl-uniequiv} and~\ref{lemdyadic-2} one obtains the following immediate extension to
the associated dyadic grids.
\begin{corollary}\label{cor:dyadic-uniequiv}
Assume that $c\, p \leq q \leq p$ for some fixed constant $c>0$. Then, $\cD_q(I)$ is locally $(A,B)$-uniformly equivalent
to both $\cG_p(I)$ and $\cD_p(I)$, with $A$ and $B$ depending on the proportionality factor $c$ but not on $p$ and
$H$. \endproof
\end{corollary}

Next, recall that Property~\ref{prop:Nhequivalence-multi} follows from its univariate counterpart.
\begin{property}
\label{prop:Nhequivalence-uni}
One has
\begin{equation}\label{eq:basic.8}
\Vert v \Vert_{0,I} \simeq \Vert v_h \Vert_{0,I},
\qquad \Vert v' \Vert_{0,I}
\simeq \Vert v_h' \Vert_{0,I}
\qquad \forall v \in \mathbb{P}_p(I)\;,
\end{equation}
where the involved constants are independent of $p$ and $H$.
\endproof
\end{property}

Note that taking as $v$ in \eqref{eq:basic.8} each Lagrange basis function
at the LGL nodes and using \eqref{eq:basic.1bis} and \eqref{grid}, it is easily seen that
the size of each interval $I_j$ is comparable to that of the LGL weight associated
with any of its endpoints, in the sense that the following bounds hold, uniformly in $p$ and $H$:
\begin{equation}
\label{eq:basic.7}
1 \lsim \min_{1\leq j \leq p} \frac{{h}_{j}}{{w}_{j}} \ \leq \
\max_{1\leq j \leq p} \frac{{h}_{j}}{{w}_{j}} \lsim 1 \;.
\end{equation}
As a consequence, the second relation in \eqref{eq:basic.8} together with \eqref{eq:basic.7} and \eqref{eq:basic.1bis} provide a simple proof
of the inverse inequality \eqref{eq:basic.9ter}. Indeed, one has
\begin{equation*}
\Vert v_h' \Vert_{0,I}^2 = \sum_{j=1}^p \left( \frac{{v}({\xi}_j)-{v}({\xi}_{j-1})}{{h}_j} \right)^2 {h}_j
\leq \frac2{{h}_0} {v}^2({\xi}_0)
 + \sum_{j=1}^{p-1} \left( \frac2{{h}_j} + \frac2{{h}_{j+1}} \right) {v}^2({\xi}_j) +
\frac2{{h}_{p}} {v}^2({\xi}_p) \;.
\end{equation*}

We address next the continuity of various univariate interpolation operators in certain Sobolev norms or seminorms.
The first result is classical.

\begin{lemma}
\label{lemma:forQ.1bis}
Let $\cG$ be any ordered grid in $I$, which defines a partition $\cT=\cT(\cG)$,
and let ${\cI}_{\cG} : H^1(I) \to V_h(\cT)$ be the associated piecewise linear interpolation operator. Then,
\begin{equation*}
| {\cI}_{\cG} v |_{1,I} \leq | v |_{1,I} \qquad \forall v \in H^1(I) \;.
\end{equation*}
\end{lemma}

\begin{lemma}\label{lemma:interp-unigrids}
Let $\cG$ and $\tilde{\cG}$ be ordered grids in $I$, with associated partitions $\cT=\cT(\cG)$ and
$\tilde{\cT}=\cT(\tilde{\cG})$. Assume that $\cG$ and $\tilde{\cG}$ are locally $(A,B)$-uniformly equivalent.
If ${\cI}_{\cG} : H^1(I) \to V_h(\cT)$ is the piecewise linear interpolation operator associated with $\cG$, one has
\begin{equation*}
\Vert {\cI}_{\cG} v \Vert_{0,I} \lsim \Vert v \Vert_{0,I} \qquad \forall v \in V_h(\tilde{\cT}) \;,
\end{equation*}
where the constant in the inequality depends only on the parameters $A$ and $B$.
\end{lemma}
\begin{proof}
Let ${\cG}= \{\xi_j \, : \, 0 \leq j \leq p \}$ and $\cT = \{I_j \, : \, 1 \leq j \leq p \}$,
with $h_j=\abs{I_j}=\xi_j - \xi_{j-1}$. For $0 \leq j \leq p$, let us define $h(\xi_j)={h}_j + {h}_{j+1}$, where we
set $h_0=h_{p+1}=0$.
Similarly, let $\tilde{\cG}= \{\eta_i \, : \, 0 \leq i \leq q \}$ and $\tilde{\cT} \{\tilde{I}_i \, : \, 1 \leq i \leq q \}$, with $\tilde{h}_i=\abs{\tilde{I}_i}=\eta_i - \eta_{i-1}$, and let
$\tilde{h}(\eta_i)$ be defined in a manner similar to $h(\xi_j)$.
Given any $v \in V_h(\tilde{\cT})$, we have
\begin{equation*}
\Vert {\cI}_{\cG} v \Vert_{0,I}^2 = \sum_{j=1}^p \Vert {\cI}_h v \Vert_{0,I_j}^2
\simeq \sum_{j=1}^p \left(v^2(\xi_j)+v^2(\xi_{j-1})\right) h_j = \sum_{j=0}^p v^2(\xi_j) h(\xi_j) \;,
\end{equation*}
and
\begin{equation*}
\Vert v \Vert_{0,I}^2 = \sum_{i=1}^q \Vert v \Vert_{0,\tilde{I}_i}^2
\simeq \sum_{i=1}^q \left(v^2(\eta_i)+v^2(\eta_{i-1})\right) \tilde{h}_i = \sum_{i=0}^q v^2(\eta_i) \tilde{h}(\eta_i) \;.
\end{equation*}
Now, for any $\xi_j$ there exist $\eta_i$ and $\theta \in [0,1)$ such that
$v(\xi_j)=(1-\theta) v(\eta_i) + \theta v(\eta_{i+1})$, whence $v^2(\xi_j) \leq v^2(\eta_i) + v^2(\eta_{i+1})$.
If $\theta \in (0,1)$, then $\xi_j \in (\eta_i, \eta_{i+1})$, hence both $I_j$ and $I_{j+1}$ intersect
$\tilde{I}_i$. By the assumption of locally uniform equivalence of the two grids, we obtain $\abs{I_j} \lsim \abs{\tilde{I}_i}$
and $\abs{I_{j+1}} \lsim \abs{\tilde{I}_i}$, whence $h(\xi_j) \lsim \tilde{h}_i \leq \min
\left(\tilde{h}(\eta_i),\tilde{h}(\eta_{i+1})\right)$. On the other hand, if $\theta=0$, i.e., $\xi_j=\eta_i$, then
$I_j$ intersects $\tilde{I}_i$ and $I_{j+1}$ intersects $\tilde{I}_{i+1}$ (with the obvious adaptation if $\xi_j$ is a
boundary point), thus $\abs{I_j} \lsim \abs{\tilde{I}_i}$ and $\abs{I_{j+1}} \lsim \abs{\tilde{I}_i}$,
which yields $h(\xi_j) \lsim \tilde{h}(\eta_i)$. This completes the proof. \end{proof}

To proceed recall the definitions of the interpolation operators $\cI_p$, $\cI_{h,p}$, $\cI_{h,D,p}$,
in \eqref{def-tens-interp1}, \eqref{def-tens-interp3}, and \eqref{eq:basic.66}, respectively.

\begin{lemma}\label{lemma:forQ.1ter}
For any $p>0$, the operator $\cI_p$ satisfies
\begin{equation*}
| \cI_p v |_{1,I} \lsim | v |_{1,I} \;, \qquad \forall v \in H^1(I) \;,
\end{equation*}
with a constant that does not depend â on $p$.
\end{lemma}
\begin{proof}
The result is classical (see, e.g.,\cite{CaHuQuZa06}). It can be derived from Property~\ref{prop:Nhequivalence-uni}
and Lemma~\ref{lemma:forQ.1bis} observing that $\cI_p v = \cI_p(\cI_{h,p}v)$. \end{proof}

\begin{property}\label{prop:unif-cont-1D}
{\rm (\cite[Remark 13.5]{BeMa97})} Assume that $c\, p \leq q \leq p$ for some fixed constant $c>0$. Then
\begin{equation*}
\Vert \cI_q v \Vert_{0,I} \lsim \Vert v \Vert_{0,I} \quad \forall v \in \mathbb{P}_p(I) \;,
\end{equation*}
with a constant depending on the proportionality factor $c$ but not on $p$.
\end{property}

\begin{lemma}
\label{lemma:forQ.3}
Assume that $c\, p \leq q \leq p$ for some fixed constant $c>0$. Then,
\begin{equation*}
| \cI_{h,q} v |_{m,I} \lsim | v |_{m,I} \quad \forall v \in \mathbb{P}_p(I)\;, \qquad m=0,1 \;,
\end{equation*}
with a constant depending on the proportionality factor $c$ but not on $p$.
\end{lemma}
\begin{proof}
For $m=0$, we observe that $\cI_{h,q} v = \cI_{h,q}(\cI_q v)$, so that
$ \Vert \cI_{h,q} v \Vert_{0,I} \lsim \Vert \cI_q v \Vert_{0,I} \lsim \Vert v \Vert_{0,I}$
by Properties~\ref{prop:Nhequivalence-uni} and~\ref{prop:unif-cont-1D}.
The result for $m=1$ is included in Lemma~\ref{lemma:forQ.1bis}.
\end{proof}

\begin{lemma}\label{lemma:forQ.2}
Assume that $c\, p \leq q \leq p$ for some fixed constant $c>0$. Then, for $m=0,1$ one has
\begin{eqnarray*}
| \cI_q v |_{m,I} \ \ \simeq \ \ | \cI_{h,q} v |_{m,I} &\lsim & |v|_{m,I} \quad \forall v \in V_{h,D,p}(I) \;, \\[5pt]
| \cI_{h,D,q} v |_{m,I} & \lsim & |v|_{m,I} \quad \forall v \in V_{h,D,p}(I) \;, \\[5pt]
| \cI_{h,D,q} v |_{m,I} & \lsim & | v |_{m,I} \quad \forall v \in V_{h,p}(I)\;.
\end{eqnarray*}
\end{lemma}
\begin{proof}
The results for $m=0$ follow from Lemma~\ref{lemma:interp-unigrids}
applied in various combinations to the grids $\cG_q(I)$, $\cG_p(I)$, $\cD_q(I)$ and $\cD_p(I)$, that are
locally uniformly equivalent to each other by Corollary~\ref{cor:dyadic-uniequiv}. The results for $m=1$ follow again from
Lemma~\ref{lemma:forQ.1bis}.
\end{proof}

We turn now to the multivariate case. Using Lemmas~\ref{lemma:forQ.1bis} and~\ref{lemma:interp-unigrids},
Proposition~\ref{propos:tensorprop} yields the following general result.

\begin{property}\label{prop:interp-unigrids-multi}
For $1 \leq k \leq d$, let $\cG_k$ and $\tilde{\cG}_k$ be ordered grids in $I_k$,
with associated partitions $\cT_k$ and $\tilde{\cT}_k$, which are locally $(A,B)$-uniformly equivalent;
let ${\cI}_{\cG_k}= {\cI}_{\cG_k}^{I_k}: H^1(I_k) \to V_h(\cT_k)$ be the piecewise linear interpolation operator associated
with $\cT_k$. Consider the spaces $V_h(\cT):=\bigotimes_{k=1}^d V_h(\cT_k)$ and
$V_h(\tilde{\cT}):=\bigotimes_{k=1}^d V_h(\tilde{\cT}_k)$ of piecewise multi-linear functions on the Cartesian
partitions $\cT:=\bigtimes_{k=1}^d \cT_k$ and $\tilde{\cT}:=\bigtimes_{k=1}^d \tilde{\cT}_k$ of $R$. Then, the
piecewise multilinear interpolation operator $\cI_{\cG} = \cI_{\cG}^R:= \bigotimes_{k=1}^d {\cI}_{\cG_k}^{I_k} : C^0(R) \to V_h(\cT)$
satisfies
\begin{equation*}
\Vert {\cI}_{\cG} v \Vert_{m,R} \lsim \Vert v \Vert_{m,R} \qquad \forall v \in V_h(\tilde{\cT})\;, \qquad m=0,1 \;,
\end{equation*}
where the constant in the inequality depends only on the parameters $A$ and $B$. \endproof
\end{property}

From Lemma~\ref{lemma:forQ.2} and Proposition~\ref{propos:tensorprop}, we immediately get the following multidimensional
result.
\begin{property}
\label{property:forQ.1}
Assume that $c\, \bp \leq \bq \leq \bp$ for some fixed constant $c>0$. Then,
\begin{equation*}
| \cI_\bq v |_{m,R} \lsim | v |_{m,R} \quad \forall v \in V_{h,D,\bp}(R)\;,
\qquad m=0, 1 \;,
\end{equation*}
with a constant depending on the proportionality factor $c$ but not on $\bp$. \endproof
\end{property}

We are now prepared to complete the\\

\noindent
{\em Proof of Proposition~\ref{propos:forQ.1}:}
We treat only the operator $Q$. The argument for $\tilde Q$ is analogous.
We have to prove that $|Q \tilde{v}|_{1,\Omega} \lsim |\tilde{v} |_{1,\Omega}$ for all $\tilde{v} \in \tilde{V}$.
This follows if, for any $R \in \cR$, we prove that
\begin{equation}\label{eq:forQ.51bis}
| v_R^* |_{1,R} \lsim |\tilde{v}_R |_{1,R} \qquad \forall \tilde{v}_R \in V_{h,D,\bp}(R) \;.
\end{equation}
A classical mapping-and-scaling argument in Finite Element analysis tells us that this result holds provided it holds when $R$
is the reference element $\hat{R}=\bigtimes_{k=1}^d \hat I$, with $\hat I = [-1,1]$. For this element, it is enough to prove that
\begin{equation}\label{eq:forQ.50}
| v^* |_{1,\hat R} \lsim \Vert \tilde{v} \Vert_{1,\hat R} \qquad \forall \tilde{v} \in V_{h,D,\bp}(\hat R)\;.
\end{equation}
Indeed, changing $\tilde v$ into $\tilde v + \lambda$, by any $\lambda \in \mathbb{R}$, does not change the left-hand side,
whence
$
| v^* |_{1,\hat R} \lsim \inf_{\lambda \in \mathbb{R}} \Vert \tilde v + \lambda \Vert_{1,\hat R} \lsim | \tilde v |_{1,\hat R} \;.
$
In order to establish \eqref{eq:forQ.50}, let us first consider a univariate function $\tilde v \in V_{h,D,p}(\hat I)$ and
let $\Phi$ denote the affine function taking the value $1$ at one endpoint of the interval and $0$ at the other one.
Let us prove that if $cp \leq q \leq p $ for some fixed constant $c>0$, one has
\begin{equation}\label{eq:forQ.51}
| \cI_{h,D,q}(\Phi \tilde v) |_{m,\hat I} \lsim \Vert \tilde v \Vert_{m,\hat I}\;, \qquad m=0,1 \;,
\end{equation}
where, of course, the involved constant depends on the proportionality factor $c$.
For $m=0$, we have
\begin{equation*}
\Vert \cI_{h,D,q}(\Phi \tilde v) \Vert_{0,\hat I}^2 \lsim \sum_{\zeta \in \cG_{D,q}(\hat I)}
(\Phi(\zeta) \tilde v(\zeta))^2 h_{D,q}(\zeta)
\lsim \sum_{\eta \in \cG_{D,p}(\hat I)} \tilde v(\eta)^2 h_{D,p}(\eta) \lsim \Vert \tilde v \Vert_{0,\hat I}^2 \;,
\end{equation*}
since $\Phi^2 \leq 1$ and $h_{D,q}(\zeta) \lsim h_{D,p}(\zeta)$ for all $\zeta \in \cD_{q}(\hat I)
\subseteq \cD_{p}(\hat I)$. For $m=1$, we have by Lemma~\ref{lemma:forQ.1bis},
$
| \cI_{h,D,q}(\Phi \tilde v) |_{1,\hat I} \lsim | \Phi \tilde v |_{1,\hat I} \lsim \Vert \tilde v \Vert_{1,\hat I},
$
where the last inequality holds since we are working on the reference element.
Hence, \eqref{eq:forQ.51} is proven. Using this result and Proposition~\ref{propos:tensorprop}, we obtain the bound
$
| \tilde{v}_z^* |_{1,\hat R} \lsim \Vert \tilde v \Vert_{1,\hat R}
$
for each function $\tilde{v}_z^*$ defined as in \eqref{eq:def.vstarz} on $\hat R$. Then, Property~\ref{property:forQ.1}
yields
$|{v}_z^* |_{1,\hat R} \lsim \Vert \tilde v \Vert_{ 1,\hat R}$,
and \eqref{eq:forQ.50} follows by the triangle inequality, since $|\cF_0(\hat R)| \simeq 1$.
\endproof

\begin{remark}\label{rem:Q}
Consider the operator $Q: \tilde V= W_{h,D,\bp} \to V= V_{h,D,\bp}$ introduced in Section~\ref{sect:aux-multilevel}, whose $H^1$-stability
is claimed in \eqref{eq:Q3bound} in Lemma~\ref{lem:Q}. Since this operator is defined in complete analogy to \eqref{def:operQ}, the proof of its stability is
similar to that of Proposition~\ref{propos:forQ.1} given above. Indeed,
the grids underlying the wavelets spanning the spaces $W_{h,D,p_k}(I_k)$ are
locally (A,B)-uniformly equivalent to the dyadic grids $\cD_{p_k}(I_k)$.
Thus the result is a consequence of Property~\ref{prop:interp-unigrids-multi}.
\end{remark}

\subsection{A localized Jackson estimate for the interpolation error}
To prove Proposition~\ref{propos:forQ.3} requires the following different types of estimate which are not covered by the results of the
preceding section.
Since these results will be applied to both the LGL and the dyadic tensorized grids, with various choices
of finite-dimensional function spaces (comprised of either piecewise multi-linear or global polynomial functions), we
first establish the key estimates in suitable generality in order to specialize them later to the cases at hand.

Consider again the general piecewise multilinear interpolation operator $\cI_{\cG} = \cI_{\cG}^R : C^0(R) \to V_h(\cT)$ introduced in
the statement of Property~\ref{prop:interp-unigrids-multi} above. In addition, assume that each grid $\cG_k$ is locally quasiuniform according
to \eqref{grid}. For $k=1, \dots, d$, let $W_k$ be a finite dimensional
subspace of $H^1(I_k)$ to be specified later. Then, the univariate piecewise linear interpolation operator $\cI_{\cG_k}=\cI_{\cG_k}^{I_k}$ is well-defined on $W_k$ and we have
\begin{equation}\label{stabL2interpWk}
\Vert \cI_{\cG_k} v \Vert_{0,I_k} \leq \bar{c}_k \Vert v \Vert_{0,I_k} \qquad \forall v \in W_k \;,
\end{equation}
for some constant $\bar{c}_k>0$ independent of the size $|I_k|$ but possibly depending on the dimension of $W_k$ (although this will not be the case
in all our applications). Let us set $W= \bigotimes_{k=1}^d W_k$.

Next, consider the cells $S_\bell$ forming the partition $ \cT= \bigtimes_{k=1}^d \cT_k$, i.e.,
$\cT = \{ S_\bell = \bigtimes_{k=1}^d I_{k, \ell_k} \ : \ I_{k,\ell_k} \in \cT_k \} \:,$
and let $h_\bell:=\max_k | I_{k, \ell_k}|$ be the largest one-dimensional size of the cell $S_\bell$.
Let $h=\sum_\bell h_\bell \chi_{S_\bell}$ be the meshsize function defined in $R$.

The following {\em localized} Jackson estimate will be used several times in the sequel.

\begin{proposition} \label{prop:hest} The following estimate holds
\begin{equation}\label{eq:hest}
 \Vert h^{-1}(v - \cI_{\cG} v) \Vert_{0,R} \lsim |v |_{1,R} \qquad \forall v \in W \;,
\end{equation}
where the constant implied by the inequality is independent of the meshsize but depends on the constants $\bar{c}_k$ introduced in \eqref{stabL2interpWk}.
\end{proposition}

\proof The result will be obtained by assembling local estimates in each cell $S_\bell$, which in turn are derived by a scaling argument from corresponding estimates on
the unit box $B^d=[0,1]^d$, with $B=[0,1]$. To this end, let
$\cI_{B^k}$ denote the multilinear interpolation operator on $B^k$, i.e.,
\begin{equation*}
\cI_{B^k}v = \sum_{\xi\in \cF_0(B^k)} v(\xi)\Phi_\xi \;,
\end{equation*}
where $\Phi_\xi$ denotes the multilinear Lagrange basis function satisfying
$\Phi_\xi(\xi')=\delta_{\xi,\xi'}$ for $\xi,\xi'\in \cF_0(B^k)$.
We make heavy use of the fact that $\cI_{B^d}$ is a tensor product operator, i.e., we have $\cI_{B^d}= \otimes^d \cI_B$.
Moreover, it will be convenient to employ the following convention to write, for any $k=1,\ldots, d$,
\begin{equation*}
B^d= A^k \times B^{d-k}\;,
\end{equation*}
meaning that $A^k$ is the unit $k$-cube representing the first $k$ variables and $B^{d-k}$
is the unit $(d-k)$-cube for the coordinates $d-k+1,\ldots,d$, .
\begin{lemma}
\label{lem:B}
Let $\cW=\bigotimes_{k=1}^d \cW_k$, where $\cW_k$ are finite-dimensional subspaces of $H^1(B)$.
Then, one has
\begin{equation}
\label{capH1}
\| v- \cI_{B^d}v\|^2_{0,B^d}\lsim \sum_{k=1}^d\sum_{\, \xi\in \cG(A^{k-1})}\|\partial_{x_k}v(\xi,\cdot)\|^2_{0,B^{d-k+1}} \;, \qquad \forall v \in \cW \;,
\end{equation}
(with the first summand on the right-hand side to be understood as $\|\partial_{x_1}v \|^2_{0,B^{d}} $), where the constant depends only on $d$.
\end{lemma}
\proof
Denoting by $\id$ the identity operator on $B$, and writing
\begin{align*}
v - \cI_{B^d}v &= \sum_{k=1}^d \left((\cI_{B^{k-1}}\otimes {\id}^{d-k+1}) v - (\cI_{B^k}\otimes{\id}^{d-k})v \right)\\
&= \sum_{k=1}^d\big({\id}^{k-1}\otimes(\id - \cI_B)\otimes {\id}^{d-k} \big)(\cI_{B^{k-1}}\otimes {\id}^{d-k+1})v,
\end{align*}
we note that, when abbreviating $w_k:= (\cI_{B^{k-1}}\otimes {\id}^{d-k+1}v)$, 
for $x'_k:= (x_1,,\ldots,x_{k-1}, x_{k+1},\ldots,x_d)$ one has
\begin{equation*}
({\id}^{k-1}\otimes(\id - \cI_B)\otimes {\id}^{d-k} )w_k(x_1,\ldots,x_{k-1},\xi_k, x_{k+1},\ldots,x_d) =0,\quad \xi_k=0,1, \quad
x'_k \in A^{k-1}\times B^{d-k}.
\end{equation*}
Hence, we can apply the classical univariate inequality $\|v- \cI_Bv \|_{0,B}\leq 2 |v|_{1,B}$, for all $v\in H^1(B)$, to get
\begin{equation*}
\big\| \big({\id}^{k-1}\otimes(\id - \cI_B)\otimes {\id}^{d-k} \big)w_k\big\|_{0,B^d}\leq 2 \|\partial_{x_k} w_k\|_{0,B^d}.
\end{equation*}
For $k \geq 2$, notice that $w_k(\xi,\cdot)= v(\xi,\cdot)$, $\xi\in \cF_0(A^{k-1})$ and that $w_k$ is affine in the first $k-1$
variables, Hence, we conclude that
\begin{equation*}
 \|\partial_{x_k} w_k\|^2_{0,B^d} \lsim \sum_{\xi\in \cF_0(A^{k-1})}\|\partial_{x_k}v(\xi,\cdot)\|^2_{0,B^{d-k+1}},
\end{equation*}
from which the assertion of Lemma~\ref{lem:B} easily follows.
\endproof

To complete the proof of Proposition~\ref{prop:hest} consider now a generic cell $S_\bell \in \cT$, which we can write as
\begin{equation*}
S_\bell = \big(\bigtimes_{l=1}^k I_{l,\ell_l} \big)\times \big( \bigtimes_{l=k+1}^d I_{l,\ell_l} \big)=: A^k_\bell \times B_\bell^{d-k}.
\end{equation*}
Given any $v \in \cW$ and considering its restriction to $S_\bell$, we write $\hat v(\hat x)=v(x)$ with $\hat x\in B^d$, so that $ (\cI_{S_\bell}v)(x) = (\cI_{B^d}\hat v)(\hat x)$;
setting $h_{l,\ell_l}:=| I_{l,\ell_l}|$,
a standard affine change of variables yields in view of~\eqref{capH1},
\begin{equation*}\label{capH1-2}
\begin{split}
 \|v- \cI_{S_\bell}v\|^2_{0,S_\bell} &= |S_\bell | \|\hat v - \cI_{B^d}\hat v\|_{0,B^d}^2
\ \lsim \ |S_\bell | \sum_{k=1}^d \sum_{\xi\in \cF_0(A^{k-1})}\|\partial_{\hat x_k} \hat v(\xi,\cdot)\|^2_{0,B^{d-k+1}} \\
&\lsim |S_\bell | \sum_{k=1}^d \sum_{\xi\in \cF_0(A_\bell^{k-1})} \!\!\!\!\!\!\!\! |B_\bell^{d-k+1}|^{-1} h_{k,\ell_k}^2
\|\partial_{x_k}v(\xi,\cdot)\|^2_{0,B_\bell^{d-k+1}}\\
&\lsim h_\bell^2 \, \sum_{k=1}^d \sum_{\xi\in \cF_0(A_\bell^{k-1})} \!\!\!\!\! \!\!\! |A_\bell^{k-1}|
\,\|\partial_{x_k}v(\xi,\cdot)\|^2_{0,B_\bell^{d-k+1}} \;.
\end{split}
\end{equation*}
Dividing both sides by $h_\bell^2$ and summing over $\bell$ provides
\begin{equation*}
\|h^{-1}(v- \cI_{\cG}v)\|^2_{0,R}\lsim \sum_{k=1}^d \sum_{S_\bell \in \cT} \sum_{\xi\in \cF_0(A_\bell^{k-1})} |A_\bell^{k-1}|
\,\|\partial_{x_k}v(\xi,\cdot)\|^2_{0,B_\bell^{d-k+1}} \;.
\end{equation*}
Now, $|A_\bell^{k-1}| = \prod_{l=1}^{k-1} h_{l,\ell_l} \simeq \prod_{l=1}^{k-1} w_{l,\xi_l}$ for each $\xi\in \cF_0(A_\bell^{k-1})$, where the weights $w_{l,\xi_l}$ are defined by the conditions
$ \sum_{\xi_l \in \cG_l} v^2(\xi_l) w_{l,\xi_l} = \Vert \cI_{\cG_l} v \Vert_{0,I_l}^2$ for all $v \in C^0(I_l)$.
Then, the assertion follows from \eqref{stabL2interpWk}.

\subsection{Proof of Proposition~\ref{propos:forQ.3}}

Again, we treat only the first relation. The second one is analogous.
For each $R \in \cR$ and each $k=1, \dots, d$, let us recall the definitions \eqref{eq:defb.b0.b1}, \eqref{eq:def_b_RkSl_0} and \eqref{eq:def_b_RkSl_1}
of the form $b_{R,k}(u,v)$ and its portions $b_{R,k}^{(0)}(u,v)$ and $b_{R,k}^{(1)}(u,v)$. Concerning the former portion, observe that
\begin{equation} \label{bRk0}
\begin{split}
b_{R,k}^{(0)}(v,v)&= \sum_{S_\bell\in\cT_{\bp,k}^{(0)}(R)}
 \sum_{\xi' \in \cG(S_{\bell,k}')} \omega_{\bell,k}' \int_{I_{k,\ell_k}}
\big |\partial_{x_{k}} \cI^R_{h,\bp}v(\xi',x_k) \big|^2 \, dx_k \\
&\simeq \sum_{S_\bell\in\cT_{\bp,k}^{(0)}(R)} \int_{S_{\bell}}
\big |\partial_{x_{k}} \cI^R_{h,\bp}v(x) \big|^2 \, dx \leq \sum_{S_\bell\in\cT_\bp (R)} \int_{S_{\bell}}
\big |\partial_{x_{k}} \cI^R_{h,\bp}v(x) \big|^2 \, dx \leq |\cI^R_{h,\bp}v|^2_{1,R}\;,
\end{split}
\end{equation}
where we have used the uniform equivalence of the weights $\omega_{\bell,k}'$ (see \eqref{eq:def.omegaprime})
with the integration weights for multi-linear functions on the cell $S_{\bell,k}'$.
Thus, for all $\tilde{v}_R \in V_{h,D,\bp}(R)$, we have
\begin{equation}\label{eq:boundJ.b0k}
\begin{split}
b_{R,k}^{(0)}(\tilde{v}_R - Q_R \tilde{v}_R, \tilde{v}_R - Q_R \tilde{v}_R) &\lsim |\cI^R_{h,\bp}(\tilde{v}_R - Q_R \tilde{v}_R)|^2_{1,R} \\
&\lsim |\cI^R_{h,\bp}(\tilde{v}_R)|^2_{1,R} + |\cI^R_{h,\bp}(Q_R \tilde{v}_R)|^2_{1,R} \lsim | \tilde{v}_R|^2_{1,R} \;,
\end{split}
\end{equation}
where the last bound follows immediately from Lemmas~\ref{lemma:interp-unigrids} and~\ref{prop:interp-unigrids-multi},
Property~\ref{lemma:forQ.3} and Proposition~\ref{propos:forQ.1}.

Consider now the portion $b_{R,k}^{(1)}(u,v)$. Arguing as above, one has
\begin{equation*} \label{bRk1}
b_{R,k}^{(1)}(v,v)
=\! \sum_{S_\bell\in\cT_{\bp,k}^{(1)}(R)}
 \sum_{\xi' \in \cG(S_{\bell,k}')} \omega_{\bell,k}' \sum_{\xi \in \cG(I_{k,\ell_k})}
h_{k,\ell_k}^{-1} \, |\cI^R_{h,\bp}v(\xi',\xi)|^2
\simeq\! \sum_{S_\bell\in\cT_{\bp,k}^{(1)}(R)} \int_{S_\bell} h_{k,\ell_k}^{-2} \, |\cI^R_{h,\bp}v(x)|^2 \, dx \;.
\end{equation*}
Now observe that, by definition of $\cT_{\bp,k}^{(1)}(R)$ (see \eqref{aspect}), one has
$h_{k,\ell_k} \geq \max(1, C_{\textnormal{aspect}}^{-1}) \, h_\bell$, with $h_\bell=\max_l h_{l,\ell_l}$. Hence,
\begin{equation*}
\begin{split}
b_{R,k}^{(1)}(v,v) &\lsim \sum_{S_\bell\in\cT_{\bp,k}^{(1)}(R)} \int_{S_\bell} h_\bell^{-2} \, |\cI^R_{h,\bp}v(x)|^2 \, dx \leq
\sum_{S_\bell\in\cT_\bp (R)} \int_{S_\bell} h_\bell^{-2} \, |\cI^R_{h,\bp}v(x)|^2 \, dx \\[5pt]
& \ = \Vert h^{-1} \cI^R_{h,\bp}v \Vert_{0,R}^2 \lsim \Vert h^{-1}(v- \cI^R_{h,\bp}v) \Vert_{0,R}^2 + \Vert h^{-1} v \Vert_{0,R}^2 \;,
\end{split}
\end{equation*}
where $h=\sum_\bell h_\bell \chi_{S_\bell}$ is the LGL meshsize function in $R$. If $\tilde{v}_R \in V_{h,D,\bp}(R)$, this yields
\begin{equation}\label{eq:boundJ.b1k.0}
\begin{split}
b_{R,k}^{(1)}(\tilde{v}_R - Q_R \tilde{v}_R, \tilde{v}_R - Q_R \tilde{v}_R) & \lsim \Vert h^{-1}(\tilde{v}_R- \cI^R_{h,\bp}\tilde{v}_R) \Vert_{0,R}^2 +
\Vert h^{-1}(Q_R \tilde{v}_R- \cI^R_{h,\bp}(Q_R \tilde{v}_R)) \Vert_{0,R}^2 \\
&\quad + \Vert h^{-1} (\tilde{v}_R - Q_R \tilde{v}_R) \Vert_{0,R}^2 \;.
\end{split}
\end{equation}
Now we invoke Proposition~\ref{prop:hest}, with different choices of the grid $\cG$ and the space $W$, to bound each of the three summands on the right-hand side.
For the first summand, we have $\cG=\cG_\bp(R)$ (the LGL grid of order $\bp$ in $R$) and $W=V_{h,D,\bp}(R)$.
We note that, due to Lemma~\ref{lemma:interp-unigrids}, the bounds \eqref{stabL2interpWk}
are satisfied with $\bar{c}_k \lsim 1$.
Thus we get
\begin{equation}\label{eq:boundJ.b1k.01}
\Vert h^{-1}(\tilde{v}_R- \cI^R_{h,\bp}\tilde{v}_R) \Vert_{0,R}^2 \lsim |\tilde{v}_R |_{1,R}^2 \;.
\end{equation}
For the second summand, we recall that $Q_R \tilde{v}_R \in \mathbb{Q}_{\bp}(R)$, so that we have again $\cG=\cG_\bp(R)$, whereas now $W=\mathbb{Q}_{\bp}(R)$.
The bounds \eqref{stabL2interpWk} are now satisfied with $\bar{c}_k \lsim 1$ because of Lemma~\ref{lemma:forQ.3}.
Thus, recalling \eqref{eq:forQ.51bis}, we obtain
\begin{equation}\label{eq:boundJ.b1k.02}
\Vert h^{-1}(Q_R \tilde{v}_R- \cI^R_{h,\bp}(Q_R \tilde{v}_R)) \Vert_{0,R}^2 \lsim |Q_R \tilde{v}_R |_{1,R}^2 \lsim |\tilde{v}_R |_{1,R}^2 \;.
\end{equation}
At last, we bound the third summand on the right-hand side of \eqref{eq:boundJ.b1k.0}. To this end,
recalling the definition \eqref{eq:def.vstarR} of $Q_R \tilde v_R$, noting that
$\tilde v_R = \sum_{z\in \cF_0(R)} \Phi_z \tilde v_R$, and defining $\tilde v_z := \Phi_z \tilde v_R$,
it suffices to bound the quantity $C_z^2 :=\Vert h^{-1} (\tilde v_z - \cI^R_{\bp^*_z}(\cI^R_{h,D,\bp_z^*} \tilde v_z)) \Vert_{0,R}^2$ for each $z \in \cF_0(R)$.
Writing
\begin{equation*}
 \tilde v_z - \cI^R_{\bp^*_z}(\cI^R_{h,D,\bp_z^*} \tilde v_z) = (\tilde v_z - \cI^R_{h,D,\bp_z^*} \tilde v_z) +
(\cI^R_{h,D,\bp_z^*} \tilde v_z - \cI^R_{\bp^*_z}(\cI^R_{h,D,\bp_z^*} \tilde v_z)) ,
\end{equation*}
and setting for simplicity $\tilde w_z := \cI^R_{h,D,\bp_z^*} \tilde v_z \in V_{h,D,\bp^*}(R)$, we thus have
\begin{equation}\label{eq:boundJ.b1k.03}
C_z^2
\lsim \Vert h^{-1} (\tilde v_z - \cI^R_{h,D,\bp_z^*} \tilde v_z)) \Vert_{0,R}^2 +
 \Vert h^{-1} (\tilde w_z - \cI^R_{\bp^*_z} \tilde w_z) \Vert_{0,R}^2 \;.
\end{equation}
We now proceed as in the proof of Proposition~\ref{propos:forQ.1}, i.e., we work on the reference element $\hat{R}$ viewed as
an affine image of $R$. This simplifies handling the factor
$\Phi_z$ when eventually bounding the $H^1$-seminorm of $\tilde v_z=\Phi_z \tilde v_R$ by that of $\tilde v_R$ on the element $R$.

We want to apply Proposition~\ref{prop:hest} to the first summand on the right-hand side of \eqref{eq:boundJ.b1k.03}, with $\cG=\cD_{\bp^*}(\hat R)$
(the dyadic grid of order $\bp^*$ in $\hat R$) and $W=\Phi_z V_{h,D,\bp^*}(\hat R)$. To this end, we observe that the meshsize function $h=h_\bp$ associated with the grid
$\cG_\bp(\hat R)$, as a consequence of Corollary~\ref{cor:dyadic-uniequiv}, is uniformly comparable to the meshsize function $h_{D,\bp^*}$ associated with the grid $\cD_{\bp^*}(\hat R)$,
i.e., $h_\bp \simeq h_{D,\bp^*}$ pointwise in $\hat R$.
On the other hand, the bounds \eqref{stabL2interpWk} are satisfied with $\bar{c}_k \lsim 1$; this easily follows from
the fact that the restriction of $\tilde v_z$ to any cell of the grid $\cD_{\bp^*}(\hat R)$ is a piecewise multi-quadratic function belonging to a finite dimensional space whose dimension
is bounded independently of $\bp$. Thus, we obtain
\begin{equation}\label{eq:boundJ.b1k.04}
\Vert h^{-1} (\tilde v_z - \cI^{\hat R}_{h,D,\bp_z^*} \tilde v_z)) \Vert_{0,\hat R}^2 \lsim |\tilde v_z |_{1,\hat R}^2 \ \lsim \Vert \tilde v_{\hat R} \Vert _{1, \hat R}^2 \;.
\end{equation}
The second summand on the right-hand side of \eqref{eq:boundJ.b1k.03} can be bounded with the aid of
Proposition~\ref{prop:hest}
as well. Indeed, using the property that
$\cI_{h,q}(\cI_q v) = \cI_{h,q}v$ if $\cI_{h,q}v$ and $\cI_q v$ are the low- and high-order interpolants of a continuous function on the same grid, one has
\begin{equation*}
v - \cI_q v = v - \cI_{h,q}(\cI_q v) + \cI_{h,q}(\cI_q v) - \cI_q v = (v - \cI_{h,q} v) - ( \cI_q v - \cI_{h,q}(\cI_q v)) \;.
\end{equation*}
In our situation, this yields with $\tilde u_z = \cI^{\hat R}_{\bp^*_z} \tilde w_z \in \mathbb{Q}_{\bp^*}(\hat R)$
\begin{equation*}
 \Vert h^{-1} (\tilde w_z - \cI^R_{\bp^*_z} \tilde w_z) \Vert_{0,\hat R}^2 \lsim \Vert h^{-1} ( \tilde w_z - \cI^{\hat R}_{h, \bp^*_z} \tilde w_z ) \Vert_{0,\hat R}^2
 + \Vert h^{-1} (\tilde u_z - \cI^{\hat R}_{h, \bp^*_z} \tilde u_z ) \Vert_{0,\hat R}^2 \;.
\end{equation*}
so that we can apply Proposition~\ref{prop:hest} with $\cG= \cG_{\bp^*}(\hat R)$ and either $W=V_{h,D,\bp}(\hat R)$ or $W=\mathbb{Q}_{\bp^*}(\hat R)$. Again, the function
$h_\bp$ is uniformly comparable to the function $h_{\bp^*}$ associated with the grid $\cG_{\bp^*}(\hat R)$, and one easily checks that the bounds \eqref{stabL2interpWk} are
satisfied with $\bar{c}_k \lsim 1$ with both choices of $W$. Thus,
\begin{equation}\label{eq:boundJ.b1k.05}
 \Vert h^{-1} (\tilde w_z - \cI^{\hat R}_{\bp^*_z} \tilde w_z) \Vert_{0,\hat R}^2 \lsim |\tilde w_z |_{1,\hat R}^2 \lsim |\tilde v_z |_{1,\hat R}^2 \ \lsim \Vert \tilde v_{\hat R} \Vert _{1, \hat R}^2 \;,
\end{equation}
where the second inequality follows from Lemma~\ref{lemma:interp-unigrids} and Property~\ref{prop:interp-unigrids-multi}.
Going back to the element $R$, the bounds \eqref{eq:boundJ.b1k.03}, \eqref{eq:boundJ.b1k.04} and \eqref{eq:boundJ.b1k.05} imply
$
C_z^2 \lsim |\tilde{v}_R |_{1,R}^2 ,
\ \forall z \in \cF_0(R)$, which yields $ \Vert h^{-1} (\tilde{v}_R - Q_R \tilde{v}_R) \Vert_{0,R}^2 \lsim |\tilde{v}_R |_{1,R}^2$.
This, together with \eqref{eq:boundJ.b1k.01} and \eqref{eq:boundJ.b1k.02}, allows us to obtain
$
b_{R,k}^{(1)}(\tilde{v}_R - Q_R \tilde{v}_R, \tilde{v}_R - Q_R \tilde{v}_R) \lsim |\tilde{v}_R |_{1,R}^2
$
from \eqref{eq:boundJ.b1k.0}. Thus, the proof of Proposition~\ref{propos:forQ.3} is complete.
\endproof

\begin{remark}\label{rem:Qapprox}
Consider the operator $Q: \tilde V= W_{h,D,\bp} \to V= V_{h,D,\bp}$ and the auxiliary bilinear form $b_3$ introduced in Section~\ref{sect:aux-multilevel}, where condition \textbf{ASM 3} is claimed to hold in \eqref{eq:Q3approx} of Lemma~\ref{lem:Q}. Since both objects are defined in complete analogy to \eqref{def:operQ}, the proof of this property is
similar to that of Proposition~\ref{propos:forQ.3} given above.
\end{remark}


\bibliographystyle{plain}
\bibliography{biblio}

\begin{thebibliography}{10}

\bibitem{AA1}
P.~Antonietti and B.~Ayuso~de Dios.
\newblock {S}chwarz domain decomposition preconditioners for discontinuous
  {G}alerkin approximations of elliptic problems: non-overlapping case.
\newblock {\em M2AN Math. Model. Numer. Anal.}, 41:21--54, 2007.

\bibitem{AA2}
P.~Antonietti and B.~Ayuso~de Dios.
\newblock Multiplicative {S}chwarz methods for discontinuous {G}alerkin
  approximations of elliptic problems.
\newblock {\em M2AN Math. Model. Numer. Anal.}, 42:443--469, 2008.

\bibitem{Ar82}
D.~N. Arnold.
\newblock An interior penalty finite element method with discontinuous
  elements.
\newblock {\em SIAM J. Numer. Anal.}, 19:742--760, 1982.

\bibitem{ArBrCoMa02}
D.~N. Arnold, F.~Brezzi, B.~Cockburn, and L.~D. Marini.
\newblock Unified analysis of discontinuous {G}alerkin methods for elliptic
  problems.
\newblock {\em SIAM J. Numer. Anal.}, 39:1749--1779, 2002.

\bibitem{BeMa97}
Ch. Bernardi and Y.~Maday.
\newblock Spectral methods.
\newblock In Ph.G. Ciarlet and J.~L. Lions, editors, {\em Handbook of Numerical
  Analysis, Vol. V, part 2}, pages 209--486. Elsevier, Amsterdam, 1997.

\bibitem{BeuSchneiSchwab}
S.~Beuchler, R.~Schneider, and C.~Schwab.
\newblock Multiresolution weighted norm equivalencies and applications.
\newblock {\em Numer. Math.}, 98:67--97, 2004.

\bibitem{Brenner1996a}
S.~C. Brenner.
\newblock Two-level additive {S}chwarz preconditioners for nonconforming finite
  element methods.
\newblock {\em Math. Comput.}, 65:897--921, 1996.

\bibitem{brix-thesis}
K.~Brix.
\newblock {\em Robust preconditioners for $hp$-discontinuous {G}alerkin
  discretizations for elliptic problems}.
\newblock PhD thesis, Institut f\"ur Geometrie und Praktische Mathematik, RWTH
  Aachen, 2014.
\newblock In preparation.

\bibitem{BCD}
K.~Brix, M.~Campos~Pinto, and W.~Dahmen.
\newblock A multilevel preconditioner for the interior penalty discontinuous
  {G}alerkin method.
\newblock {\em SIAM J. Numer. Anal.}, 46:2742--2768, 2008.

\bibitem{BCDM}
K.~Brix, M.~Campos~Pinto, W.~Dahmen, and R.~Massjung.
\newblock Multilevel preconditioners for the interior penalty discontinuous
  {G}alerkin method {II} - {Q}uantitative studies.
\newblock {\em Commun. Comput. Phys.}, 5:296--325, 2009.

\bibitem{BCD2012}
K.~Brix, C.~Canuto, and W.~Dahmen.
\newblock Robust preconditioners for {DG}-discretizations with arbitrary
  polynomial degrees.
\newblock In {\em Proceedings of the 21st International Conference on Domain
  Decomposition Methods, Rennes, France, June 25th--29th, 2012}, Berlin, 2012.
  Springer.
\newblock Submitted. {\tt arXiv:1212.6385 [math.NA].} URL:
  \url{http://dd21.inria.fr/DD21_editor-svmult-ddm/editor.pdf}.

\bibitem{BCD2013}
K.~Brix, C.~Canuto, and W.~Dahmen.
\newblock {L}egendre-{G}auss-{L}obatto meshes and associated nested dyadic
  meshes.
\newblock {\em IGPM preprint 378, RWTH Aachen}, 2013.
\newblock Submitted. {\tt arXiv:1311.0028 [math.NA]}.

\bibitem{Ca94}
C.~Canuto.
\newblock Stabilization of spectral methods by finite element bubble functions.
\newblock {\em Comput. Methods Appl. Mech. Eng.}, 116:13--26, 1994.

\bibitem{CaHuQuZa06}
C.~Canuto, M.~Y. Hussaini, A.~Quarteroni, and T.~A. Zang.
\newblock {\em Spectral Methods. Fundamentals in Single Domains}.
\newblock Springer Verlag, Berlin, 2006.

\bibitem{CaHuQuZa07}
C.~Canuto, M.~Y. Hussaini, A.~Quarteroni, and T.~A. Zang.
\newblock {\em Spectral Methods. Evolution to Complex Geometries and
  Applications to Fluid Dynamics}.
\newblock Springer Verlag, Berlin, 2007.

\bibitem{CPP}
C.~Canuto, L.~F. Pavarino, and A.~B. Pieri.
\newblock {BDDC} preconditioners for continuous and discontinuous {G}alerkin
  methods using spectral/hp elements with variable polynomial degree.
\newblock 8 2013.
\newblock To appear in IMA J. Numer. Anal.

\bibitem{CTU1999}
C.~Canuto, A.~Tabacco, and K.~Urban.
\newblock The wavelet element method. 1. {C}onstruction and analysis.
\newblock {\em Appl. Comput. Harmon. Anal.}, 6:1--52, 1999.

\bibitem{DS1999}
W.~Dahmen and R.~Schneider.
\newblock Wavelets on manifolds {I}: Construction and domain decomposition.
\newblock {\em SIAM J. Math. Anal.}, 31:184--230, 1999.

\bibitem{DM1985}
M.~Deville and E.. Mund.
\newblock Chebyshev pseudospectral solution of second-order elliptic equations
  with finite element preconditioning.
\newblock {\em J. Comput. Phys.}, 60:517--533, 1985.

\bibitem{DLVZ2006}
V.~A. Dobrev, R.~D. Lazarov, P.~S. Vassilevski, and L.~T. Zikatanov.
\newblock Two-level preconditioning of discontinuous {G}alerkin approximations
  of second-order elliptic equations.
\newblock {\em Numer. Linear Algebra Appl.}, 13:753--770, 2006.

\bibitem{DGH1996}
G.~C. Donovan, J.~S. Geronimo, and D.~P. Hardin.
\newblock Intertwining multiresolution analyses and the construction of
  piecewise-polynomial wavelets.
\newblock {\em SIAM J. Math. Anal.}, 27:1791--1815, 1996.

\bibitem{DGH1999}
G.~C. Donovan, J.~S. Geronimo, and D.~P. Hardin.
\newblock Orthogonal polynomials and the construction of piecewise polynomial
  smooth wavelets.
\newblock {\em SIAM J. Math. Anal.}, 30:1029--1056, 1999.

\bibitem{kanschat}
J.~Gopalakrishnan and G.~Kanschat.
\newblock A multilevel discontinuous {G}alerkin method.
\newblock {\em Numer. Math.}, 95:527--550, 2003.

\bibitem{GO1995}
M.~Griebel and P.~Oswald.
\newblock On the abstract theory of additive and multiplicative {S}chwarz
  algorithms.
\newblock {\em Numer. Math.}, 70(2):163--180, 1995.

\bibitem{GO1995a}
M.~Griebel and P.~Oswald.
\newblock Tensor product type subspace splittings and multilevel iterative
  methods for anisotropic problems.
\newblock {\em Adv. Comput. Math.}, 4(1--2):171--206, 1995.

\bibitem{HeWa08}
J.S. Hesthaven and T.~Warburton.
\newblock {\em Nodal Discontinuous Galerkin Methods}.
\newblock Springer Verlag, New York, 2008.

\bibitem{Kanschat2004}
G.~Kanschat.
\newblock Multilevel methods for discontinuous {G}alerkin {FEM} on locally
  refined meshes.
\newblock {\em Comput. \& Structures}, 82:2437--2445, 2004.

\bibitem{Nepomnyaschikh1990}
S.~V. Nepomnyaschikh.
\newblock Fictitious components and subdomain alternating methods.
\newblock {\em Soviet J. Numer. Anal. Math. Modelling}, 5:53--68, 1990.

\bibitem{Nepomnyaschikh1992}
S.~V. Nepomnyaschikh.
\newblock Decomposition and fictitious domains methods for elliptic boundary
  value problems.
\newblock In D.~E. Keyes, T.~F. Chan, G.~A. Meurant, J.~S. Scroggs, and R.~G.
  Voigt, editors, {\em Fifth International Symposium on Domain Decomposition
  Methods for Partial Differential Equations}, pages 62--72. Society for
  Industrial and Applied Mathematics, Philadelphia, PA, 1992.

\bibitem{Oswald1996}
P.~Oswald.
\newblock Preconditioners for nonconforming discretizations.
\newblock {\em Math. Comp.}, 65:923--941, 1996.

\bibitem{PaRo95}
S.~V. Parter and E.~E. Rothman.
\newblock Preconditioning {L}egendre spectral collocation approximations to
  elliptic problems.
\newblock {\em SIAM J. Numer. Anal.}, 32:333--385, 1995.

\bibitem{SSH}
E.~S\"uli, C.~Schwab, and P.~Houston.
\newblock $hp$-{DGFEM} for partial differential equations with nonnegative
  characteristic form.
\newblock In B.~Cockburn, G.~E. Karniadakis, and C.-W. Shu, editors, {\em
  Discontinuous {G}alerkin Methods: Theory, Computation and Applications},
  volume~11, pages 211--230. Springer Verlag, Berlin, 2000.

\bibitem{TW2005}
A.~Toselli and O.~Widlund.
\newblock {\em Domain Decomposition Methods - Algorithms and Theory},
  volume~34.
\newblock Springer Verlag, Berlin, 2005.

\bibitem{WFS}
T.~P. Wihler, P.~Frauenfelder, and C.~Schwab.
\newblock Exponential convergence of the $hp$-{DGFEM} for diffusion problems.
\newblock {\em Comput. Math. Appl.}, 46:183--205, 2003.

\bibitem{Xu1992}
J.~Xu.
\newblock Iterative methods by space decomposition and subspace correction.
\newblock {\em SIAM Rev.}, 34:581--613, 1992.

\bibitem{Xu1996}
J.~Xu.
\newblock The auxiliary space method and optimal multigrid preconditioning
  techniques for unstructured grids.
\newblock {\em Computing}, 56:215--235, 1996.

\end{thebibliography}


\end{document}